\documentclass[12pt]{article}

\usepackage{enumerate}
\usepackage[OT1]{fontenc}
\usepackage{amsmath,amssymb}
\usepackage{natbib}
\usepackage{bbm}
\usepackage[usenames]{color}

\usepackage{smile}
\usepackage{natbib}
\usepackage{multirow}
\usepackage{mathtools}
\usepackage[colorlinks,
            linkcolor=red,
            anchorcolor=blue,
            citecolor=blue
            ]{hyperref}

 \newcommand{\limn}{\lim_{n\rightarrow \infty} }
\def\given{\,|\,}

\def\tr{\mathop{\text{tr}}\kern.2ex}

\def\sign{\mathop{\text{sign}}}

\def\supp{\mathop{\text{supp}}}

\long\def\comment#1{}

\def\cS{{\mathcal{S}}}

\newcommand{\bel}{\begin{eqnarray}\label}
\newcommand{\eel}{\end{eqnarray}}
\newcommand{\bes}{\begin{eqnarray*}}
\newcommand{\ees}{\end{eqnarray*}}

\newcommand{\la}{\langle}
\newcommand{\ra}{\rangle}

\let\emptyset\varnothing

\usepackage{mathrsfs}

\usepackage{fullpage}

\def\sep{\;\big|\;}

\usepackage[protrusion=false,expansion=true]{microtype}
\usepackage[british,UKenglish,USenglish,english,american]{babel}

\def\##1\#{\begin{align}#1\end{align}}
\def\$#1\${\begin{align*}#1\end{align*}}
\usepackage{enumitem}

\def \mA {\mathscr{A} }
\newcommand\floor[1]{\lfloor#1\rfloor}

\begin{document}

\title{\LARGE Curse of  Heterogeneity: Computational Barriers in Sparse Mixture Models and Phase Retrieval}
\author{Jianqing Fan\thanks{Princeton University; e-mail: \texttt{jqfan@princeton.edu}, supported by NSF grants DMS-1712591 and DMS-1662139 and NIH grant 2R01-GM072611.} \qquad Han Liu\thanks{Northwestern University; e-mail: \texttt{hanliu.cmu@gmail.com}.}\qquad Zhaoran Wang\thanks{Northwestern University; e-mail: \texttt{zhaoranwang@gmail.com}.} \qquad Zhuoran Yang\thanks{Princeton University; e-mail: \texttt{zy6@princeton.edu}.}
}
\maketitle				
\date{}

\begin{abstract}
We study the fundamental tradeoffs between statistical accuracy and computational tractability in the analysis  of high dimensional heterogeneous data. As examples, we study sparse Gaussian mixture model, mixture of sparse linear regressions, and sparse phase retrieval model.    For these models, we exploit an oracle-based computational model to establish conjecture-free computationally feasible minimax lower bounds, which quantify the minimum signal strength required for the existence of any algorithm that is both computationally tractable and statistically accurate. Our analysis shows that there exist significant gaps between computationally feasible minimax risks and classical ones. These gaps quantify the statistical price we must pay to achieve computational tractability in the presence of data heterogeneity.     Our results cover the problems of detection, estimation, support recovery, and clustering, and moreover, resolve several conjectures of \cite{azizyan2013minimax, azizyan2015efficient, arias2017detection,cai2015optimal}.  Interestingly, our results reveal a new but counter-intuitive phenomenon in heterogeneous data analysis that more data might lead to less computation complexity.

\end{abstract}
\section{Introduction}\label{sec::intro}

Computational efficiency and statistical accuracy are two key factors for designing learning algorithms. Nevertheless, classical~statistical theory focuses more on characterizing the minimax risk of a learning procedure rather than its computational efficiency. In high dimensional heterogeneous data analysis, it is usually observed that statistically optimal procedures are not computationally tractable, while computationally efficient methods are suboptimal in terms of statistical risk \citep{azizyan2013minimax, azizyan2015efficient, arias2017detection,cai2015optimal}.  This discrepancy motivates~us to study the fundamental statistical limits of learning high dimensional heterogeneous models under computational tractability constraints. As examples, we consider two heterogeneous models, namely sparse Gaussian mixture model and mixture of sparse linear regressions.  These two models are prominently featured in the analysis of big data \citep{fan2014challenges}.



Gaussian mixture model  is one of the most fundamental statistical models. It has broad applications in a variety of areas, including speech and image processing \citep{reynolds1995robust, zhuang1996gaussian}, social science \citep{titterington1985statistical}, as well as biology \citep{yeung2001model}. Specifically, for observable $\bX \in \RR^{d}$ and a discrete latent variable $ Z \in \cZ$, Gaussian mixture model assumes
\$
\bX| Z = z \sim N(\bmu_{z}, \bSigma_z),\quad \text{where~~} \PP(Z = z) = p_z,~~\text{and~~}\textstyle{\sum_{z\in \cZ}} p_z = 1,
\$
where $\bmu_{z}$ and $\bSigma_z$ denote the mean and covariance matrix of $\bX$ conditioning on $Z = z$. Let $n$ be the number of observations.  In this paper, we study the high dimensional setting where $d \gg n$, which  is challenging for consistently recovering $\bmu_{z}\ (z\in\cZ)$, even assuming $|\cZ| = 2$ and $\bSigma_z$'s are known. To address such an issue, one popular assumption is that the difference between the two~$\bmu_{z}$'s is sparse \citep{azizyan2013minimax, arias2017detection}. In detail, for $\cZ = \{1,2\}$, they~assume~that $\Delta \bmu = \bmu_2 - \bmu_1$ is $s$-sparse, i.e., $\Delta \bmu$ has $s$ nonzero entries ($s \ll n$). Under this sparsity assumption,  \cite{azizyan2013minimax, arias2017detection} establish information-theoretic lower bounds and efficient algorithms for detection, estimation, support recovery, and clustering. However, there remain rate gaps between the information-theoretic lower bounds and the upper bounds that are attained by efficient algorithms. Is the gap intrinsic to the difficulty of the mixture problem?  We will show that such a lower bound is indeed sharp if no computational constraints are imposed, and such an upper bound is also sharp if we restrict our estimators to computationally feasible ones.

Another example of heterogeneity is the mixture of linear regression model, which characterizes the regression problem  where the observations consist of multiple subgroups with different regression parameters. Specifically, we assume that $Y = \bmu_{z}^\top \bX + \epsilon$ conditioning on the discrete latent variable $Z = z$, where $\bmu_z$'s are the regression parameters, $ \bX\in \RR^d$,  and $\epsilon \sim N(0, \sigma^2)$ is the random noise, which is independent of everything else.
Here we also focus on the high dimensional setting in which $ d\gg n $, where $n$ is sample size. 
In this setting, consistently estimating mixture of regressions is challenging even when $| \cZ| = 2$.~Similar to Gaussian mixture model, we focus on the setting in which $\cZ = \{ 1, 2\}$, $p_1 = p_2 = 1/2$, and $\bmu_1 = - \bmu_2 = \bbeta$ is $s$-sparse to illustrate the difficulty of the problem. As we will illustrate in \S\ref{sec::theory_reg}, this symmetric setting is closely related to sparse phase retrieval \citep{chen2014convex}, for which \cite{cai2015optimal} observe a gap in terms of sample complexity between the information-theoretic limit and upper bounds that are attained by computationally tractable algorithms.


One question is left open: Are such gaps intrinsic to these statistical models with heterogeneity, which can not be eliminated by more complicated algorithms or proofs? In other words, do we have to sacrifice statistical accuracy to achieve computational tractability?

In this paper, we provide an affirmative answer to this question. In detail, we study the detection problem, i.e., testing whether $\Delta \bmu = \zero$ or $\bbeta = \zero$ in the above models, since the fundamental limit of detection further implies the limits of estimation, support recovery, as well as clustering. We~establish sharp computational-statistical phase transitions in terms of the sparsity level $s$, dimension $d$, sample size $n$, as well as the signal strength, which~is determined by the model parameters. More specifically, under the simplest setting of Gaussian mixture model where the covariance matrices are identity, up to a term that is logarithmic in  $n$,~the computational-statistical phase transitions are as follows~under certain regularity conditions.
\begin{enumerate}[label=(\roman*)]
\item  In the weak-signal regime where $\|\Delta \bmu\|_2^2 = o (\sqrt { s\log d /n  } ) $, any algorithm fails to detect~the~sparse Gaussian mixtures.
\item In the regime with moderate signal strength 
\$
\|\Delta \bmu\|_2^2 = \Omega (\sqrt { s\log d /n  } )  \quad\text{and}\quad \|\Delta \bmu\|_2^2 = o (\sqrt{s^2/n} ) ,
\$
under a generalization of the statistical query model \citep{kearns1998efficient}, any efficient algorithm~that has polynomial computational complexity fails to detect the sparse Gaussian mixture. (We will specify the computational model and the notion of oracle complexity in details in \S\ref{sec::bg}.)
Meanwhile,~there exists an algorithm with superpolynomial oracle complexity that successfully detects the sparse Gaussian mixtures.
\item  In the strong-signal regime where $\|\Delta \bmu\|_2^2 = \Omega (\sqrt{s^2/n} )$, there exists an efficient algorithm~with polynomial oracle complexity that succeeds.
\end{enumerate}
Here regime (ii) exhibits the tradeoffs between statistical optimality and computational tractability. More specifically, $\sqrt{s^2/n}$ is the minimum detectable signal strength under computational tractability constraints, which contrasts with the classical minimax lower bound $\sqrt { s\log d /n  }  $. In other words, to attain computational tractability, we must pay a price of $\sqrt { s\log d /n  } $ in the minimum detectable signal strength. We will also establish the results for more general covariance matrices  in \S\ref{sec::unknown}, where $\bSigma_{z} = \bSigma \ (z \in \{1,2\})$ may even be unknown. In addition, for mixture of regressions, we establish similar phase transitions as in (i)-(iii) with $\| \Delta \bmu\|_2^2 $ replaced by $\| \bbeta \|_2^2 / \sigma^2$, where $\sigma$ is the standard deviation of the noise.~See \S\ref{sec::theory_GMM} and \S\ref{sec::theory_reg} for details.

From another point of view, the above statistical-computational~tradeoffs reveal a new and counter-intuitive phenomenon, i.e., with a larger sample size $n = \Omega (s^2 /\|\Delta\bmu\|_2^4 )$, which corresponds to the strong-signal regime, we can achieve much lower computational complexity (polynomial oracle complexity). In contrast,  with a smaller sample size $n = o(s^2/\|\Delta\bmu\|_2^4  )$, which corresponds to the moderate-signal regime, we suffer from  superpolynomial oracle complexity. In other words, with more data, we need less computation.  Such a novel and counter-intuitive phenomenon is first captured in the literature.
On the other hand, this new phenomena is not totally unexpected.  With  a larger $n$, the mixture problem becomes locally more convex around the true parameters of interest, which helps the optimization.

Our results are of the same nature as a recent line of work on statistical-computational~tradeoffs    \citep{berthet2013computational, berthet2013optimal,  ma2013computational,
	daniely2013more,gao2017sparse,wang2014statistical, zhang2014lower, chen2014statistical, krauthgamer2013semidefinite,  cai2015computational, chen2015incoherence, hajek2014computational, perry2016optimality,lelarge2016fundamental, brennan2018reducibility, zhang2018tensor,wu2018statistical}. Such a line of work is mostly based upon randomized polynomial-time reductions~from average-case computational hardness conjectures, such as \textsf{planted clique} conjecture \citep{alon1998finding} and \textsf{random 3SAT} conjecture \citep{feige2002relations}. In detail, they build a reduction from a problem that is conjectured to be computationally difficult  to an instance of the statistical problem of interest, which implies the computational difficulty of the statistical problem. Such a reduction-based  approach has several drawbacks. Firstly, there lacks a consensus on the correctness of average-case computational hardness conjectures \citep{applebaum2008basing, barak2012truth}. Secondly, there lacks a systematic way to connect a statistical problem with a proper computational hardness conjecture.

In this paper, we employ a different approach. Instead of reducing a problem that is conjectured to be computationally hard to solve to  the statistical problem of interest, we directly characterize the computationally feasible minimax lower bounds using the intrinsic structure of the sparse mixture model. In detail, we~focus on an oracle-based computational model, which generalizes the statistical query model proposed by \cite{kearns1998efficient} and recently generalized by \cite{feldman2013statistical, feldman2015statistical, feldman2015complexity, wang2015sharp}. In particular, compared with their work, we focus on a more powerful computational model that allows continuous-valued query functions, which is more natural to sparse mixture models.~Under~such a computational~model, we establish  sharp computationally feasible minimax lower bounds for sparse mixture models under the regimes with moderate signal strength. Such lower bounds do not depend on any unproven conjecture, and are applicable for almost all commonly used~learning algorithms, such as convex optimization algorithms, matrix decomposition algorithms, expectation-maximization~algorithms, and sampling algorithms \citep{blum2005practical, chu06mapreduce}.

There exists a vast body of literature on learning mixture models.  The study of Gaussian mixture model dates back to \cite{pearson1894contributions, lindsay1993multivariate,fukunaga1983estimation}. To attain the sample complexity that is polynomial in $d$ and $|\cZ|$,  \cite{dasgupta1999learning, dasgupta2000two,sanjeev2001learning,vempala2004spectral,brubaker2008isotropic} develop a variety of efficient algorithms for learning Gaussian mixture model with well-separated means. In the general settings with nonseparated means, \cite{belkin2009learning, belkin2010polynomial,kalai2010efficiently,moitra2010settling,hsu2013learning,bhaskara2014smoothed,anderson2014more, anandkumar2014tensor, ge2015learning, cai2016rate} construct efficient algorithms based on the method of moments. A more related piece of work is \cite{srebro2006investigation}, which focuses on the information-theoretic and computational limits of clustering spherical Gaussian mixtures. In contrast with this line of work, we focus on the sparse Gaussian~mixture~model in high dimensions, for which we establish the existence~of~fundamental~gaps between computational tractability and information-theoretic optimality.

For sparse Gaussian mixture model,  \cite{raftery2006variable,maugis2009variable,pan2007penalized,maugis2008slope,stadler2010ell,maugis2011non,krishnamurthy2011high,ruan2011regularized,he2011laplacian,lee2012variable,lotsi2013high,malsiner2013model,azizyan2013minimax,gaiffas2014sparse} study the problems of clustering and feature selection, but mostly either lack efficient algorithms to attain the proposed estimators, or do not have finite-sample guarantees. \cite{arias2017detection, azizyan2015efficient} establish efficient algorithms for detection, feature selection, and clustering with finite-sample guarantees.~They observe the gaps between the information-theoretic lower bounds and the minimum signal strengths required by the computationally tractable learning algorithms proposed therein. It remains unclear whether these gaps  are  intrinsic to the statistical problems. In this paper we close this open question by proving that these gaps can not be eliminated, which gives rise to the fundamental tradeoffs between statistical accuracy and computational efficiency.

In addition, mixture of regression model is first introduced by \cite{quandt1978estimating}, where estimators based upon the moment-generating function are proposed.  In subsequent work,  \cite{de1989mixtures,wedel1995mixture,mclachlan2004finite,zhu2004hypothesis,faria2010fitting} study the  likelihood-based estimators along with expectation-maximization (EM) or gradient descent algorithms, which  are vulnerable to local optima. In addition, \cite{khalili2007variable} propose a penalized likelihood method for variable selection under the low dimensional setting, which lacks finite-sample guarantees. To attain computational efficient estimators~with finite-sample error bounds, \cite{chaganty2013spectral,yi2013alternating,chen2014convex,balakrishnan2014statistical} tackle the problem of parameter estimation using spectral methods, alternating minimization, convex optimization, and EM algorithm. For high dimensional mixture of regressions, \cite{stadler2010ell} propose $\ell_1$-regularization for parameter estimation. In more recent work, \cite{wang2014high,yi2015regularized} propose estimators based upon high dimensional variants of EM algorithm. Although gaining computational efficiency, the upper bounds in terms of sample complexity obtained in \cite{wang2014high,yi2015regularized} are statistically suboptimal. It is natural to ask whether there exists a computationally tractable estimator that attains statistical optimality.~Similar to Gaussian mixture model, we resolve this question by showing that the gap between computational tractability and information-theoretic optimality is also intrinsic to sparse mixture of regressions.


It is worth noting that \cite{jin2015phase, jin2014rare} study the phase transition in mixture detection in the context of multiple testing. They study the statistical and computational tradeoffs for specific methods in the upper bounds. In comparison, our tradeoffs hold for all algorithms under a generalization of the statistical query model. Also, our setting is different from theirs,~which~leads to incomparable statistical rates of convergence.
Besides, \cite{diakonikolas2017statistical} establish a  statistical query lower bound  for learning   Gaussian mixture model with multiple components, which shows how computational complexity scales with the dimension and the number of components. In contrast,  we exhibit the statistical-computational phase transition in sparse mixture models with two components.
 In addition, \cite{wang2015sharp} consider the problems of structural normal mean detection and sparse principal component detection. The former problem exhibits drastically different computational-statistical phase transitions compared with the problems considered in this paper. Meanwhile, sparse principal component detection~is~closely~related~to~sparse Gaussian mixture detection, which will be discussed in \S\ref{sec::connet_sparse_PCA}. In particular, we will show that sparse principal component detection is  more difficult in comparison with sparse Gaussian mixture detection. Hence, the computational lower bounds for sparse Gaussian mixture detection are  more challenging to establish, and imply the lower bounds for sparse principal component detection. To address this challenge, we employ a sharp characterization of the $\chi^2$-divergence between the null and alternative hypotheses under a localized prior, which is tailored towards sparse Gaussian mixture model. See \S\ref{pf::lower} for details.

Our analysis of the statistical-computational tradeoffs is based on a sequence of work on statistical query models by \cite{kearns1998efficient, blum1994weakly, blum1998polynomial, servedio1999computational, yang2001learning, yang2005new, jackson2003efficiency, szorenyi2009characterizing, feldman2012complete, feldman2012computational, feldman2013statistical, feldman2015statistical, feldman2015complexity,diakonikolas2017statistical,wang2015sharp,  yi2016more, lu2018edge}. Our computational model    is based upon the {\sf VSTAT} oracle model proposed  in \cite{feldman2013statistical, feldman2015complexity, feldman2015statistical}, which is a powerful tool for understanding the computational hardness  of statistical problems. This model  is
 used to study  problems such as planted clique \citep{feldman2013statistical}, random $k$-{\sf SAT} \citep{feldman2015complexity},   stochastic convex optimization \citep{feldman2015statistical}, low dimensional Gaussian mixture model \citep{diakonikolas2017statistical}, detection of structured normal mean and sparse principal components \citep{wang2015sharp}, weakly supervised learning \citep{yi2016more}, and combinatorial inference \citep{lu2018edge}. Following this line of work, we study the computational aspects of high dimensional mixture models under the oracle model framework.


In summary, our contribution is two-fold.
\begin{enumerate}[label=(\roman*)]

\item We establish the first conjecture-free computationally feasible minimax lower bound for sparse Gaussian mixture model and sparse mixture of regression model. Our theory sharply characterizes~the computational-statistical phase transitions in these models, and moreover resolves the questions left open by \cite{arias2017detection, azizyan2013minimax, azizyan2015efficient,cai2015optimal}. Such phase transitions reveal a counter-intuitive ``more data, less computation''~phenomenon in heterogeneous data analysis, which is observed for the first time.

\item Our analysis is built on a slight modification of the statistical query model \citep{kearns1998efficient, feldman2013statistical, feldman2015statistical, feldman2015complexity, wang2015sharp} that   captures the algorithms~for sparse mixture models in the  real world. The analytic techniques used to establish  the  computationally feasible minimax lower bounds under this  computational model are of independent interest.

\end{enumerate}

\section{Background}\label{sec::bg}

In the following, we first define the computational model. Then we introduce the detection problem in sparse Gaussian mixture model and mixture of linear regression model.

\subsection{Computational Model}
To solve statistical problems, algorithms must interact with data. Therefore, the number of rounds of interactions with data serves as a good proxy for the algorithmic complexity of learning algorithms.~In the following, we define a slight modification of the statistical query model \citep{kearns1998efficient, feldman2013statistical, feldman2015statistical, feldman2015complexity, wang2015sharp} to quantify the interactions between algorithms and data. 

\begin{definition}[\em Statistical query model]\label{def::oracle}
 An algorithm $\mA$ is allowed to query an oracle $r$ up to $T$ rounds,   each round gets  an estimate of the expected value of a univariate query function. Let  $M$ be a fixed number. We define $\cQ_{\mA} \subseteq \{q:\cX\rightarrow [ -M, M]\}$ as~the query space of $\mA$, that is, the set of all query functions that algorithm $\mA$ can use to interact with any oracle. Here we consider query functions that take bounded values. At each round, $\mA$ queries the oracle $r$ with a   function $q \in \cQ_{\mA}$, and obtain a realization of $Z_q \in \RR$, where $Z_q$ satisfies
	\#\label{eq::query_00}
	\PP \biggl({\bigcap_{q \in \cQ_{\mA}} } \Bigl\{ \bigl| Z_q  - \EE\bigl[q(\bX)\bigr] \bigr| \leq \tau_q \Bigr\} \biggr) \geq 1- 2\xi.
	\#
	Here $\xi\in [0,1)$ is the tail probability, $\tau_q >0$ is the  tolerance parameter, which is given by
	\#\label{eq::query_2}
	\tau_q = \max  \Biggl\{ \frac{ \bigl [ \eta(\cQ_{\mA}) + \log (1/\xi) \bigr ] \cdot M}{n},  \sqrt{\frac{2 \bigl[\eta(\cQ_{\mA}) + \log (1/\xi)\bigr] \cdot \bigl \{ M^2 - \EE ^2 [  q(\bX)  ]  \bigr \}  }{n} } \Biggr\} .
	\#
Moreover,  $\eta(\cQ_{\mA}) \geq 0$ in \eqref{eq::query_2} measures~the capacity of $\cQ_{\mA}$ in logarithmic scale, e.g., for finite $\cQ_{\mA}$, $\eta(\cQ_{\mA}) = \log (|\cQ_{\mA}|)$. We define $T$ as the oracle complexity, and $\cR[\xi,n,T, M,  \eta(\cQ_{\mA}) ]$ as the set of valid oracles satisfying the above definition. See Figure \ref{fig:illu} for an illustration.
\end{definition}

\begin{figure*}[tpb]
\centering
\includegraphics[width=0.8\textwidth]{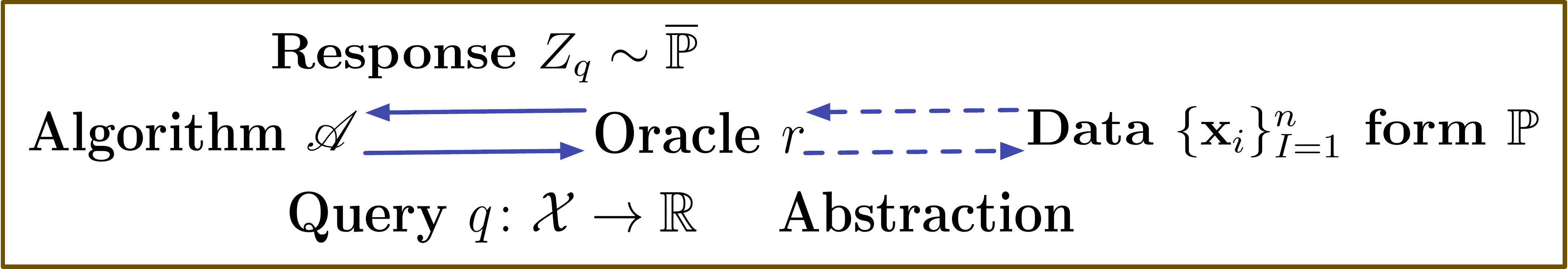}
\caption{An illustration of Definition \ref{def::oracle}.  In each iteration of an algorithm $\mA$, it sends a query $q$ to the oracle $r$ and obtain a realization of $Z_q$, which is close to $\EE [ q(\bX)] $ in the sense of \eqref{eq::query_00}.   Interaction with an   oracle  can be viewed  as an abstraction of the setting where we have direct access to   data $\{ \xb_i\}_{i =1}^n$ drawn from a distribution $\PP$. We note that the distribution of $\{ Z_q, q \in \cQ_{\mA} \}$ depends on the choice of $r$, which is denoted by $\overline{\PP}$. For a fixed query function $q$, $Z_q$ will have different distributions under different statistical oracles. }
\label{fig:illu}
\end{figure*}

 The intuition behind Definition \ref{def::oracle} can be understood from the following two aspects.
\begin{enumerate}[label=(\roman*)]

\item Suppose $n\rightarrow\infty$ such that $[\eta(\cQ_{\mA}) + \log (1/\xi)]/n\rightarrow 0$. Then, $\mA$ directly queries the population distribution of $\bX$ using $q$ and obtains a  consistent estimate of $\EE[q(\bX)]$.  For a given algorithm $\mA$, its associated query space $\cQ_{\mA}$ consists of all queries functions whose answers are uniformly consistent in the sense of \eqref{eq::query_00}.  The algorithm chooses the query functions $\{q_t\}_{t=1}^T \subseteq \cQ_{\mA}$ and gets oracle answers $\{Z_{q_t}\}_{t=1}^T$.  Statistical decisions are now based on these oracle answers.  We count the computation complexity as $T$, the number of
 estimated means 
  we are allowed to ask the oracle.


\item  In realistic cases, we have a realization of the data $\{\xb_i\}_{i=1}^n $ from the population. 
In this setting, it is common to use sample average to approximate $\EE [q(\bX)]$, which~incurs a statistical error that is governed by~Bernstein's inequality for bounded variables
\#\label{eq::bernstein-type-query}
 \PP \biggl\{ \biggl| \frac{1}{n}{\sum_{i=1}^n} q(\xb_i) - \EE \bigl[q(\bX)\bigr]  \biggr| >  t   \biggr\} \leq 2 \exp\biggl \{  \frac{ - n \cdot t^2 }{ 2 \cdot \text{Var} \bigl [ q(X) \bigr ] + 2 M/ 3  \cdot t } \biggr \}.
\#
In addition, since $q(X) \in [-M, M]$, its variance can be bounded  by
\#\label{eq:var_q_bound}
\text{Var}\bigl[ q(X) \bigr ]  = \EE \bigl  [ q^2(X) \bigr ]  - \EE ^2 \bigl [ q(X) \bigr ] \leq M^2 - \EE ^2 \bigl [ q(X) \bigr ].
\#
Thus,   in \eqref{eq::query_2} we replace the unknown  $\text{Var}[ q(X)  ] $ by its upper bound in \eqref{eq:var_q_bound}, which is tight when $q(X) $ only takes values in $\{-M, M\}$.
Moreover, uniform concentration over $\cQ_{\mA}$ can be obtained~by bounding the suprema of empirical processes. For example, when $\cQ_{\mA}$ is countable, by taking a union bound over $q \in \cQ_{\mA}$ in \eqref{eq::bernstein-type-query}, we obtain
\#\label{eq::bernstein-type-query2}
\PP\Biggl ( { \sup_{q \in \cQ_{\mA}}}  \Biggl \{ \biggl | \frac{1}{n} {\sum_{i=1}^n} q(\xb_i)  - \EE \bigl[q(\bX)\bigr] \biggr |   \leq   c \tau_q  \Biggr \}  \Biggr ) \geq 1- 2\xi
\#
for an absolute constant $c$, where $\eta(\cQ_{\mA})  $ in $\tau_q$ can be set as $  \log ( |\cQ_{\mA}|)$. Thus, the oracle $r^\star$ which answers $Z_q = n^{-1} {\sum_{i=1}^n} q(\xb_i)$  for every query $q$ satisfies Definition \ref{def::oracle}.   As for  uncountable query spaces, $\eta(\cQ_{\mA})$ can be replaced with  other capacity measures such as~the Vapnik-Chervonenkis dimension and metric entropy. \end{enumerate}

To better illustrate this computational model, we formulate the proximal gradient descent algorithm for $\ell_1$-regularized estimation into the framework of  Definition \ref{def::oracle} as an    example.
\begin{example}[\em Proximal gradient descent for $\ell_1$-regularized estimation]\label{eg::sgd} Here we aim to estimate a sparse parameter vector $\btheta^*\in \cT \subseteq \RR^d$. Let $\{ \xb_i\}_{i=1}^n  \in\cX^n$ be the $n$ realizations of a  random vector $\bX\in \cX$ and let  $\ell \colon \cT \times \cX \rightarrow \RR$ be a loss function. We define the population and empirical loss functions respectively as
\$ 
L(\btheta ) = \EE\bigl  [\ell(\btheta; \bX)\bigr ]~~\text{and}~~L_n(\btheta) = \frac{1}{n} {\sum_{i=1}^n } \ell(\btheta; \xb_i).
\$
We consider the proximal gradient algorithm for minimizing $L_n(\btheta) + \lambda \| \btheta\|_1$, in which $\lambda > 0$~is~a~regularization parameter. In detail, let $S_{\lambda}\colon \RR^d \rightarrow \RR^d$ be the soft-thresholdng operator:
\$\bigl [S_{\lambda}(\ub)  \bigr]_j = \max ( 0, | u_j| - \lambda ) \cdot \sign (u_j), ~~j \in [d] .\$
The proximal gradient algorithm iteratively performs
\#\label{eq:w588}
\btheta ^{(t+1)} \leftarrow S_{\lambda } \bigl [ \btheta^{(t)} - \eta _t  \nabla L_n(\btheta^{(t)}) \bigr],
\#
where $\eta_t > 0$ is the step-size. This algorithm can be cast into the statistical query model as follows. For simplicity, for $f \colon \cT\rightarrow \RR$,  let $\partial _j f(\btheta )$ denote the partial derivative of   $f (\btheta )$ with respect to $\theta_j$. Then the query space is given by \$
\cQ_{\mA} =  \bigl \{  {\partial _ j}\ell(\btheta; \cdot ) : j\in [d], \btheta \in \cT \bigr  \}.
\$
At the $t$-th iteration of the   algorithm, we query the oracle $r^\star$, which returns $n^{-1}  \sum_{i=1}^n  q(\xb_i)$ for any query $q$, using the query function $\partial _{j} \ell(\btheta^{(t)}, \cdot )$ for each $j\in[d]$.  That is, $r^\star$   returns the $j$-th  component of   $\nabla L_n(\btheta^{(t)})$.
The algorithm then performs \eqref{eq:w588} using the responses of the oracle $r^\star$ and obtains $\btheta ^{(t+1)}$. Let $T'$ denote the total number of iterations of the proximal gradient descent algorithm. Here the corresponding oracle complexity is $T = T'd$, since we query the gradient in a coordinate-wise manner.

\end{example}

As is discussed in \S\ref{sec::intro},  our definition of computational model follows from the one in  \cite{feldman2013statistical, feldman2015statistical, feldman2015complexity, yi2016more, wang2015sharp, lu2018edge}.
	    In order to faithfully characterize the uniform deviation of the response random variable $Z_q$, we slightly modify the {\sf VSTAT} oracle model  \citep{feldman2013statistical, feldman2015statistical, feldman2015complexity} by introducing the notions of tail probability $\xi$ and query space capacity $\eta(\cQ_{\cA})$.
	    We will illustrate the necessity of  these two notions in \S\ref{sec::upper_bound}.
	

Based on Definition \ref{def::oracle},  we consider the lower bounds of oracle complexity for hypothesis testing problems.  Let the  statistical model of interest be indexed by a  parameter $\btheta  $. We consider the hypothesis testing problem $H_0\colon \btheta \in \cG_0$ versus $H_1 \colon \btheta\in \cG_1$, where $\cG_0$ and $\cG_1$ are two disjoint~parameter spaces. Let  $\cR  [\xi,n,T, M, \eta(\cQ_{\mA}) ]$ be the set of  oracles that answer the queries of $\mA$, and $\overline{\PP}_{\btheta}$ be the distribution of the random variables output by the oracle  $r\in\cR[\xi,n,T, M, \eta(\cQ_{\mA}) ]$  when the true parameter is~$\btheta$. We note that $\overline {\PP}_{\btheta}$ depends on both the parameter $\btheta$ and the oracle $r$. Even   for the same $\btheta$,   different statistical oracles yield different $\overline \PP_{\btheta}$'s. Here we omit the dependence of $\overline{\PP}_{\btheta}$ on $r$, since any oracle returns random variables that satisfy the same tail behavior, namely \eqref{eq::query_00}. Let $\cH(\mA,r)$ be the set of  test functions that deterministically depend on the  responses to $\mA$'s queries given by  $r $.   We define $\cA(T)$ as the family of $\mA$'s that interact with an oracle for no more than $T$ rounds. For an  algorithm $\mA \in \cA(T)$ and an oracle $r\in\cR[\xi,n,T, M,  \eta(\cQ_{\mA})]$, the minimax testing risk is defined as
\#\label{eq::minimax_risk_oracle}
\overline{R}_n^*(\cG_0, \cG_1;\mA, r) = \inf_{\phi\in \cH(\mA,r)}   \Bigl[ \sup _{\btheta \in \cG_0} \overline{\PP}_{\btheta}(\phi = 1) +  \sup _{\btheta \in \cG_1} \overline{\PP}_{\btheta}(\phi = 0)  \Bigr],
\#
where the infimum on the right-hand side is taken over all allowable test functions taking values in $\{0, 1\}$.
Compared with the classical notion of minimax testing risk, \eqref{eq::minimax_risk_oracle} explicitly incorporates the~computational budgets using oracle complexity $T$.  In other words, the tests are constructed  based only on the answers to $T$ queries returned by the oracle $r$.

For $\mA \in \cA(T)$, recall that we denote by $r^\star$ the specific oracle that outputs $z_{q}^* =n ^{-1} \sum_{i=1}^n q(\xb_i)$  for any  query function $q$.   Then, for $T$ rounds of queries $\{ q_t\}_{i=1}^T$, any test function based on  $z_{q_1}^*, \ldots, z_{q_T}^*$ is also a test function based on the original data $\{\xb_i\}_{i=1}^n$.  Thus, for this specific oracle, we have
\#\label{eq::minimax_2}
\overline{R}_n^* ( \cG_0, \cG_1; \cA, r^*) \geq \inf_{\phi}  \Bigl\{ \sup_{\btheta\in \cG_0} \PP_{\btheta}\bigl[\phi(\{\xb_i\}_{i=1}^n) = 1\bigr] + \sup_{\btheta\in \cG_1} \PP_{\btheta}\bigl[\phi(\{\xb_i\}_{i=1}^n) = 0\bigr]  \Bigr\},
\#
where the infimum on the right-hand side is taken over all measurable test functions on $\cX^n$, and $\PP_{\btheta}$ is the distribution of $\bX$ when the true parameter is $\btheta$.  In other words, the minimax risk in~\eqref{eq::minimax_risk_oracle}~serves as an upper bound of the classical notion of minimax risk on the right-hand side of \eqref{eq::minimax_2}.

\subsection{Sparse Gaussian Mixture Model}
To illustrate the statistical-computational tradeoffs in statistical models with heterogeneity, we first focus on the detection of sparse Gaussian mixture model as a showcase.
Given $n$ observations $\{\xb_i\}_{i=1}^n$ of a random vector $ \bX \in \RR^d$, we consider the following hypothesis testing problem
\# \label{eq::testing}
H_0 \colon \bX \sim N(\bmu_0, \bSigma) ~~\text{versus}~~
H_1 \colon \bX \sim \nu  N(\bmu_1, \bSigma) + (1- \nu) N(\bmu_2,\bSigma).
\#
Here $\nu\in(0,1)$ is a constant, $\bmu_0, \bmu_1, \bmu_2\in \RR^d$, and $\bSigma\in \RR^{d\times d}$ is a positive definite~symmetric matrix. We are interested in the high dimensional regime, where $d$ is much larger than $n$ and~$\Delta \bmu = \bmu_2 - \bmu_1 $ is $s$-sparse \citep{arias2017detection, azizyan2013minimax, azizyan2015efficient}. We assume that all the parameters, including $d$, scale with $n$ and all the limits hereafter are taken as $n\rightarrow\infty$ and that
\#\label{eq:weigen}
\lambda_* \leq \lambda_{\min} (\bSigma) \leq \lambda_{\max}(\bSigma) \leq \lambda^*.
\#
Here $\lambda_{\min} (\bSigma)$ and $\lambda_{\max}(\bSigma)$ are the largest and smallest eigenvalues of $\bSigma$, and $\lambda_*$ and $\lambda^*$ are positive absolute constants.

 Let $\btheta = (\bmu, \bmu', \bSigma)$ and $\PP_{\btheta}$ be $\nu  N(\bmu, \bSigma) + (1- \nu) N(\bmu',\bSigma)$. Therefore, $H_0$ and $H_1$ in \eqref{eq::testing} correspond to $\btheta = (\bmu_0, \bmu_0, \bSigma)$ and $\btheta = (\bmu_1, \bmu_2, \bSigma)$, respectively. We denote by $\{\xb_i\}_{i=1}^n$ the $n$ observations of $\bX$. Let $\cG_0$ and $\cG_1$ be the parameter spaces of $\btheta$ under $H_0$ and $H_1$, respectively. The classical testing risk is defined as
 \$
R_n(\cG_0, \cG_1;\phi) = \sup_{\btheta\in \cG_0} \PP_{\btheta}\bigl[\phi(\{\xb_i\}_{i=1}^n) = 1\bigr] + \sup_{\btheta\in \cG_1} \PP_{\btheta}\bigl[\phi(\{\xb_i\}_{i=1}^n) = 0\bigr],
\$
for any nonrandomized test $\phi(\{\xb_i\}_{i=1}^n)$ that takes values 0 or 1.
 Then the classical minimax risk is given by
\$ 
R_n^*(\cG_0, \cG_1) = \inf_{\phi} R_n (\cG_0, \cG_1;\phi),
\$
the infimum is over all possible tests for $H_0$ versus $H_1$,
with no limit on computational complexity.
Also, we define the signal strength  $\rho(\btheta) $ as a nonnegative function of $\btheta$.  Let $\cG_0 = \{ \btheta \colon \rho(\btheta) = 0\}$ and $\cG_1 = \{ \btheta\colon \| \Delta \bmu \|_0 \leq s,\rho(\btheta) \geq \gamma_n\}$ for some $\gamma_n > 0$. Next we define two quantities that characterize  the difficulty of the detection problem from statistical and computational perspectives, respectively.
\begin{definition}\label{def::minimax_computation_rate}
A sequence $\alpha_n^*$ is a minimax separation rate if it satisfies the following two conditions:
\begin{enumerate}[label=(\roman*)]
\item For any sequence $\gamma_n$ such that $\gamma_n = o(\alpha_n^*)$, we have $\lim_{n\rightarrow \infty} R_n ^* (\cG_0, \cG_1)  =  1$.
\item For any sequence $\gamma_n$ such that $\gamma_n = \Omega(\alpha_n^*)$, we have $\limn R_n ^* (\cG_0, \cG_1) < 1$.
\end{enumerate}
  In addition, a sequence $\beta_n^*$ is a computationally feasible minimax separation rate if it satisfies the~following conditions:
\begin{enumerate}[label=(\roman*)]
\item  For any sequence $\gamma_n$ such that $\gamma_n =  o(\beta_n^*)$,  for any constant $\eta > 0$ and any $\mA \in \cA(d^\eta)$, there~is an 
 oracle $r\in\cR[\xi, n, T,M, \eta(\cQ_{\mA})]$ such that  $\lim_{n\rightarrow \infty}  \overline{R}_n^*(\cG_0, \cG_1;\mA, r) =  1$.
\item For any sequence $\gamma_n$ that satisfies $\gamma_n =  \Omega(\beta_n^*)$, there exist some $\eta >0$  and $\mA \in \cA(d^\eta)$ such~that, for any $r\in\cR[\xi, n, T, M, \eta(\cQ_{\mA})]$, it holds that $\limn \overline{R}_n^*(\cG_0, \cG_1;\mA, r) <1 $.
\end{enumerate}
\end{definition}
By this definition, an algorithm is considered efficient is it can be implemented  using  $d^\eta$ queries for  some  $\eta > 0$. That is,  the computational budget   of the algorithm is a  polynomial in  $d$.
In \S\ref{sec::theory_GMM}, we will show that a gap between $\alpha_n^*$ and $\beta_n^*$.  Namely, computational feasibility comes at a cost of statistical accuracy.

\subsection{Sparse Mixture of Regression Model}\label{sec::bg_reg}
In addition to the  sparse Gaussian mixture  detection, our second example is the detection of sparse mixture of regressions. We focus the emblematic setting where the mixture consists of two symmetric components. In detail,
 we assume that the response  and the covariates  satisfy
\# \label{eq::regression_model}
Y = \eta \cdot \bbeta ^\top \bX + \epsilon,
\#
where  $\bbeta \in \RR^d$ is the regression parameter,  $\eta$ is  the latent variable that has Rademacher distribution over $\{-1, 1\}$, and  $\epsilon\sim N(0, \sigma^2)$ is the Gaussian random noise.~Moreover, we assume that~$\bX \sim N(0 ,\Ib)$, and  $\sigma$ is unknown. Let $ \bZ = (Y, \bX) \in \RR^{d+1}$.~Given $n$ observations $\{ \zb_i = (y_i, \xb_i) \}_{i=1}^n$ of the model in \eqref{eq::regression_model}, we  aim to~test whether the distribution of $\bZ$  is a    mixture. In detail, we consider the   testing problem
\#\label{eq::testing2}
H_0 \colon \bbeta = \zero ~~\text{versus}~~
H_1 \colon \bbeta\neq \zero.
\#
We are interested in the high dimensional setting where $n \ll d$ and $\bbeta $ is $s$-sparse, where  $s$ is known.
For notational simplicity, we  denote $ ( \bbeta, \sigma^2) $ by $\btheta$ and define $\PP_{\btheta}$ as the distribution of $\bZ$ satisfying~\eqref{eq::regression_model} with regression parameter $\bbeta$ and noise level $\sigma^2$.

 We define the parameter spaces of the null and alternative hypotheses as $\cG_0 =  \{  \btheta   \colon\rho(\btheta) = 0 \}$~and $\cG_1 =  \{ \btheta  \colon    \| \bbeta\| _0 = s, \rho(\btheta) \geq \gamma_n\} $ respectively, where  $\rho(\btheta) = \| \bbeta \|_2^2 /\sigma^2$ denotes the signal strength.
For hypothesis  testing $H_0\colon  \btheta \in \cG_0$ versus $H_1 \colon \btheta  \in \cG_1  $, the minimax risk is given by
 \$
R_n^*(\cG_0, \cG_1) = \inf_{\phi} \Bigl\{  \sup_{\btheta\in \cG_0} \PP_{\btheta}\bigl [\phi(\{\zb_i\}_{i=1}^n) = 1 \bigr] +  \sup_{\btheta\in \cG_1 } \PP_{\btheta} \bigl [\phi(\{\zb_i\}_{i=1}^n) = 0 \bigr]  \Bigr\}.
\$
For detecting  mixture of regressions, we can similarly define $\alpha_n^*$ and  $\beta_n^*$ as in Definition \ref{def::minimax_computation_rate}. In \S\ref{sec::theory_reg}, we will show that a gap between $\alpha_n^*$ and $\beta_n^*$  also arises in this problem, which~implies the universality of the tradeoffs between statistical optimality and computational efficiency in statistical models with heterogeneity.

\section{Main Results for Gaussian Mixture Model} \label{sec::theory_GMM}
In the following, we present the  statistical-computational tradeoffs  in detecting Gaussian mixture model.~In specific, we establish a computational lower bound as well as matching  upper bounds  under the statistical query model specified in Definition \ref{def::oracle}.~Then we show that our results of detection~also imply tradeoffs in the problems of estimation, clustering, and feature selection in \S \ref{sec::implications_GMM}.

For the detection problem, we will first assume that $\bSigma$ is known; the unknown case will be treated in \S\ref{sec::unknown}.~When $\bSigma$ is known, we define the signal strength   as $\rho (\btheta)  =\Delta \bmu ^\top \bSigma^{-1} \Delta \bmu$, which is also known as the Mahalanobis distance. For the detection problem in \eqref{eq::testing}, we define~the null parameter space as
 \#\label{eq:wg0}
 \cG_0 (\bSigma)= \bigl\{ \btheta = (\bmu, \bmu, \bSigma)\colon \bmu \in \RR^d \bigr\}.
 \#
Let $\|\cdot\|_0$ be the number of nonzero entries of a vector. For $\gamma_n > 0 $, we~define
 \#\label{eq:wg1}
 \cG_1 (\bSigma,s, \gamma_n) =  \bigl\{ \btheta = (\bmu_1, \bmu_2, \bSigma)\colon \bmu_1, \bmu_2 \in \RR^d, \| \Delta \bmu\|_0 = s, \rho(\btheta) \geq \gamma_n \bigr\},
 \#
 where $\Delta \bmu = \bmu_2 - \bmu_1$.
In the following, we derive lower bounds and upper bounds for the detection problem
\#\label{eq::testing_w}
H_0\colon \btheta \in \cG_0(\bSigma) ~~\text{versus}~~ H_1\colon \btheta \in \cG_1(\bSigma, s, \gamma_n).
\#
For simplicity, we assume the sparsity level $s$ is known, $n \gg s$, and the mixing probability $\nu\in(0,1)$ in \eqref{eq::testing}  is a known absolute constant.
\subsection{Lower Bounds}\label{sec::lower_bound}
Recall that we define the minimax separation rate $\alpha_n^*$ and the computationally feasible minimax separation~rate $\beta_n^*$ in Definition \ref{def::minimax_computation_rate}. Before we present the result for  $\beta_n^*$, we first present a lower bound  for  $\alpha_n^*$~obtained in  \cite{arias2017detection}  for completeness.

\begin{proposition}\label{prop::info_lower_bound}
We consider the detection problem defined in \eqref{eq::testing_w} where   $\bSigma$ and $s$ are known. We assume that $\max(s , n)/d =o(1)$.  Then if $\gamma_n = o(\sqrt { s\log d /n  } )$, any hypothesis test is asymptotically powerless, i.e.,
\$\limn R_n^*  \bigl[ \cG_0(\bSigma),\cG_1( \bSigma, s, \gamma_n)\bigr] = 1.\$

\end{proposition}

\begin{proof}
See \cite{arias2017detection} for a detailed proof.
\end{proof}

In the special case where $\bSigma = \Ib$, the signal strength  is $\rho(\btheta) = \| \Delta \bmu \|_2^2$. Proposition \ref{prop::info_lower_bound} shows~that detection is impossible  if
 $
 \| \Delta\bmu \|_{2}^2 = o   ( \sqrt{ s\log d /n }  ).
$
For the sparse regime where  $ (\log n)^2 \cdot s\log d /n     = o(1)$, in \S \ref{sec::upper_bound} we will present an algorithm that constructs a hypothesis test whose   risk converges to zero asymptotically under the statistical query model, as long as  $\rho(\btheta) = \Omega ( \log n \cdot \sqrt{ s\log d /n } )$.
Here the $\log n$ term arises due to an artificial truncation that ensures the query functions to be bounded, as specified in Definition \ref{def::oracle}.
In other words, neglecting this $\log n$ term, 
the information-theoretic lower bound in Proposition \ref{prop::info_lower_bound} is tight and the minimax separation rate is $\alpha _n^* =  \sqrt{ s\log d /n } $ for Gaussian mixture detection~with known covariance matrix.

The next theorem establishes a lower bound for $\beta^*_n$, which shows that $\alpha_n^*$ is not achievable by any  computationally efficient algorithm under the statistical query model. Meanwhile, the existence~of~a computationally tractable test requires a much larger signal strength   than that in Proposition \ref{prop::info_lower_bound}.

\begin{theorem} \label{thm::lower_known_cov}
For the detection problem defined in \eqref{eq::testing_w} with both $\bSigma$ and $s$ known,  we assume that $ \max(s^2 , n)/d = o(1)$ and there exists a sufficiently small constant $\delta >0$ such that $s^2 / d^{1-\delta} = O(1)$.
Then  if $\gamma_n = o ( \sqrt{s^2/ n}  )$, for any constant $\eta>0$,   and any  $\mA \in \cA(T)$ with $T = O(d^\eta)$, there exists   an oracle $r\in\cR[\xi,n,T, M,  \eta(\cQ_{\mA}) ]$ such that
\$ \lim_{n\rightarrow \infty}   \overline{R}_n^* \bigl[ \cG_0(\bSigma), \cG_1(\bSigma, s, \gamma_n); \mA , r\bigr] =  1.
\$
In other words,  any test procedures  under the  statistical query model  defined in \eqref{def::oracle} with oracle complexity $T = O(d^\eta)$  is asymptotically powerless if
$\gamma_n = o ( \sqrt{s^2/ n}  )$.
\end{theorem}
\begin{proof}
See \S\ref{proof::thm::lower_known_cov} for a detailed proof.
\end{proof}

In the special case where $\bSigma = \Ib$, Theorem \ref{thm::lower_known_cov}  shows that
  any  computationally tractable test is asymptotically powerless if $\| \Delta \bmu \|_2^2 =  o ( \sqrt{ s ^2 /n}  )$. As we will show in \S\ref{sec::upper_bound}, there is a computationally tractable test under the statistical query model that is asymptotically powerful if
  $
  \rho(\btheta) =   \Omega  (   \sqrt{ s^2 \log d/n} ),
  $
  where we ignore a $\log n$   term incurred by truncation.
Hence, the lower bound in Theorem \ref{thm::lower_known_cov} is tight up to a logarithmic factor.~Ignoring this $\log d$ term,  the computationally feasible minimax separation rate is roughly $\beta _n^* = \sqrt{s^2/n}$. The gap between $\alpha _n^*$~and $\beta _n^*$ suggests that we have~to~pay at least a factor of $\sqrt{s/\log d}$ in terms of the signal strength to~attain computational tractability, which~exhibits  fundamental tradeoffs between computation and statistics in   sparse Gaussian mixture~model.

Note that the lower bounds in Proposition \ref{prop::info_lower_bound} and Theorem \ref{thm::lower_known_cov} are attained in subsets~of~$\cG_0 (\bSigma)$ and $\cG_1 (\bSigma,s, \gamma_n)$ defined in \eqref{eq:wg0} and \eqref{eq:wg1}, namely
\#\label{eq:wbarg}
& \overbar{\cG}_0 (\bSigma)=\bigl\{ \btheta = (\bmu, \bmu, \bSigma)\colon \bmu = \zero \bigr\} \subseteq \cG_0 (\bSigma),\notag\\
& \overbar{\cG}_1 (\bSigma,s, \gamma_n) = \Bigl\{ \btheta = \bigl[- \beta (1- \nu)   \vb, \beta\nu  \vb, \bSigma\bigr]\colon \vb \in \cG( s), \rho(\btheta) \geq \gamma_n\Bigr\} \subseteq \cG_1 (\bSigma,s, \gamma_n),\notag\\
& \text{where} ~~\cG( s) = \bigl\{ \vb \in \{ -1, 0,1\}^d \colon \| \vb \|_0  = s \bigr\}.
\#
In other words, this model subclass represents one of the most challenging settings in terms of both~computational and information-theoretic difficulties. To better illustrate the sharpness of Proposition~\ref{prop::info_lower_bound}~and Theorem \ref{thm::lower_known_cov}, in \S\ref{sec::upper_bound} we mainly focus on the upper bounds for $\overbar{\cG}_0 (\bSigma)$ and $\overbar{\cG}_1 (\bSigma,s, \gamma_n)$ for simplicity. More general upper bounds for    $\cG_0 (\bSigma)$ and $\cG_1 (\bSigma,s, \gamma_n)$ are deferred to Appendix \ref{ap:GMM}.

\subsection{Upper Bounds}\label{sec::upper_bound}
As discussed above, in this section we first consider the following restricted testing problem
\#\label{eq::reduced_testing}
H_0\colon \btheta \in \overbar{\cG}_0(\bSigma) ~~\text{versus}~~ H_1\colon \btheta \in \overbar{\cG}_1(\bSigma, s, \gamma_n),
\#
where $\overbar{\cG}_0(\bSigma)$ and $\overbar{\cG}_1(\bSigma, s, \gamma_n)$ are defined in \eqref{eq:wbarg}.   Note both classes have mean zero, but class 0 has a variance $\bSigma$ whereas class 1 has variance $\bSigma + \nu (1-\nu) \beta^2 \vb\vb^T$.  To match the information-theoretic lower bound~in Proposition \ref{prop::info_lower_bound}, we consider
  the following~sequence of query functions
\#\label{eq::query_fun1}
q_{\vb}(\xb) =  \frac{(\vb^\top \bSigma^{-1 } \xb  )^2}{\vb^\top \bSigma^{-1} \vb} \cdot \ind \Bigl \{  \bigl |   \vb^\top \bSigma^{-1} \xb  \bigr |   \leq R \sqrt{ \log n}  \cdot \sqrt{\vb^\top \bSigma^{-1 } \vb }   \Bigr  \} , ~\text{where}~ \vb \in \cG( s).
\#
Here $R$ is an absolute constant and  $\cG( s)$ is defined in \eqref{eq:wbarg}. We apply truncation in \eqref{eq::query_fun1} to obtain bounded queries. In this case, we have computational budget $T = |\cG( s) |  = 2^{ s } \cdot {d \choose { s }}$ and $\eta(\cQ_{\mA}) = \log \bigl[2^{ s } \cdot {d \choose { s }} \bigr]$.~For each query function $q_{\vb}$, let the random variable returned by the oracle be $Z_{q_{ \vb} }$.

Note that if  ignoring the truncation  in \eqref{eq::query_fun1}, we have  $\EE [ q_{\vb} (\bX) ]  = 1$ under $H_0$ and $\EE [ q_{\vb} (\bX) ] = 1 + \nu(1-\nu) \beta^2 \cdot  (\vb ^\top \bSigma^{-1} \vb)$ under $H_1$.  Therefore, we would reject $H_0$ whenever there is a direction $\vb$ such that $\EE[ q_{\vb} (\bX)  ] > 1$.  This leads us to define the following test function,
\#\label{eq::test_fun1}
\ind \Bigl\{ \sup_{\vb\in \cG( s) }Z_{q_{ \vb} }\geq  1 + 2 R ^2\cdot  \log n  \cdot  \sqrt{    [ s \log  (2d) + \log ( 1/\xi) ]   / n} \Bigr\}.
\#

The next theorem shows that the lower bound in Proposition~\ref{prop::info_lower_bound} is tight within $\overbar{\cG}_0(\bSigma)$ and $\overbar{\cG}_1(\bSigma, s, \gamma_n)$ up to a logarithmic factor in $n$.

\begin{theorem}\label{thm::test1}
We consider the sparse mixture detection problem in \eqref{eq::reduced_testing}. Let $R$ in \eqref{eq::query_fun1} be a sufficiently large constant. If
\#\label{eq:wrhobeta}
\rho(\btheta)  =\nu(1- \nu)\bmu ^\top \bSigma^{-1} \bmu \geq \gamma_n = \Omega \bigl \{  \sqrt{   [s \log  (2d) + \log ( 1/\xi)]\cdot \log n /n } \bigr \} ,
\#
then for the test function $\phi$ defined  in \eqref{eq::test_fun1}, we have
\$
{ \sup_{\btheta\in \overbar{\cG}_0(\bSigma)}} \overline{\PP}_{\btheta}(\phi = 1) + { \sup_{\btheta\in \overbar{\cG}_1(\bSigma, s, \gamma_n)}} \overline{\PP}_{\btheta}(\phi = 0)  \leq 2\xi.
\$
\end{theorem}
\begin{proof}
See \S\ref{proof::thm::test1} for a detailed proof.
\end{proof}

Note that here $\bmu$ corresponds to $\Delta \bmu = \bmu_2 - \bmu_1$ in the previous discussion, since in \eqref{eq:wbarg} we~have $\bmu_1 = -(1- \nu)\bmu$ and $\bmu_2 = \nu\bmu$. To illustrate, we consider the special case with $\bSigma = \Ib$, where we have 
\$
\rho(\btheta) = \nu(1- \nu) \| \bmu \|_2^2   =  \Omega  \bigl \{\log n \cdot \sqrt{    [s \log  (2d) + \log ( 1/\xi)]  / n}  \bigr\} .
\$ Note that the hypothesis test defined in  \eqref{eq::test_fun1} is asymptotically powerful if $\xi = o(1)$. Setting $\xi = 1/d$,  \eqref{eq:wrhobeta} is equivalent to $\gamma_n = \Omega  ( \log n \cdot  \sqrt{  s \log d /n}  )$. We note that the $\log n$ term arises due to the truncation in   \eqref{eq::query_fun1}, which ensures the query functions to be bounded. Such a truncation is  unnecessary if we construct the hypothesis test in \eqref{eq::test_fun1} using $\{ \xb_i \}_{i =1}^n$. Thus,
by Proposition \ref{prop::info_lower_bound}, we conclude that $\alpha_n^* = \sqrt{ s\log d/n}$~is the minimax separation rate for $\bSigma = \Ib$. Similar argument also holds for general $\bSigma$.

Recall that the test defined by \eqref{eq::test_fun1}   requires superpolynomial oracle complexity.~To
construct~a computationally tractable test,
we consider the following sequence of query functions,
\#\label{eq::query_fun2}
q_j(\bX) =   X_j^2 / \sigma_{j} \cdot \ind \{ | X_j / \sqrt{ \sigma_j } | \leq R \cdot \sqrt{\log n} \} ,~\text{where~} j \in [d].
\#
Here  $\sigma_{j}$ is the $j$-th diagonal element of $\bSigma$ and $R >0$ is an absolute constant. Similar to  \eqref{eq::query_fun1}, we apply truncation in \eqref{eq::query_fun2} to ensure boundedness. Then we have $T = d$, $\eta(\cQ_{\mA}) = \log d$, and $M = R^2 \cdot \log n$.~We define the test function as
\#\label{eq::test_fun2}
\ind \Bigl[\max _{j\in [d]} Z_{q_j} \geq  1 + 2R ^2 \cdot  \log n \cdot\sqrt{\log ( d / \xi)  / n } \Bigr].
\#
The test seeks to detect if any random variable $X_j^2/\sigma_j$ has variance bigger than 1 and reject the null hypothesis when there is such one. The following theorem~shows that the test defined above is asymptotically powerful if $\gamma_n = \Omega (\log n \cdot \sqrt{s^2 \log d / n}  )$.

\begin{theorem}\label{thm::test2}
We consider the sparse mixture detection problem in \eqref{eq::reduced_testing}. Let $R$ in \eqref{eq::query_fun2} be a sufficiently large constant. If
\#\label{eq:wnvsd}
  \nu(1-\nu) \beta^2 / \min_{j\in [d]} \sigma_{j}  = \Omega \bigl[\log n \cdot  \sqrt{\log ( d / \xi)/ n} \bigr],
\#
 then for $\phi$ being the test function in \eqref{eq::test_fun2}, we have
\$
 \sup_{\btheta\in \overbar{\cG}_0(\bSigma)} \overline{\PP}_{\btheta}(\phi = 1) + \sup_{\btheta\in \overbar{\cG}_1(\bSigma, s, \gamma_n)} \overline{\PP}_{\btheta}(\phi = 0)  \leq 2\xi.
\$
\end{theorem}
\begin{proof}
See \S\ref{proof::thm::test2} for a detailed proof.
\end{proof}

Theorem \ref{thm::test2} can be understood as follows. By \eqref{eq:weigen} we have that \eqref{eq:wnvsd} is equivalent to
\$
\rho(\btheta) = \nu(1- \nu) \bmu ^\top \bSigma^{-1} \bmu \geq \gamma_n = \Omega  \bigl[ \log n \cdot\sqrt{  s^2  \log (d/\xi) / n }  \bigr].
\$
Also, the hypothesis test defined in  \eqref{eq::test_fun2} is asymptotically powerful when $\xi = o(1)$. Setting $\xi = 1/d$, we have $\gamma_n = \Omega  (\log n \cdot \sqrt{ s^2 \log d /n} )$.   Together with the lower bound derived in Theorem \ref{thm::lower_known_cov}, we conclude that, when ignoring the $\log n$ term incurred by truncation,  the computationally feasible minimax separation rate $\beta_n^*$ is between $\sqrt{s^2 /n}$ and $\sqrt{s^2 \log d/n}$. In fact,~the test defined in \eqref{eq::query_fun2} and \eqref{eq::test_fun2} can be viewed as the diagonal thresholding procedure applied on the covariance matrix of $\bX$ \citep{johnstone2012consistency} under the statistical query model. Applying the covariance thresholding algorithm of \cite{deshpande2014sparse}, we can further close the gap between $\sqrt{s^2 /n}$ and $\sqrt{s^2 \log d/n}$, which implies the computationally feasible minimax separation rate $\beta_n^*$ is $\sqrt{s^2 /n}$.
This approach can similarly be formulated into the statistical query model, for which we omit the details for the sake of succinctness. Besides, we remark that the upper bounds can be extended from $\overbar{\cG}_0 (\bSigma)$ and $\overbar{\cG}_1 (\bSigma,s, \gamma_n)$ in \eqref{eq:wbarg} to ${\cG}_0 (\bSigma)$ and ${\cG}_1 (\bSigma,s, \gamma_n)$ in \eqref{eq:wg0}~and \eqref{eq:wg1}, with changes of the procedure and proof,  which are deferred to Appendix \ref{ap:GMM} since our main focus is on the computational lower bounds.

It is worth mentioning that the tests defined in \eqref{eq::query_fun1}-\eqref{eq::test_fun1} and \eqref{eq::query_fun2}-\eqref{eq::test_fun2} can be implemented using $\{ \xb_{i} \}_{i =1}^n$ by replacing $Z_q$ with $n^{-1}    \sum_{i=1}^n q(\xb_i)$, which yields the same guarantees by Bernstein's inequality and the union bound. Recall that for the algorithm defined by \eqref{eq::query_fun1}, $\eta(\cQ_{\mA}) = \log [|\cG( s) |]  = \log \bigl[2^{ s } \cdot {d \choose { s }} \bigr]$, and for the one defined by \eqref{eq::query_fun2}, $\eta(\cQ_{\mA}) =  \log d$. Suppose we set $\xi = 0$ and~$\tau_q =R^2 \log n \cdot  \sqrt{2/n}$ in Definition \ref{def::oracle}. Then following the proof of Theorem \ref{thm::test1},  the test defined in \eqref{eq::test_fun1} is asymptotically powerful if $\gamma_n  = \Omega  (\log n \cdot \sqrt{1/n}  )$, which contradicts the information-theoretic lower bound in Proposition \ref{prop::info_lower_bound}. This is because the statistical query model with  $\xi = 0$ and $\tau_q  = R^2 \log n \cdot  \sqrt{2/n} $ can not be implemented using $\{ \xb_{i} \}_{i =1}^n$.
 This hypothetical example indicates the necessity of incorporating the notions of tail probability~and the capacity of  query spaces  to better capture real-world algorithms.

\subsection{Extensions to Unknown Covariance}\label{sec::unknown}

In the sequel, we extend our analysis to the case in which $\bSigma $ is unknown. In this case, we define the signal strength as  $\rho^\star(\btheta) = \| \Delta \bmu \|_2^4 / \Delta \bmu^\top \bSigma \Delta \bmu$.  Throughout this section, we assume that $\nu = 1/2$~and the parameter spaces of the null and alternative hypotheses are given by
\$
&\cG_0 = \bigl\{ \btheta= (\bmu, \bmu, \bSigma ) \colon \bmu \in \RR^d, \bSigma \succ {\bf{0}} \bigr\}, \\
&\cG_1(s,\gamma_n) = \bigl\{\btheta= (\bmu_1, \bmu_2, \bSigma ) \colon \bmu_1,\bmu_2 \in \RR^d, \bSigma \succ {\bf{0}}, \| \Delta \bmu \|_0 = s, \rho^\star(\btheta) \geq \gamma_n \bigr\}.\notag
\$
Here we assume the sparsity level $s$ of $\Delta \bmu$ is known. The following proposition of \cite{arias2017detection} gives the minimax lower bound for the sparse mixture detection problem
\#\label{eq:wos}
H_0\colon \btheta \in \cG_0,~~\text{versus}~~H_1\colon \btheta\in \cG_1 (s, \gamma_n).
\#

\begin{proposition}\label{prop::info_lower_bound2}
For the testing problem in \eqref{eq:wos}, we assume that $\limn \max(s , n)/d = 0$.~Then any hypothesis test is asymptotically powerless, i.e.,
$
{ \lim_{n\rightarrow  \infty}} R_n^*  [ \cG_0, \cG_1(  s, \gamma_n)]  =  1,
$ if
\$
\gamma_n = o  \bigl[ ( s\log d /n )^{1/4}\bigr].
\$
\end{proposition}
\begin{proof}
See \cite{arias2017detection} for a detailed proof.
\end{proof}
By  Proposition \ref{prop::info_lower_bound2},  the minimax separation rate   $\alpha_n^* $ is at least $ ( s\log d /n )^{1/4} $ for unknown $\bSigma$. Thus, seen from  Theorem~\ref{thm::test1},
this setting  is  harder than the case where $\bSigma$ is known.~The next  theorem establishes the  computational lower bound under the statistical query model.

\begin{theorem}\label{thm::lower_unknown_cov}
For the testing problem in \eqref{eq:wos}, we assume that $\limn \max(s , n)/d = 0$ and    there exists a sufficiently small constant $\delta >0$ such that $s^2 / d^{1-\delta} = O(1)$.
If  $\gamma_n = o  [ (s^3/ n)^{1/4}  ]$, then for any constant $\eta>0$,    and any  $\mA \in \cA(T)$ with $T = O(d^\eta)$, there exist an oracle $r\in\cR[\xi,n,T, \eta(\cQ_{\mA}) ]$ such that
\$\lim _{n\rightarrow \infty} \overline{R}_n^*\bigl[ \cG_0, \cG_1(  s, \gamma_n); \mA , r \bigr]   =  1.
\$
\end{theorem}
\begin{proof}
See \S\ref{proof::thm::lower_unknown_cov} for a detailed proof.
\end{proof}

Combining Proposition \ref{prop::info_lower_bound2} and Theorem \ref{thm::lower_unknown_cov},  we observe a  similar statistical-computational tradeoff when the covariance matrix $\bSigma $ is unknown. More specifically, while the existence of an asymptotically powerful test requires  $\rho^\star(\btheta) = \Omega [   (  s\log d/n) ^{1/4}  ]$, the existence of a powerful and computationally tractable test requires $\rho^\star(\btheta) = \Omega [ ( s^3 /n )^{ 1/4}  ]$. Indeed, \cite{arias2017detection} propose an asymptotically powerful but computationally intractable test for  $\rho^\star(\btheta) = \Omega [   (  s\log d/n) ^{1/4}  ]$, and a computationally efficient and asymptotically powerful test for $\rho^\star(\btheta) = \Omega [ ( s ^4 \log d /n) ^{1/4} ]$. These two tests can both be formulated using the statistical query model. Thus we conclude that,  for the setting with $\bSigma$ unknown, the minimax separation rate is $\alpha^*_n =(  s\log d/n) ^{1/4}$, and the computationally feasible minimax separation rate $\beta^*_n$ is between $(s^3/ n)^{1/4}$ and $( s ^4 \log d /n) ^{1/4}$. 
That is to say, there~exists at least an $(s^2 /\log d)^{1/4}$ price to pay in the minimum signal strength to obtain computational tractability.

\subsection{Implications for Estimation, Support Recovery, and Clustering}\label{sec::implications_GMM}
Note that detection is an easier task than estimation, support recovery, and clustering. For example, if it is possible to consistently estimate $\bmu_1$ and $\bmu_2$ in $\btheta = (\bmu_1, \bmu_2, \bSigma)$ using $\hat{\bmu}_1$ and $\hat{\bmu }_2$, then we can construct an asymptotically powerful test based upon $ \hat{\bmu}_1 - \hat{\bmu}_2 $ to detect the mixtures. Therefore, the lower~bounds for detection also hold for estimation, e.g., if there exists no asymptotically~powerful~test for $\rho(\btheta) = o(\zeta_n)$, then under the same condition, we can not consistently estimate $\bmu_1$ and $\bmu_2$. Similar arguments also hold for support recovery and clustering.  Therefore, Theorem \ref{thm::lower_known_cov} has the~following implications.
\begin{enumerate}[label=(\roman*)]
\item \cite{arias2017detection} consider the recovery of the support of $\Delta\bmu$. The information-theoretic lower bound for consistent recovery is $\rho(\btheta) = o   (  \sqrt{ s \log d/n }) $. In comparison, efficient algorithms can succeed as long as $\rho(\btheta) = \Omega  (\log n \cdot  \sqrt{s^2 \log d/n }  ) $. Theorem \ref{thm::lower_known_cov} indicates that, efficient algorithms can not do better than $\rho(\btheta) = \Omega  ( \sqrt{s^2 /n}  )$. In other words, ignoring the logarithmic factors, the gap observed by \cite{arias2017detection} can not be eliminated.
\item \cite{azizyan2013minimax, azizyan2015efficient} consider the clustering problem and observe the same phenomenon. Theorem \ref{thm::lower_known_cov} implies consistent clustering with efficient algorithms requires $\rho(\btheta) = \Omega  ( \sqrt{s^2 /n}  )$. In other words, the conjecture of \cite{azizyan2013minimax, azizyan2015efficient} is correct, i.e., to achieve consistent clustering with computational efficiency, a statistical price of $\sqrt{s/\log d}$ must be paid.
\end{enumerate}
More formally, we summarize the aforementioned implications with the next theorem. For simplicity, we focus on the setting with known $\bSigma$. Similar results can be obtained for unknown $\bSigma$ in the same fashion.
\begin{theorem}\label{col::implications}
For the high dimensional Gaussian mixture model with sparse mean separation,~we assume $\| \Delta \bmu \|_0 = s$. If $\Delta\bmu ^\top \bSigma^{-1} \Delta \bmu =  \gamma_n  = o(\sqrt{s^2 /n})$ and the assumptions in Theorem~\ref{thm::lower_known_cov} hold, for any constant  $\eta >0$,    and any $\mA \in \cA(T)$ with $T = O(d^{\eta})$, there~exists an oracle $r \in\cR[\xi,n,T, M, \eta(\cQ_{\mA}) ]$ such that~the following claims hold under the statistical query model.
\begin{enumerate}[label=(\roman*)]
\item There exists an absolute constant $C>0$ such that for any estimator $\hat \bmu_1$ of $\bmu_1$ and $\hat\bmu_2$ of $\bmu_2$,~we have
\#\label{eq::argument1}
\overline \PP_{\btheta} \Bigl[  \max_{\ell\in\{1,2\}}( \hat\bmu_{\ell} - \bmu_{\ell} ) ^\top \bSigma ^{-1} ( \hat \bmu_{\ell} - \bmu_{\ell})  >  \gamma_n  / 64  \Bigr] \geq C.
\#
\item  There exists an absolute constant $C> 0$ such that for any estimator $\Delta \hat{\bmu}$ of $\Delta \bmu$, we have
\#\label{eq::argument2}
\overline{\PP }_{\btheta} \bigl[\supp( \Delta \hat \bmu) \neq \supp ( \Delta \bmu) \bigr] \geq  C.
\#
\item We define the density function of $N(\bmu, \bSigma)$ as $f(\xb;\bmu, \bSigma)$. Let
\#\label{eq:optimal_cluster}
F_{\btheta} (\xb)  = \begin{cases}
	  1 \qquad & \text{if}~\nu \cdot f(\xb, \bmu_1, \bSigma ) \geq (1 - \nu ) \cdot f(\xb, \bmu_2, \bSigma )  \\
	2 \qquad & \text{otherwise},
\end{cases}
\#
be the assignment function of correct clustering. There exists an absolute constant $C >0$~such that for any assignment function $F \colon \RR^d \rightarrow \{ 1,2\}$, we have
\#\label{eq::argument3}
\min_{\Pi}  \overline \PP _{\theta}\Bigl\{ \Pi\bigl[F(\bX)\bigr]  \neq   F_{\btheta} (\bX) \Bigr\}   \geq C.
\#
Here the minimum is taken over all permutations $\Pi\colon \{1,2\} \rightarrow \{1,2\}$.
\end{enumerate}

\end{theorem}
\begin{proof}
See \S\ref{proof::col::implications} for a detailed proof.
\end{proof}

\subsection{Relationship to  Sparse PCA}\label{sec::connet_sparse_PCA}
The detection of sparse Gaussian mixtures is closely related to sparse principal component detection \citep{berthet2013computational, berthet2013optimal}, where one aims to test
\$ 
H_0 \colon \bX \sim N(0, \Ib)~~\text{versus} ~~H_1 \colon \bX \sim N(0, \Ib + \lambda \vb \vb^{\top}).
\$
Here $\lambda >0$ and $\vb$ is an $s$-sparse vector satisfying $\| \vb\|_2 = 1$. Let $\bSigma = \Ib$ in \eqref{eq::testing}. Under the~alternative hypothesis, we have
\$
\Cov(\bX) = \Ib + \nu(1- \nu) \Delta \bmu \Delta \bmu^\top,
\$
while under the null we have $\Cov(\bX) = \Ib$. That is to say, sparse Gaussian mixture detection~is~easier than sparse principal component detection, in the sense that we can use algorithms for the latter to solve the former. Therefore, Theorem \ref{thm::lower_known_cov} implies an unconditional lower bound for sparse principal component detection under the statistical query model, i.e., $\lambda = \|\Delta \bmu\|_2^2 = \Omega (\sqrt{s^2/n} )$ is necessary~for attaining computational tractability. This result is also obtained in \cite{wang2015sharp} and mirrors the conditional computational lower bound of \cite{berthet2013computational, berthet2013optimal}.  Furthermore, it is also worth mentioning that \cite{berthet2013computational} show that detecting the presence of a planted clique in a graph is easier than detecting the existence of a sparse principal component. However, it remains unclear whether we can use any algorithm that successfully detects the Gaussian mixtures to solve the planted clique detection problem or vice versa.

\section{Main Results for Mixture of Regressions} \label{sec::theory_reg}
As another example of statistical models with heterogeneity, we introduce the theoretical results for detecting mixture of regressions in this section. For ease of presentation, as stated in \S \ref{sec::bg_reg}, we focus on the mixture of two symmetric sparse regression components in high dimensions.

Recall that for the mixture of regression model in \eqref{eq::regression_model}, we denote by $\btheta = ( \bbeta, \sigma^2)$ the model  parameters.
For the  detection problem defined in \eqref{eq::testing2}, we assume that $\sigma$ is unknown.  Then the  parameter space for the null hypothesis is defined as
 $
 \cG_0 =\{ \btheta = (\zero, \sigma^2 )\colon \sigma >0\}.
 $
For any $\gamma_n >0 $ and sparsity level $s$, we consider the following parameter space for the alternative hypothesis,
 \$
 \cG_1 ( s, \gamma_n) = \bigl\{ \btheta = (\bbeta, \sigma^2)\in \RR^{d+1}\colon    \|  \bbeta   \|_0 = s, \rho(\btheta) \geq \gamma_n \bigr\},
 \$
where $\rho(\btheta) = \| \bbeta\|_2^2 / \sigma ^2 $ is the signal strength.

\subsection{Lower Bounds}\label{sec::lower_bound_reg}
Now we  establish the lower bounds  for the detection of mixture of regressions.
The next proposition characterizes the minimax separation rate $\alpha_n^*$  of the detection problem.

\begin{proposition}\label{prop::info_lower_bound_reg}
We consider the detection problem defined in \eqref{eq::testing2}  where  $\sigma$  is unknown and~the parameter spaces for the null hypothesis and the alternative hypothesis are given by $\cG_0 $ and $\cG_1( s, \gamma_n)$, respectively.  We assume that $\limn \max(s, n)/d = 0$. If
\#\label{eq::info_detection_rate_reg}
 \gamma_n = o ( \sqrt{ s \log d /n}),
\#
then any hypothesis test is asymptotically powerless, that is, $ \lim_{n\rightarrow \infty} R_n^*  [ \cG_0 , \cG_1( s, \gamma_n)] = 1.$
\end{proposition}
\begin{proof}
See \S \ref{proof::prop::info_lower_bound_reg} for a detailed proof.
\end{proof}

In \S \ref{sec::upper_bound_reg} we will show, there exists an algorithm with superpolynomial oracle complexity that gives an asymptotically powerful test under the statistical query model as long as
  $\gamma_n = \Omega [ \log n \cdot \sqrt{ s \log d /n}]$, where $\log n$  arises due to an artificial truncation which ensures the query functions to be  bounded.~Then, together with  \eqref{eq::info_detection_rate_reg}, we conclude that  the information-theoretic lower bound in Proposition  \ref{prop::info_lower_bound_reg}~is~tight up to a  $\log n$ term,  and that~$\alpha _n^* = \sqrt{ s \log d /n}$ is the minimax separation rate.

  In the sequel, we establish the computational lower bound, which implies the above information-theoretic lower bound  is not achievable by any computationally tractable hypothesis tests under the statistical query model.


\begin{theorem} \label{thm::lower_reg}
For the detection problem defined in \eqref{eq::testing} with unknown $\sigma$ and known $s$, we assume $\limn \max(s^2, n)/d = 0$, and there exists a sufficiently small constant $\delta >0$ such that~$s^2 / d^{1-\delta} = O(1)$. If $\gamma_n = o ( \sqrt{s^2/ n}  )$, for any constant $\eta>0$,   and any  $\mA \in \cA(T)$ with $T = O(d^\eta)$, there exists~an oracle $r\in\cR[\xi,n,T, M, \eta(\cQ_{\mA}) ]$ such that
\$\limn \overline{R}_n^* \bigl[ \cG_0 , \cG_1(s, \gamma_n); \mA , r\bigr] = 1.
\$
Therefore,  for detecting mixture of regressions, any hypothesis test  with $T = O(d^\eta)$ oracle complexity under the statistical query model  is asymptotically powerless if
$\gamma_n = o ( \sqrt{s^2/ n}  )$.
\end{theorem}
\begin{proof}
See \S\ref{proof::thm::lower_reg} for a detailed proof.
\end{proof}

Furthermore, similar to the Gaussian mixture model, if   $
  \rho(\btheta) =   \Omega  (\log n \cdot  \sqrt{s^2 \log d/n})
  $,  it can be shown that    there exists a computationally tractable test under the statistical query model that  is asymptotically powerful. Hence, our computational  lower bound in Theorem \ref{thm::lower_known_cov} is tight up to  logarithmic factors, and the computationally feasible minimax separation rate is $\beta _n^* = \sqrt{s^2/n}$ when ignoring the logarithmic terms. Such a  gap between $\alpha _n^*$ and $\beta _n^*$ indicates that a  factor of $\sqrt{s/\log d}$ in terms of statistical optimality  has to be compromised so as to achieve computational tractability.

As shown in their proofs, the lower bounds in Proposition \ref{prop::info_lower_bound_reg} and Theorem \ref{thm::lower_reg} are obtained by restricting on the following subsets of $\cG_0  $ and $\cG_1 ( s, \gamma_n)$:
\#\label{eq:wbarg2}
& \overbar{\cG}_0  = \bigl\{ \btheta = (\zero,\sigma_0^2 ) \bigr\} \subseteq \cG_0 ,~~\text{and}~~\overbar{\cG}_1 (s, \gamma_n) = \bigl\{ \btheta = ( \beta \cdot \vb, \sigma^2 )\colon \vb \in \cG(s) \bigr\} \subseteq \cG_1 (s, \gamma_n),
\#
where  $\cG( s) =  \{ \vb \in \{ -1, 0,1\}^d \colon \| \vb \|_0  = s  \}$, $\sigma > 0$ is a constant, and $s \beta^2 /\sigma^2 = \gamma_n$. Note that the variance of the noise term  $\epsilon$ in \eqref{eq::regression_model} is unknown. Here we set $\sigma_0^2 = \sigma^2 + s\beta^2 $ to ensure that the marginal distribution of $Y$ is the same under both the null and alternative hypotheses. 
This model subclass captures  the most challenging setting of detecting mixture of regressions in terms of statistical error and computational complexity. 

\subsection{Upper Bounds}\label{sec::upper_bound_reg}

In this section, we introduce hypothesis tests for detecting mixture of regressions under the~statistical query model. Specifically, for simplicity we primarily focus on the restricted testing problem~$H_0\colon \btheta\in \overbar{\cG}_0$ versus $H_1 \colon \btheta \in \overbar{\cG}_1 (s, \gamma_n)$, where  the parameter spaces are defined in \eqref{eq:wbarg2}. It is worth noting that the corresponding upper bounds can be extended from $\overbar{\cG}_0$ and $\overbar{\cG}_1 (s, \gamma_n)$ in \eqref{eq:wbarg2} to the more general parameter spaces ${\cG}_0$ and ${\cG}_1 (s, \gamma_n)$ in the same way as for Gaussian mixture model in Appendix \ref{ap:GMM}.

 For notational simplicity, we denote the distribution of $\bZ = (Y, \bX)$ by $\PP_0$ under the null~hypothesis  and by $\PP_{\vb}$ under the alternative hypothesis when $\bbeta = \beta \cdot \vb$ for some $\vb \in \cG(s)$. Note  that
 \#\label{eq::Second_Moment}
 \EE_{\PP_0} ( Y^2 \bX \bX^{\top}) = (\sigma^2 + s\beta^2) \cdot \Ib~~\text{and}~~ \EE_{\PP_{\vb}}  ( Y^2 \bX \bX^{\top})  = (\sigma^2 +   s \beta^2) \cdot \Ib + 2 \beta ^2 \vb \vb^\top.
\#
 Similar to  the hypothesis tests constructed  in \S \ref{sec::upper_bound} for Gaussian mixture model, we define test~functions  based on the second moments of $Y\bX$, as specified in \eqref{eq::Second_Moment}.
 Recall that the
statistical query model  in Definition \ref{def::oracle} only allow bounded queries.
We truncate  both  $Y$ and $\bX$  so as  to obtain valid hypothesis tests.

More specifically, to obtain a hypothesis test that attains the information-theoretic lower~bound in Proposition \ref{prop::info_lower_bound_reg},   for all $ \vb\in \cG(s)$, we consider the query function
 \#\label{eq::query_reg1}
q_{\vb}(Y, \bX) =  Y ^2\cdot \bigl[s^{-1}   (   \bX^\top \vb )^2 - 1\bigr] \cdot  \ind(|Y| \leq \sigma R)  \cdot \ind \bigl \{ | \vb^\top   \bX  | \leq  R \sqrt{s \log n} \bigr \} ,
\#
where $R>0$ is an absolute constant. Hence in this case we have $T = | \cG(s)| = 2^s {d \choose s}$.
By direct computation, we have
\$
\EE_{\PP_0} \{ Y^2 \cdot [ s^{-1}  ( \bX^\top \vb)^2  - 1]\}= 0\quad \text{and}\quad \EE_{\PP_{\vb}} \{  Y^2 \cdot [ s^{-1}  ( \bX^\top \vb)^2  - 1] \} = 2s  \beta^2
\$
for all $\vb \in \cG(s)$.
   As we will show in \S \ref{sec::truncation_level},   we can set the truncation level $R$ to be a sufficiently large absolute constant such that
\# \label{eq:goal_truncation_reg}
{   \sup_{\vb' \in \cG(s) } } \Bigl\{ \EE_{\PP_{\vb}} \bigl[q_{\vb'} (Y, \bX)\bigr]  -  \EE_{\PP_0} \bigl[q_{\vb'}( Y , \bX)\bigr] \Bigr\} \geq s\beta ^2,  
\#

 Similar to the test function  in \eqref{eq::test_fun1}, let $Z_{q_{\vb}}$ be the output of the~oracle for query function $q_{\vb}$~defined in \eqref{eq::query_reg1}, we define the test function as
\#\label{eq::test_fun_reg1}
\ind \Bigl[{ \sup_{\vb\in \cG( s)}} Z_{q_{ \vb} }\geq  C \sigma^2  \cdot \log n \cdot  \sqrt{[ s \log  (2d) + \log ( 1/\xi)] /n} \Bigr],
\#
where $C$ is an absolute constant.
Similar to the proof of Theorem \ref{thm::test1}, we can prove that this hypothesis test has risk no more than $2\xi$ given that
\#\label{eq:reg_signal1}
 s \beta ^2 / \sigma ^2 = \gamma_n  = \Omega\bigl\{ \log n \cdot \sqrt{   [s \log  (2d) + \log ( 1/\xi) ] / n }\bigr\}.
 \#
Thus, setting $\xi = 1/d$ in \eqref{eq::test_fun_reg1}, we conclude that the lower bound in Proposition \ref{prop::info_lower_bound_reg} is tight up to a term logarithmic in $n$.

Notice that the test function defined in \eqref{eq::test_fun_reg1} requires superpolynomial oracle complexity under the statistical query model. For the computationally tractable test, we consider query functions
 \#\label{eq::query_reg2}
  q_j (Y, \bX) =  Y ^2 \cdot   (X_j^2- 1 ) \cdot  \ind( | Y| \leq \sigma R) \cdot  \ind \bigl \{ | X_j |  \leq  R \sqrt{\log n} \bigr \},~\text{for~all~} j \in [d].
\#
In this case,
the oracle complexity is $T =d $.
Similar to \eqref{eq:goal_truncation_reg}, as will be shown in \S \ref{sec::truncation_level},
we can set  the truncation level~$R$~to be a sufficiently large absolute constant such that
\#\label{eq:goal_truncation_reg2}
{  \sup_{j \in [d]} } \Bigl\{ \EE_{\PP_{\vb}}\bigl[q_{j } (Y, \bX)\bigr]  -  \EE_{\PP_0} \bigl[q_{j}( Y , \bX)\bigr] \Bigr\} \geq \beta^2.
\#
Let $Z_{q_j}$ be the output of the oracle for query function  $q_j$ in \eqref{eq::query_reg2}.  Similar to the test in~\eqref{eq::test_fun2},  we define the computationally tractable test function as
\#\label{eq::test_fun_reg2}
\ind \Bigl[{ \max _{j\in [d]}} Z_{q_j} \geq C'\sigma^2   \cdot \log n \cdot  \sqrt{ \log ( d / \xi)/ n}  \Bigr],
\#
where  $C'$ is an absolute constant. Similar to the proof of Theorem \ref{thm::test2}, we can prove that~this  test has risk no more than $2\xi$ given $  \beta ^2  / \sigma^2  = \Omega[\log n \cdot \sqrt{\log ( d / \xi)/n}]$, which is equivalent to
\#\label{eq:reg_signal2}
s \beta^2 / \sigma ^2 = \gamma_n = \Omega\bigl[\log n \cdot \sqrt{ s^2 \log ( d / \xi) / n}\bigr].
\#
Thus, by setting $\xi = 1/d$ in \eqref{eq::test_fun_reg2}, we conclude that the    computational lower bound in Theorem \ref{thm::lower_reg} is tight up to logarithmic terms.

\subsection{Implication for Parameter Estimation}
For the mixture of regression model, our statistical-computational tradeoff  in the detection problem also implies  the computational barrier for tasks including parameter estimation, support recovery,~and clustering. \cite{wang2014high} tackle the estimation problem by proposing an EM algorithm,~which attains an estimator with statistical rate of the order $ \| \bbeta \|_2 ^{-1} \sqrt{s \log d /n}$. However, they assume~the signal strength   $\rho(\btheta) = \| \bbeta \|_2^2 / \sigma^2$ to be sufficiently large and the existence of a good initialization~for the algorithm. Here we prove that the $O( \| \bbeta \| _2^{-1} \sqrt{s \log d /n})$ rate of convergence is not achievable by computationally feasible algorithms when $\rho(\btheta) = o( \sqrt{s^2 /n})$.
\begin{theorem}\label{col::implications_reg}
For the sparse mixture of regression model in \eqref{eq::regression_model} with  $\| \bbeta\| _0 = s$, we assume that   $ \| \bbeta \|_2^2 /\sigma^2 =  \gamma_n  = o(\sqrt{s^2 /n})$. Then  for any constant  $\eta >0$  and any $\mA \in \cA(T)$ with $T = O(d^{\eta})$, there exists an oracle $r \in\cR[\xi,n,T, \eta(\cQ_{\mA}) ]$ and an absolute constant $C>0$  such that under the statistical query model,  for any estimator $\hat \bbeta $ of $\bbeta$ with polynomial oracle complexity, it holds that
\#\label{eq::argument_reg}
\overline \PP_{\btheta} \big(  \| \hat \bbeta - \bbeta \|_2 ^2>    s \gamma_n  /64 \bigr)\geq C.
\#
 \end{theorem}
\begin{proof}
See \S\ref{proof::col::implications_reg} for a detailed proof.
\end{proof}

Therefore in the regime where   the signal strength is $\gamma_n = o(\sqrt{s^2 /n})$, no computationally tractable algorithm under the statistical query model can yield an estimator with statistical rate of the order $O( \| \bbeta \| _2^{-1} \gamma_n)$. In addition, similar to Gaussian mixture model, implications for feature selection and clustering can also be established using the same  techniques in those in the proof of  Theorem \ref{col::implications}.

Similar phenomenon also arises in   estimating the phase retrieval model \citep{cai2015optimal}. More specifically, the mixture of regression model in \eqref{eq::regression_model} can be transformed into  the noisy phase model for phase retrieval \citep{chen2014convex}, i.e.,
  \#\label{eq::phase}
  \tilde{Y} = | \bX^\top \bbeta + \epsilon|,
  \#
  in which $\bbeta \in \RR^d$ denotes the parameter of interest, $\epsilon \sim N(0,\sigma^2)$ is the random noise, $\bX \sim N(0, \Ib)$~is the measurement vector, and $\tilde{Y}$ is the response.  To see this, letting $W = |Y|$ in \eqref{eq::regression_model}, we  have
  \$ 
  W = | \eta \cdot \bX^\top \bbeta + \epsilon |  = | \eta | \cdot | \bX^\top \bbeta + \eta \cdot \epsilon |  \stackrel{D}{=} | \bX^\top \bbeta + \epsilon | = \tilde{Y}.
  \$
  Here the last equation indicates that $W$ and $|\bX^\top \bbeta + \epsilon|$ have the same distribution, which follows from the symmetry of the Gaussian noise $\epsilon$. Hence, we obtain  the noisy phase model from    mixture of regressions. This  implies that if an algorithm   solves the noisy phase   model in \eqref{eq::phase}, the same algorithm can be used to solve the mixture of regression model. Therefore, our lower bounds in \S \ref{sec::lower_bound_reg} also hold for the noisy phase model.

   In high dimensional settings with $\bbeta$ $s$-sparse, based on a slightly different noise~model,  $\tilde{Y} = | \bX^\top \bbeta  |^2 + \epsilon$,  \cite{cai2015optimal} establish the $O( \| \bbeta \| _2^{-1} \sqrt{s \log d /n})$ rate of convergence for a computationally tractable estimator.~Their results achieve the information-theoretic lower bound for parameter estimation under the assumption that $n \geq C (1 + \sigma/ \| \bbeta \|_2^2)^2 \cdot s^2 \log d $, in which $C$ is a sufficiently large constant. In comparison, without such an assumption, the information-theoretic lower bound can only be attained by a computationally intractable estimator based on empirical risk minimization~\citep{lecue2013minimax}. Therefore, it is conjectured by \cite{cai2015optimal} that~such an assumption~on sample complexity is necessary for any computationally efficient algorithm. Our results on statistical-computational tradeoffs confirm this conjecture under the statistical query model for the noisy phase model in \eqref{eq::phase}.

\section{Proofs of the Main Results} \label{sec::proof}
In this section, we lay out the proofs of the theoretical results in \S\ref{sec::theory_GMM} and \S\ref{sec::theory_reg}.

\subsection{Proofs of Lower Bounds}\label{pf::lower}
In the sequel,  we first prove the computationally feasible minimax lower bounds  for Gaussian mixture detection
 as well as their implications. Then we present the proofs of the lower bounds for detecting mixture of regressions.

\subsubsection{Proof of Theorem \ref{thm::lower_known_cov}}\label{proof::thm::lower_known_cov}
Now we prove the computationally feasible minimax lower bound for  detecting Gaussian mixture models when $\bSigma$ is known. Our proof is based on the $\chi^2$-divergence  between the null and alternative distributions.
\begin{proof}
In this proof, we consider a specific instance of the sparse mixture detection problem in \eqref{eq::testing_w}, namely
\$
H_0 \colon   \btheta = ( {\bf 0} , {\bf 0} , \Ib)~~\text{versus}~~H_1 \colon \btheta = \bigl[ -\beta(1- \nu) \vb, \beta\nu  \vb , \Ib \bigr], \$
where  $\vb \in \cG(s)= \{ \vb \in \{ -1, 0,1\}^d \colon \| \vb \|_0  = s  \}$.
  In this case, under the alternative hypothesis, the signal strength   is given by $\rho(\btheta)  = \| \beta \vb \|_2^2 = s\beta^2$.~We focus on  the setting where $\beta = o(n^{-1/4})$,~which implies that $\rho(\btheta) = o ( \sqrt{s^2 /n} )$.

  For notational simplicity, let $\PP_0$ denote   the probability distribution under the null hypothesis  and  let $\PP_{\vb}$ be the probability distribution under the alternative hypothesis with $\btheta  =  [ -\beta(1- \nu) \vb, \beta\nu \vb , \Ib ] $.
Moreover, we define $\overline{\PP}_0$ as the distribution of the random variables returned by the  oracle  when the true model is $\PP_0$ and define $\overline{\PP}_{\vb}$  correspondingly.
The minimax testing risk $\overline {R}_n^* ( \cG_0, \cG_1; \mA, r)$ defined in \eqref{eq::minimax_risk_oracle} with $\cG_0$ and $\cG_1$ given in \eqref{eq:wg0} and \eqref{eq:wg1} satisfies
\$
\sup_{\bSigma}\overline{R}_n^* \bigl[ \cG(\bSigma), \cG_1(\bSigma, s, \gamma_n); \mA , r\bigr]  \geq  \inf_{\phi\in \cH(\mA,r)} \Bigl[ \overline{\PP}_{0}(\phi = 1) + \sup _{\vb\in \cG(s)} \overline{\PP}_{\vb}(\phi = 0) \Bigr].\$

The next lemma establishes a sufficient condition that any hypothesis test  under  the statistical query model   is asymptotically powerless.    This lemma is in the same flavor as Theorem 4.2 in \cite{wang2015sharp}.

\begin{lemma} \label{lemma::distinguish}
For any algorithm $\mA \in \cA(T)$ and any query function $q \in \cQ_{\mA}$, we define the hypotheses that can be distinguished by $q$ as
\#\label{eq::distinguished_set}
\cC(q) = \Bigl\{ \vb \in \cG (s) \colon \bigl| \EE_{\PP_{0}} \bigl[q(\bX)\bigr] -  \EE_{\PP_{\vb}} \bigl[q(\bX)\bigr] \bigr| \geq     \tau_{q, \vb}   \Bigr\},
\#
where $\tau_{q, 0}$ and $\tau_{q, \vb}$ are  the tolerance parameters defined in \eqref{eq::query_2} under distributions $\PP_0$ and  $\PP_{\vb}$, respectively.
%
 Then, if
 \#\label{eq:capapcity_bound}
 T \cdot  \sup _{q\in \cQ_{\mA}} | \cC(q) | < | \cG(s)|,
 \#  there exists an oracle $r\in \cR[\xi,n,T, M,  \eta(\cQ_{\mA}) ]$ such~that
\$
\inf_{\phi\in \cH(\mA,r)} \Bigl[ \overline{\PP}_{0}(\phi = 1) + \sup _{ \vb \in \cG (s)} \overline{\PP}_{\vb}(\phi = 0) \Bigr]  = 1.
\$
\end{lemma}

\begin{proof}
See \S \ref{proof::lemma::distinguish} for a detailed proof.
\end{proof}

To apply Lemma \ref{lemma::distinguish}, we need to upper bound $\sup_{q \in \cQ_{\mA}} |\cC(q)|$.  We first  decompose  $\cC(q)$ into~two disjoint subsets $\cC_1(q)$ and $\cC_2(q)$, which are defined by
\# \label{eq::def_cq}
\cC_1(q) &=   \Bigl\{ \vb \in \cG (s): \EE_{\PP_\vb}\bigl[q(\bX)\bigr] - \EE_{\PP_0}\bigl[q(\bX)\bigr] >    \tau_{q, \vb}    \Bigr\} \quad \text{and}\quad \cC_2(q)  = \cC(q) \setminus \cC_1(q),
\#
Then, by definition,  it holds that
$${  \sup_{q \in \cQ_{\mA}} } |\cC(q)|\leq {  \sup_{q \in \cQ_{\mA}} } |\cC_1(q)| + {  \sup_{q \in \cQ_{\mA}} }|\cC_2(q)|.$$
For notational simplicity, we define
\#\label{eq::define_two_mixtures}
\PP_{ \cC_1(q)} =  \frac{{  { \sum_{ \vb \in \cC_1(q)}  }  }\PP_{\vb}}{|\cC_1(q)|}~~\text{and}~~\PP_{ \cC_2(q)} =  \frac{{  \sum_{\vb \in \cC_2(q)}}   \PP_{\vb}}{|\cC_2(q)|}
\#
as  the uniform mixture of $  \{ \PP_{\vb}\colon \vb \in \cC_1(q)  \}$ and  $ \{ \PP_{\vb} \colon \vb \in \cC_2(q)  \}$, respectively.
For $\ell \in \{1,2\}$, by the definition of  $\chi^2$-divergence we have
\#\label{eq::chi_square_divergence}
D_{\chi^2} (\PP_{ \cC_{\ell}(q)}, \PP_0 ) &= \EE_{\PP_0} \biggl \{\biggl [\frac{\ud \PP_{ \cC_{\ell}(q)}}{\ud \PP_0}(\bX) -1 \biggr]^2\biggr\} =\frac{1}{|\cC_{\ell}(q)|^2}  {\sum_{\vb , \vb' \in \cC_{\ell}(q)}} \EE_{\PP_0}   \biggl [\frac{\ud \PP_{\vb }}{\ud \PP_0} \frac{\ud \PP_{\vb'}}{\ud \PP_0}(\bX) \biggr] - 1\notag\\
&\leq \sup_{\vb  \in \cC_{\ell}(q)} \frac{1}{|\cC_{\ell}(q)|}  {\sum_{\vb' \in \cC_{\ell}(q)} }  \EE_{\PP_0}  \biggl[\frac{\ud \PP_{\vb }}{\ud \PP_0} \frac{\ud \PP_{\vb'}}{\ud \PP_0}(\bX) \biggr]- 1 \notag \\
&\leq \sup_{\vb\in \cC_{\ell}(q) }  \frac{1}{|\cC_{\ell}(q)|} {\sum_{\vb'\in \overline{\cC }_{\ell} (q, \vb ) } } \EE_{\PP_0} \biggl [\frac{\ud \PP_{\vb }}{\ud \PP_0} \frac{\ud \PP_{\vb'}}{\ud \PP_0}(\bX)\biggr ]- 1 ,
\#
where we define
\# \label{def::bar_C}
\overline{ \cC}_{\ell} (q, \vb)   = \argmax_{ {\cC}}  \biggl \{  \frac{1}{|  { \cC}|}  {  \sum_{\vb'\in {\cC }  } } \EE_{\PP_0}  \biggl[\frac{\ud \PP_{\vb}}{\ud \PP_0} \frac{\ud \PP_{\vb'}} {\ud \PP_0}(\bX) \biggr]-1~ \bigg \vert  ~|  {\cC} | = | \cC_\ell(q) |\biggr\}\subseteq \cG(s)  \#
 for $\ell\in\{0,1\}$. Here the maximization is taken over $\cG(s)$. The following lemma  gives an explicit characterization of   the last term in  \eqref{eq::chi_square_divergence}.
\begin{lemma}\label{lemma::compute_chi_square}
For any $\beta >0$ and  any $\vb_1, \vb_2 \subseteq \cG(s)$, we have
\$
\EE_{\PP_0} \biggl[ \frac{\ud \PP_{\vb_1}}{\ud \PP_{0}} \frac{\ud \PP_{\vb_2}}{\ud \PP_{0}} (\bX) \biggr] = \EE_{U} \bigl[ \cosh   ( \beta^2  U    \la \vb_1, \vb_2 \ra  ) \bigr].
\$
Here $U$ is a discrete random variable taking values in $\{ (1- \nu)^2, -\nu(1- \nu), \nu^2\}$, which satisfies
 \$
 \PP[U = (1-\nu)^2] = \nu^2,~~\PP[ U = -\nu(1-\nu) ] = 2\nu(1-\nu),~~\PP(U = \nu^2) = (1- \nu)^2.
 \$
\end{lemma}
 \begin{proof}
See \S \ref{proof::lemma::compute_chi_square} for a detailed proof.
\end{proof}

For notational simplicity, we define
\$
h(t) = \EE_U  \Bigl\{ \cosh \bigl[(s-t) \beta^2  U \bigr]\Bigr\}
\$ for any $t \in \{0,\ldots, s\}$. To establish an upper bound for   the last term in \eqref{eq::chi_square_divergence},  for any $j \in\{0,\ldots, s\}$ and any fixed  $\vb\in \cG(s)$, we define
\$
\cC_j(\vb) = \bigl\{\vb' \in \cG(s) : |\la \vb , \vb' \ra| = s-j \bigr\}.
\$

Since $h(t)$  is monotone decreasing  with $h(t) \geq h(s) = 1$, for any $\ell \in \{ 1,2 \}$, any  query function~$q \in \cQ_{\mA}$, and any~$\vb \in \cC_{\ell} (q)$, by Lemma \ref{lemma::compute_chi_square} and  the definition of~$\overline{\cC}_{\ell}(q, \vb)$ in~\eqref{def::bar_C},  there exists  an integer~$k_{\ell} (q, \vb )$  that satisfies
 \$
  \overline{\cC}_{\ell}(q, \vb) = \cC_{0}(\vb) \cup \cC_1(\vb) \cup \cdots \cup  \cC_{k_{\ell} (q, \vb) -1}(\vb) \cup \cC'_{\ell} (q,\vb ).
 \$
 Here $\cC'_{\ell} (q,\vb) = \overline \cC_{\ell}(q, \vb ) \setminus { \bigcup_{j=0}^{k_{\ell}(q, \vb) - 1}} \cC_j(\vb)$, which  has cardinality
 \$
 |\cC'_{\ell} (q,\vb)| = | \cC_{\ell}(q)| -{   \sum_{  j = 0}^{ k_{\ell}(q,\vb) -1}  } | \cC_{j} (\vb)| < | \cC_{k_{\ell} (q, \vb )} (\vb)|.
 \$
 Therefore, we can sandwich the cardinality of $\overline{\cC} _{\ell} (q , \vb)$ by
  \#\label{eq::set_bound}
{   \sum_{j=0}^{k_{\ell}(q, \vb ) } } |\cC_j(\vb) | > | \overline{ \cC} _{\ell} (q , \vb)|\geq {  \sum_{j=0}^{k_{\ell} (q, \vb )-1}} |\cC_j (\vb) |.
 \#
 Then by  \eqref{eq::chi_square_divergence} and Lemma \ref{lemma::compute_chi_square}, we further have
 \#\label{eq::upper_bound_chi_square1}
 1 + D_{\chi^2} (\PP_{  \cC_{\ell}(q)}, \PP_0   )\leq \frac{\sum_{j=0}^ {k_{\ell}(q, \vb )-1} h(j)\cdot |\cC_j(\vb)| + h\bigl[k_{\ell}(q, \vb )\bigr] \cdot | \cC_{\ell}'(q, \vb ) |}  { \sum _{j=0}^{k_{\ell}(q, \vb)-1} | \cC_j(\vb)| + | \cC_{\ell}'(q, \vb)|},~\text{for all}~\vb \in \cC_{\ell} (q).
 \#
Then by \eqref{eq::upper_bound_chi_square1} and~the   monotonicity of $h(t)$  we obtain that
 \#\label{eq::upper_bound_chi_square}
 1 + D_{\chi^2} (\PP_{ \cC_{\ell}(q)}, \PP_0 )  \leq \frac{\sum_{j=0}^ {k_{\ell}(q , \vb )-1} h(j) \cdot    | \cC_j(\vb)| }  { \sum _{j=0}^{k_{\ell}(q, \vb )-1} | \cC_j(\vb)| }.
 \#
 According to the symmetry of $\cG(s)$, the cardinality of $\cC_j(\vb)$ does not depend on the choice of $\vb$.  To further upper bound the right-hand side of \eqref{eq::upper_bound_chi_square}, we establish the following lemma to characterize the growth of $|\cC_j(\vb)|$.
\begin{lemma}\label{lemma::growth_Cj}
For any $\vb \in \cG(s)$, let $\cC_j(\vb) = \{\vb' \in \cG(s) : |\la \vb , \vb' \ra| = s-j\}$. Then we have
\#\label{eq::compare_geometric}
	|\cC_{j+1}(\vb)|/|\cC_{j}(\vb)| &\geq d/(2{ s }^2 ), ~~\text{for~all~} j\in \{0, \ldots, s-1\}.
	\#
\end{lemma}
\begin{proof}
See \S \ref{proof::lemma::growth_Cj} for a detailed proof.
\end{proof}
	We define $\zeta = d/ (2s^2)$.
By Lemma \ref{lemma::growth_Cj} we have $| \cC_{j}(\vb)|  \leq \zeta^{j-s}| \cC_{s}(\vb)|$ for $  j \in \{0,\ldots, s\}$.
	By the definition of $k_{\ell}(q, \vb )$ in \eqref{eq::set_bound}, for any  $q \in \cQ_{\mA}$, we further obtain
		\#\label{eq::upper_bound_Cq}
	 |\cC_{\ell}(q)| &\leq { \sum_{j=0}^{k_{\ell}(q, \vb ) }} |\cC_j(\vb)|
	 \leq   |\cC_{ s }(\vb)|  { \sum_{j=0}^{k_{\ell}(q, \vb ) }} \zeta^{j - s} \notag \\
	 &\leq \frac{\zeta^{-[ s  - k_{\ell}(q, \vb ) ]} |\cG(s) |}{1 - \zeta^{-1}} \leq 2\zeta^{-[ s  - k_{\ell}(q, \vb ) ] } | \cG(s)| ,
	 \#
	 where the last inequality follows from the fact that     $2s^2/d  = \zeta^{-1} = o(1)$.

	  Moreover,  let  $k \in \{ 0, \ldots, s \}$  be an integer. For two positive sequences $\{ a_i \}_{i=0}^k $ and $\{ b_i \}_{i=0}^k$, which satisfy $a_{i}/ a_{i-1} \geq b_{i } / b_{i-1} >1$ for all $i \in [k]$, the monotonicity of $h(t)$ implies that
	 \#\label{eq::useless}
	 { \sum_{0\leq i< j \leq k} }(a_i b_j - a_j b_i ) \cdot [ h(i) - h(j) ] \leq 0.
	 \#
	 Furthermore,  by expanding and simplifying the terms in \eqref{eq::useless} we have
	 \#\label{eq::useless2}
	 \frac{{\sum_{i=0}^k} a_i \cdot h(i)}{{  \sum_{i=0}^k} a_i} \leq \frac{{ \sum_{i = 0}^k} b_i \cdot  h(i)}{{ \sum_{i=0}^k} b_i}.
	 \#

	 In the sequel, we establish an upper bound on $k_{\ell}(q,\vb)$ for $\ell \in \{1,2\}$ and $\vb \in \cC_{\ell}(q)$. For notational simplicity, we denote $k_{\ell} = k_{\ell}(q, \vb)$.
 First, by combining \eqref{eq::upper_bound_chi_square}, \eqref{eq::compare_geometric}, and \eqref{eq::useless2} with   $a_j = | \cC_j(\vb)|$~and $b_j = \zeta^j$ we  have
  \#\label{eq::compute_expectation}
	1 + D_{\chi^2} (\PP_{ \cC_{\ell}(q)}, \PP_0 )&\leq \frac{ \sum_{j=0}^{k_{\ell}-1} \zeta^j \EE_{U}\Bigl\{ \cosh\bigl[( s  - j) \beta^2 U \bigr]\Bigr\}}{\sum_{j=0}^{k_{\ell}-1} \zeta^j} \notag\\
	&\leq  \frac{ \EE_{U} \Bigl\{  \cosh\bigl[(s - k_{\ell} + 1     ) \beta^2 U \bigr]  \Bigr\}(1-\zeta^{-1})}{1 - \zeta^{-1}\cosh({\beta}^2)}.
	\#
	Here we use the fact that $\cosh(\beta^2) / \zeta = o(1)$, which holds because
	  $ \beta = o(n^{-1/4})$ and  $s^2 /d =o(1)$. In addition, we employ the  following lemma to establish a lower bound for $D_{\chi^2} (\PP_{ \cC_{\ell}(q)}, \PP_0 )$. Combining with the upper bound in \eqref{eq::compute_expectation}, we obtain an upper bound on $k_{\ell}$.
	
 \begin{lemma}\label{lem:1}\label{lemma::bound_chi_square_div}
For any query function $q$ and $\ell\in \{1,2\}$, we have
\$
D_{\chi^2} (\PP_{  \cC_\ell(q)}, \PP_0 ) \geq \frac{2 \log (T/\xi)}{3 n}.
\$
\end{lemma}
 \begin{proof}
See \S \ref{proof::bound_chi_square_div} for a detailed proof.
\end{proof}

We remark that Lemmas \ref{lemma::distinguish} and \ref{lemma::bound_chi_square_div} are closely related to  the Le Cam's method \citep{le2012asymptotic} in   the classical minimax framework.  The main idea of Le Cam's method is that, any hypothesis test incurs a large risk if the divergence between the null and alternative distributions is small.~In~detail, as shown in \cite{arias2017detection}, the proofs of Propositions \ref{prop::info_lower_bound} and \ref{prop::info_lower_bound2} are based upon the $\chi^2$-divergence between $\PP_{0}$ and the uniform mixture of $\{ \PP_{\vb} \colon \vb \in \cG( s)\}$. In comparison, our proof relies on the $\chi^2$-divergence between $\PP_0$ and $\PP_{\cC_{\ell}(q)}$, where $\PP_{\cC_{\ell}(q)}$ denotes the uniform mixture of the distributions in
\$
\bigl\{\PP_{\vb} : \supp(\vb) \in \cC_{\ell}(q)\bigr\} \subseteq \bigl\{ \PP_{\vb} \colon \vb \in \cG(  s)\bigr\},
\$
which leverages the local structure of the family of alternative distributions. Therefore, our analysis of the computationally feasible minimax  lower bound can be viewed as a localized refinement of the classical Le Cam's method.

For notational simplicity, we denote $\sqrt{2 \log (T/\xi)/ (3n)}$ by $\tau$  hereafter.~Combining \eqref{eq::compute_expectation}, Lemma \ref{lemma::bound_chi_square_div} and inequality $\cosh(x) \leq \exp(x^2/2)$, we obtain
	\#\label{eqn::second_bound}
	 (s  - k_{\ell} + 1 )^2 \geq \frac{2 \log (1+\tau^2 )}{{\beta}^4} - 2\log \biggl[ \frac{1-\zeta^{-1}}{{1 - \zeta^{-1}\cosh({\beta}^2)}} \biggr]\bigg/{\beta}^4.
	\#
	Moreover, by Taylor expansion and the fact that $\cosh(\beta^2) /\zeta = o(1)$, we obtain
	\#\label{eq::bound_by_taylor1}
	\log \biggl[\frac{1-\zeta^{-1}}{1 - \zeta^{-1}\cosh({\beta}^2)} \biggr] = \log \biggl \{ 1 + \frac{\zeta^{-1} \bigl[ \cosh(\beta^2) -1 \bigr]   }{1-  \zeta^{-1} \cosh(\beta^2) } \biggr\} = O( \zeta^{-1} \beta^4  ).
	\#
	Since $\beta = o(n^{-1/4})$, we have $\zeta ^{-1} \beta ^4 = o( \zeta^{-1} n^{-1})$. In addition, we   have  $\log (1+ \tau ^2) \geq \tau^2 /2  \geq 1/n $ by inequality $\log(1 + x) \geq x/2$. Thus combining \eqref{eqn::second_bound} and \eqref{eq::bound_by_taylor1},    the right hand side of \eqref{eqn::second_bound} is dominated by the first term.  Hence we obtain $
	 (s  - k_{\ell} + 1 )^2 \geq  \log (1+\tau^2 ) / {\beta}^4,
$
which implies that
	\#\label{eq::bound_integer_k_21}
	  k_{\ell} (q, \vb) \leq  s + 1 - \sqrt{    \log (1+\tau^2 ) /{\beta}^4 }, ~ \text{for all}~ \ell \in \{ 1,2\}.
	\#
	Moreover, inequality \eqref{eq::bound_integer_k_21} holds for all $q \in \cQ_{\mA}$ and all $\vb \in \cC_{\ell}(q)$.
	After obtaining upper bounds~for $k_1$ and $k_2$,
	 combining   \eqref{eq::set_bound}, \eqref{eq::upper_bound_Cq},  and \eqref{eq::bound_integer_k_21}, we further have
\#\label{eq::bound_Main_term}
 \frac{T \cdot    \sup  _{q\in \cQ_{\mA}} | \cC(q)|}{|\cG(s)|}
 \leq 4T \cdot   \exp \Bigl\{ -\log \zeta\cdot \bigl[\sqrt{   \log (1+\tau^2 ) / {\beta}^4  }- 1\bigr] \Bigl\}.
 \#
Recall that we denote $\tau = \sqrt{\log (T/\xi) / n} $  where $\xi = o(1)$.~For any constant $\eta >0$, we set $T = O(d^\eta)$.
Also, under the assumption of the theorem, there exists a  sufficiently small  constant $\delta >0$ such that $s ^2 / d^{1- \delta } = O(1)$. Hence, we have $\zeta = d/ (2s^2) =  \Omega(d^{\delta} )$. By inequality $\log(1+x) \geq x/2$, it holds that
	$\log (1+ \tau^2) \geq \tau^2/2 = \log (T/\xi) / (3n) $. Under the condition  $\beta ^4 n = o(1)$,~we have
	\$
	\frac{\log (T/\xi)}{3n\beta^4} > \frac{\eta \log d}{3n \beta^4}\rightarrow \infty.
	\$
	Hence if $n$ is sufficiently large, we have
	\$
	\frac{\log (T/\xi)}{3n\beta^4} > C^2,
	\$
	where the absolute constant $C$ satisfies $\delta (C-1) > \eta$.~Then by  \eqref{eq::bound_Main_term} we have
	\#\label{eq::bound_key_quantity_1}
	& \frac{T \cdot \sup_{q \in \cQ_{\mA}} |\cC(q)|}{|\cG(s)|} \leq 4T \cdot   \exp \Bigl\{ -\log \zeta\cdot \bigl[\sqrt{   \log (1+\tau^2 ) / {\beta}^4  }- 1\bigr] \Bigl\} \notag\\
	&\quad  =  O\bigl[4 d^\eta \zeta^{-(C-1) }  \bigr]=   O\bigl[4d^{\eta-\delta (C-1)} \bigr] = o(1).
	\#
	By combining \eqref{eq::bound_key_quantity_1} and Lemma \ref{lemma::distinguish}, we conclude that $\overline {R}_n^* ( \cG_0, \cG_1; \mA, r)$ converges to  $ 1$ as $n$ goes to infinity if $\gamma_n = o ( \sqrt{s^2/n})$.~ This concludes the proof of Theorem \ref{thm::lower_known_cov}.
\end{proof}

\subsubsection{Proof of Theorem \ref{thm::lower_unknown_cov}}\label{proof::thm::lower_unknown_cov}
Now we prove Theorem \ref{thm::lower_unknown_cov}, the computational lower bound for Gaussian mixture detection when $\bSigma$ is unknown.
\begin{proof}
The proof is similar to that of Theorem \ref{thm::lower_known_cov}. To characterize the fundamental difficulty of~the testing problem, we  consider the following specific instance
\$
H_0 \colon   \btheta = ( {\bf 0} , {\bf 0} , \Ib)~~\text{versus}~~H_1 \colon \btheta  = (  - \beta \vb ,   \beta \vb  , \bSigma_1 ),
\$
where $\vb \in \cG( s) =   \{ \vb \in \{ - 1,0, 1\}^d, \| \vb \|_0 = s \}$ and $\bSigma_1 = \Ib - \beta^2 \vb \vb^\top$.
In this case we have~$\Delta \bmu = 2\beta \vb$, and the signal strength  for parameter  of  alternative distribution  is
 \$
 \rho'(\btheta) = \frac{\| \Delta \bmu \|_2^4}{\Delta \bmu^\top \bSigma_1 \Delta \bmu} = \Delta \bmu ^\top \bSigma _1^{-1} \Delta \bmu  = \frac{4 s\beta^2}{1- s\beta^2},
 \$
where the second equality follows from the Woodbury matrix identity.
Provided the  assumption that  $\rho'(\btheta) = o[ (s^3 /n)^{1/4}]$,
 we have $ \beta^8 s  n  = o(1)$. For notational simplicity,  let $\PP_0$ be the distribution of $\bX$ under the null and $\PP_{\vb}$ be the distribution of $\bX$ under the alternative with model parameters $\btheta = (  - \beta \vb ,  \beta  \vb , \bSigma_1 )$. Due to the similar structure of the problem, we define quantities  $\cG(s)$, $\overline{ \PP} _{\vb}$  and $\PP_{\vb}$ in the same way as in the proof of Theorem \ref{thm::lower_known_cov}.
Then we have~that  the minimax testing   risk $\overline {R}_n^* [ \cG_0, \cG_1( s, \gamma_n); \mA, r]$ defined in \eqref{eq::minimax_risk_oracle} is lower bounded by
\$\overline {R}_n^* \bigl[  \cG_0, \cG_1( s, \gamma_n) ; \mA, r \bigr]  &\geq    \inf_{\phi\in \cH(\mA,r)}  \Bigl[ \overline{\PP}_{0}(\phi = 1) +  \sup _{\vb\in \cG( s) } \overline{\PP}_{\vb}(\phi = 0) \Bigr].\$
  By  Lemma \ref{lemma::distinguish}, to show that any hypothesis test with polynomial  oracle complexity is asymptotically powerless, it remains to show that $ T \cdot  \sup_{q \in \cQ_{\mA}} | \cC(q)| / | \cG(s) | = o(1)$.~For any query function~$q\in \cQ_{\mA}$,~let $\cC_1(q)$ and $\cC_2(q)$ be defined as in \eqref{eq::def_cq}, and  $\PP_{\cC_1(q)}$ and $\PP_{\cC_2(q)}$ be defined as in \eqref{eq::define_two_mixtures}. Besides, we~define $k_{\ell}(q, \vb)$ for $\ell \in \{1,2\}$ and any $\vb \in \cC_{\ell}(q)$  in the same fashion as in \eqref{eq::set_bound}.
  The following lemma, as a counterpart of Lemma \ref{lemma::compute_chi_square}, characterizes the cross moment of the likelihood ratios in \eqref{eq::chi_square_divergence}.

\begin{lemma}\label{lemma::compute_h_2}
For any integer $s>0$ and any $\vb_1, \vb_2 \in \cG(s)$,     we have
\#\label{eq::h2_formula}
\EE_{\PP_0} \biggl[ \frac{\ud  \PP_{\vb_1 }}{\ud  \PP_{0}} \frac{\ud\PP_{\vb_2}}{\ud \PP_{0}} (\bX) \biggr] = \EE_{W}  \biggl[(1-\beta^4 W^2)^{-1/2} \cdot\exp \biggl (\frac{-\beta ^4 W^2}{1-\beta ^4 W^2} \biggr)\cdot  \cosh \biggl(\frac{\beta^2 W}{1-\beta^4 W^2}\biggr )\biggr ],
\#
where $W$ is the sum of $| \la \vb_1 ,  \vb_2\ra |$ independent Rademacher random variables.

\end{lemma}

\begin{proof}
See \S \ref{proof::lemma::compute_h_2} for a detailed proof.
\end{proof}

For any $\vb , \vb' \in \cG(s)$,
combining inequalities  \$ -  \log ( 1-x) \leq x /(1-x),~\text{for~all~} x \in [0,1),~~ \cosh(x) \leq \exp(x^2/2),~\text{for~all~} x \geq 0, \$ we obtain the following upper bound for    the right-hand side of  \eqref{eq::h2_formula}
\#\label{eq::what}
\EE_{\PP_0} \biggl[ \frac{\ud  \PP_{\vb}}{\ud  \PP_{0}} \frac{\ud\PP_{\vb'}}{\ud \PP_{0}} (\bX) \biggr] \leq \EE_{\overbar{T}} \biggl\{    \exp \biggl[ \frac{ \beta^8 \overbar{T}^4}{ 2( 1- \beta^4 \overbar{T}^2)^2}\biggr]  \biggr\} \leq \EE_{\overbar{T}} \biggl\{  \exp \biggl[ \frac{ \beta^8 \overbar{T}^4} { 2 ( 1- s^2 \beta^4)^2}\biggr] \biggr\} ,
\#
where $\overbar{T}$ is the sum of $| \la \vb, \vb' \ra| $ independent Rademacher random variables.~Since $s\beta^2 = o(1)$,  it~holds that $2 ( 1- s^2 \beta^4)^2 > 1$ when $n$ is sufficiently large. Thus, by \eqref{eq::what} we further have
\#\label{eq::what2}
\EE_{\PP_0} \biggl[ \frac{\ud  \PP_{\vb}}{\ud  \PP_{0}} \frac{\ud\PP_{\vb'}}{\ud \PP_{0}} (\bX) \biggr]  \leq  \EE_{\overbar{T}}  \bigl[ \exp (\beta^8 \overbar{T}^4) \bigr].
\#
Note that $\beta^8 \overbar{T}^4 = o(1)$.  By second-order Taylor expansion  we further have
\#\label{eq::what22}
 \EE_{\overbar{T}}  \bigl[\exp (\beta^8 \overbar{T}^4) \bigr] \leq  \EE_{\overbar{T}} \bigl[ 1 +  ( \beta^8 \overbar{T}^4) +  (  \beta^8 \overbar{T}^4)^2 \bigr].\#
Since $\overbar{T}$ is the sum of $| \la \vb, \vb' \ra|$ Rademacher random variables, a calculation of its moments yields
\#\label{eq::what4}
\EE_{\overbar{T}} \overbar{T}^4 &= 3| \la \vb, \vb' \ra|^2 - 2| \la \vb, \vb' \ra|,\notag\\
\EE_{\overbar{T}} \overbar{T}^8 &= 105 | \la \vb, \vb' \ra|^4 - 420 | \la \vb, \vb' \ra|^3 + 588 | \la \vb, \vb' \ra|^2 - 272 | \la \vb, \vb' \ra|.
\#
Combining \eqref{eq::what2}, \eqref{eq::what22}, and \eqref{eq::what4}, we conclude that there exists a constant $C_0$ such that
\#\label{eq::compute_h_fun}
\EE_{\PP_0} \biggl[ \frac{\ud  \PP_{\vb}}{\ud  \PP_{0}} \frac{\ud\PP_{\vb'}}{\ud \PP_{0}} (\bX) \biggr]  \leq   \exp ( C_0   \beta^8s  | \la \vb, \vb' \ra|   ) .
\#
Similar to the proof of Theorem \ref{thm::lower_known_cov}, we define $\zeta = d/ (2s^2)$.
Combining \eqref{eq::upper_bound_chi_square}, \eqref{eq::compare_geometric}, \eqref{eq::useless2}, and~\eqref{eq::compute_h_fun} with   $a_j = | \cC_j(\vb)|$, $b_j = \zeta^j$, and  $h(t) =\exp[C _0 \beta^8 s (s-t)] $, we obtain
	\#\label{eqn::compute_expectation2}
	1 + D_{\chi^2}(\PP_{  \cC_{\ell}(q) }, \PP_0)&\leq \frac{ \sum_{j=0}^{k_{\ell} -1} \zeta^j \exp\bigl[C_0  \beta^8 s (s-j)\bigr]}{\sum_{j=0}^{k_{\ell}  -1} \zeta^j} \notag\\
	&\leq  \frac{ \exp \bigl[ C_0  (s - k_{\ell}   + 1) \beta^8 s \bigr] \cdot (1-\zeta^{-1})}{1 - \zeta^{-1}\exp (C  \beta^8 s)} .
	\#
Here we use $\zeta^{-1}\exp (C_0 \beta^8 s) = o(1)$ and denote $ k_{\ell}(q, \vb)$ by $k_{\ell}$. For notational simplicity, we denote $\sqrt{2 \log (T/\xi)/(3n)}$ by $\tau$  hereafter.
Combining \eqref{eqn::compute_expectation2} and  Lemma \ref{lemma::bound_chi_square_div}
 we obtain that
	\#\label{eq:w909}
	 s  - k _{\ell}+ 1 \geq \frac{\log (1+\tau^2 )}{C_0  {\beta}^8 s} - \log \biggl[ \frac{1-\zeta^{-1}}{1 - \zeta^{-1}\exp (C_0    \beta^8 s)}\biggr]\bigg/( C_0   {\beta}^8 s) .
	\#
Note that by Taylor expansion we have
	\#\label{eq:w9999}
	& \log \biggl[\frac{1-\zeta^{-1}}{1 - \zeta^{-1}\exp(C_0   {\beta}^8  s)} \biggr] = \log \biggl\{1 + \frac{\bigl[\exp(C_0  {\beta}^8 s) - 1\bigr]\zeta^{-1}}{1 - \zeta^{-1}\exp(C_0  {\beta}^8 s )} \biggr\} \notag\\
	&\quad = O\biggl\{\frac{\bigl[\exp(C _0{\beta}^8 s) - 1\bigr]\zeta^{-1}}{1 - \zeta^{-1}\exp(C _0 {\beta}^8 s)}\biggr\} = O(\zeta^{-1}{\beta}^8 s),
	\#
	where we use the fact that $\zeta^{-1}\exp(C_0  {\beta}^8 s) = o(1)$.~Thus,  from \eqref{eq:w909} and  \eqref{eq:w9999}, we have that, when $n$ is sufficiently large,
	\#\label{eq:w7777}
	k_{\ell} \leq s + 2 - \frac{\log (1+\tau^2 )}{C_0  {\beta}^8 s}, ~\text{for~all~}\ell \in \{1,2\}.
	\#
	Now combining    \eqref{eq::set_bound}, \eqref{eq::upper_bound_Cq},  and \eqref{eq:w7777}, we obtain
\#\label{eqn::bound_Main_term11}
\frac{T \cdot  \sup_{q \in \cQ_{\mA}} |\cC(q)|}{|\cG(s)|}
 \leq 4T \cdot  \exp \biggl\{ -\log \zeta\cdot  \biggl[\frac{\log (1+\tau^2 )}{C_0  {\beta}^8 s} - 2 \biggr] \biggr\}.
 \#
For any positive absolute constant $\eta$, we set $T = O(d^{\eta})$. In addition, under the assumption  that  there exists a  sufficiently small  constant $\delta >0$ such that $s ^2 / d^{1- \delta } = O(1)$, we have $\zeta = d/ (2s^2) =  \Omega(d^{\delta} )$.~By inequality $\log(1+x) \geq x/2$,  we have $
	\log (1+ \tau^2) \geq \tau^2/2 = \log (T/\xi) / (3n).$
	 Under the condition that  $\beta^8  s n = o(1)$,  it holds that
	  \$
	  \frac{\log (T/\xi)}{3C_0  n\beta^8 s}  \rightarrow \infty.
	  \$
	 Hence, for $n$ large enough, we have $\log (T/\xi) /( 3C_0   n\beta^8 s)  >C'$  for some sufficiently large  constant $C'$  satisfying $\delta (C'-2) > \eta$. Then by  \eqref{eqn::bound_Main_term11} we have
	\#\label{eq:w27365}
	& \frac{T \cdot   \sup_{q \in \cQ_{\mA}} |\cC(q)|}{|\cG(s)|} \leq 4T \cdot  \exp \biggl\{ -\log \zeta\cdot \biggl [\frac{\log (T/ \xi)}{3 C_0 {\beta}^8 sn} - 2 \biggr] \biggr\} \notag\\
	& \quad =  O\bigl[4 d^\eta \zeta^{-(C'-2) } \bigr]=   O\bigl[4d^{\eta-\delta (C'-2)} \bigr] = o(1).
	\#
	 Combining \eqref{eq:w27365}  and Lemma \ref{lemma::distinguish}, we obtain that $\overline {R}_n^* [ \cG_0, \cG_1( s, \gamma_n) ; \mA, r]\rightarrow 1 $ under~the assumption that  $\gamma_n = o [ (s^3/ n)^{1/4}  ]$. This concludes the proof of Theorem \ref{thm::lower_unknown_cov}.
\end{proof}

\subsubsection{Proof of Theorem \ref{col::implications}}\label{proof::col::implications}
In the sequel, we prove Theorem \ref{col::implications}, which shows that the lower bounds for the detection problem also hold for estimation, support recovery, and clustering.

\begin{proof}
We prove the three claims by contradiction. We show that  if any of the~arguments is false, we can construct an asymptotically powerful test for the detection problem, i.e., testing $H_0\colon \btheta\in   {\cG}_0(\bSigma)$ against $H_1 \colon \btheta \in    \cG_1(\bSigma,s, \gamma_n)$.  Recall that $\gamma_n = o(\sqrt{s^2/n})$. Then the existence of a computationally tractable test contradicts Theorem  \ref{thm::lower_known_cov}.

We first assume that \eqref{eq::argument1} does not hold. That is, suppose that there exists $\eta >0$ and $\mA \in \cA(T)$ with $T = O(d^{\eta})$ such that under the alternative hypothesis,  for any oracle $r \in \cR[\xi,n,T, M, \eta(\cQ_{\mA}) ]$, we  obtain estimators $\hat \bmu_1$ of $\bmu_1$ and $\hat\bmu_2$ of $\bmu_2$ satisfying
\$
\overline\PP_{\btheta} \Bigl[   \max_{\ell \in \{1,2\}} ( \hat\bmu_{\ell} - \bmu_{\ell} ) ^\top \bSigma ^{-1} ( \hat \bmu_{\ell} - \bmu_{\ell})>   \gamma_n / 64  \Bigr]=  o(1),
\$
or equivalently,
\#\label{eq::basedon_estimation1}
\max_{\ell \in \{1,2\}} ( \hat\bmu_{\ell} - \bmu_{\ell} ) ^\top \bSigma ^{-1} ( \hat \bmu_{\ell} - \bmu_{\ell}) \leq  \gamma_n / 64
\#
 with probability tending to one. Recall that the signal strength is
 $\Delta \bmu ^\top \bSigma^{-1} \Delta \bmu = \gamma_n$. Based on \eqref{eq::basedon_estimation1},  the test function for the sparse mixture detection problem can be defined as
 \#\label{eq::first_test}
 \phi_1(\{z_t\}_{t=1}^T) = \ind ( \Delta \hat \bmu^{\top} \bSigma^{-1} \Delta \hat \bmu  \geq \gamma_n /3).
 \#
By  inequality $(a+b)^2 \leq 2 a^2 + 2b^2 $ and  \eqref{eq::basedon_estimation1},  with high probability, we have
\#\label{eq::estimation2}
&(\Delta \hat  \bmu- \Delta \bmu)^{\top} \bSigma^{-1} (\Delta \hat \bmu - \Delta \bmu) \notag\\
& \quad \leq  2 \bigl[ ( \hat\bmu_1 - \bmu_1 ) ^\top \bSigma ^{-1} ( \hat \bmu_1 - \bmu_1)+  ( \hat\bmu_2- \bmu_2 ) ^\top \bSigma ^{-1} ( \hat \bmu_2 - \bmu_2) \bigr]
 \leq     \gamma_n  /16.
\#
 Also, by direct calculation we have
\#\label{eq::estimation3}
&(\Delta \hat  \bmu+  \Delta \bmu)^{\top} \bSigma^{-1} (\Delta \hat \bmu + \Delta \bmu) \\
 &\quad = (\Delta \hat  \bmu-   \Delta \bmu) ^\top \bSigma^{-1}(\Delta \hat \bmu - \Delta \bmu) + 4 (\Delta \hat  \bmu-   \Delta \bmu)^\top \bSigma^{-1} \Delta \bmu  + 4 \Delta \bmu^{\top } \bSigma^{-1} \Delta \bmu \notag\\
& \quad \leq  (\Delta \hat  \bmu-   \Delta \bmu)^\top \bSigma^{-1}(\Delta \hat \bmu - \Delta \bmu) + 4 \sqrt{\gamma_n} \cdot \bigl[ (\Delta \hat  \bmu-   \Delta \bmu)^\top \bSigma^{-1}(\Delta \hat \bmu - \Delta \bmu) \bigr]^{1/2} + 4 \gamma_n,\notag
\#
where the last inequality follows from Cauchy-Schwarz inequality. Thus, by \eqref{eq::estimation2} and  \eqref{eq::estimation3},  we have $(\Delta \hat  \bmu+  \Delta \bmu)^{\top} \bSigma^{-1} (\Delta \hat \bmu + \Delta \bmu)  \leq  6 \gamma_n$.
Furthermore,  combining this inequality with \eqref{eq::estimation2}, we obtain that
\#\label{eq::estimation4}
&| \Delta \hat \bmu ^{\top} \bSigma^{-1} \Delta \hat \bmu - \Delta \bmu ^{\top} \bSigma^{-1} \Delta \bmu   |^2  \notag\\
&\quad \leq  \bigl[ (\Delta \hat  \bmu- \Delta \bmu)^{\top} \bSigma^{-1} (\Delta \hat \bmu - \Delta \bmu) \bigr] \cdot \bigl[(\Delta \hat  \bmu+  \Delta \bmu)^{\top} \bSigma^{-1} (\Delta \hat \bmu + \Delta \bmu) \bigr] \leq  3/8 \cdot \gamma_n^2.
\#
Since $\Delta \bmu ^{\top} \bSigma^{-1} \Delta \bmu =   \gamma_n$, \eqref{eq::estimation4} implies that
\$
\Delta \hat \bmu ^{\top} \bSigma^{-1} \Delta \hat \bmu \geq (1- \sqrt{3/8}) \gamma_n \geq \gamma_n /3
\$ with high probability.

Furthermore, under the null hypothesis, we have $\bmu_1 = \bmu_2$ and $\Delta \bmu = {\bf 0}$. In this case, \eqref{eq::basedon_estimation1} still holds with probability tending to one, which implies that
\$
& \Delta \hat \bmu ^{\top} \bSigma^{-1} \Delta \hat \bmu = ( \hat \bmu_1 - \hat \bmu_2 ) ^\top \bSigma^{-1} ( \hat \bmu_1 - \hat \bmu_2 ) \\
 &\quad  =  \bigl [   ( \hat \bmu_1-  \bmu_1 )  -( \hat \bmu_2 - \bmu_2  )
 \bigr ]  ^\top \bSigma^{-1} \bigl [   ( \hat \bmu_1-  \bmu_1 )  -( \hat \bmu_2 - \bmu_2  )
 \bigr ] \\
 &\quad \leq   2 \bigl[ ( \hat\bmu_1 - \bmu_1 ) ^\top \bSigma ^{-1} ( \hat \bmu_1 - \bmu_1)+  ( \hat\bmu_2- \bmu_2 ) ^\top \bSigma ^{-1} ( \hat \bmu_2 - \bmu_2) \bigr]
 \leq     \gamma_n  /16.
\$
Hence,  the test function defined in \eqref{eq::first_test}  is asymptotically powerful. However, since $\gamma_n = o(\sqrt{ s^2/n})$, this contradicts the computationally feasible minimax lower bound in Theorem \ref{thm::lower_known_cov}.

Second, for any $\btheta \in \cG_1(\bSigma,s, \gamma_n)$, suppose that we have a polynomial-time algorithm that returns an index set $\hat \cS \subseteq[d]$ such that $\hat \cS = \supp(\Delta \bmu)$~with high probability   under $\overline \PP_{\btheta}$.~Furthermore, under the null hypothesis, we assume that this algorithm yields any one of the $2^d$ index subsets of $[d]$ with equal probability under $\overline \PP_{\btheta}$ for $\btheta\in \cG_0(\bSigma)$.~Then the test function for detecting Gaussian mixtures can be defined as
\#\label{eq::test2_selection}
\phi_2(\{z_t\}_{t=1}^T) = \ind \bigl[ \hat \cS  = \supp (\Delta \bmu) \bigr].
\#
Under  the alternative hypothesis, since $\hat \cS= \supp(\Delta \bmu)$ with high probability, by the definition in \eqref{eq::test2_selection} we have
\$
  \sup _{\btheta\in \cG_1(\bSigma,s,\gamma_n)} \overline{\PP}_{\btheta} (\phi_2 =0) =\overline\PP_{\btheta}  \bigl[  \hat \cS \neq \supp(\Delta \bmu)  \bigr]  = o(1).
\$
Moreover, under the null hypothesis,  we have $\Delta \bmu ={\bf 0} $ and thus $\supp(\Delta \bmu) = \emptyset$. Since $\hat \cS$ is uniformly random under $\overline {\PP}_{\btheta}$, we have
\$\sup_{\btheta \in \cG_0(\bSigma) }\overline{\PP}_{\btheta} (\phi_2 = 1) =   \overline {\PP} _{\btheta} \bigl ( \hat \cS =  \emptyset \bigr ) = 2^{-d}= o(1).\$
Therefore, the test function $\phi_2$ defined  in \eqref{eq::test2_selection} is asymptotically powerful, which
  is impossible~when $\gamma_n = o(\sqrt{ s^2/n})$ by  Theorem \ref{thm::lower_known_cov}. Hence, there exists an absolute constant $C$ such that \eqref{eq::argument2} holds.

Finally, to see \eqref{eq::argument3}, suppose that we obtain an assignment function $ F\colon \RR^d \rightarrow \{1,2\}$  such that clustering by $F$ is asymptotically accurate, that is,
\$
\min _{\Pi}  \PP _{\btheta}\Bigl\{ \Pi\bigl[F(\bX)\bigr]  \neq   F_{\btheta} (\bX) \Bigr\} = o(1),
\$
where $\Pi:\{1,2\}\rightarrow \{1,2\}$ is any permutation function. In addition, we define
\#\label{eq:two_events}
\cE_1 = \bigl\{ F(\bX)= F_{\btheta}(\bX) \bigr\}~~\text{and}~~\cE_2 = \bigl\{ F(\bX)= 3 - F_{\btheta}(\bX) \bigr\}.
\#
Let $\cE= \cE_1 \cup \cE_2$, which
is  the event that clustering by $F$ is accurate. Now we consider the problem of detecting the Gaussian mixture model, i.e.,
\#\label{eq:define_a_new_mixturemodel}
H_0 \colon \bX \sim N( \bmu, \bSigma ) \quad \text{versus} \quad H_1 \colon \bX\sim  \nu \cdot N(\bmu_1, \bSigma )+  ( 1-  \nu ) \cdot N(\bmu_1, \bSigma ),
\#
where we assume that $\bmu = \nu \cdot \bmu_1 + ( 1- \nu) \cdot \bmu_2$.     For ease of presentation, let $\Delta \bmu = \bmu_1 - \bmu_2$ and let $\eta \in \{ 1, 2\}$ be the latent variable of the Gaussian mixture model in \eqref{eq:define_a_new_mixturemodel} under $H_1$.
By the definition of $F_{\btheta} $ in \eqref{eq:optimal_cluster}, it can be verified that, under $H_1$,  we have
\$
& \PP( \eta = 1 \given \bX= \xb ) =\frac{  \PP(\eta = 1, \bX = \xb) }{ \PP(  \bX = \xb)   } \notag \\
& \quad  =\frac{  \nu \cdot f(\xb; \bmu_1 , \bSigma) } {\nu \cdot f(\xb; \bmu_1 , \bSigma) + ( 1 - \nu)  \cdot f(\xb; \bmu_1 , \bSigma) }
 = \PP\bigl [ F_{\btheta} (\bX) = 1\given \bX = \xb \bigr ],
\$
which implies that $(\eta , \bX)$ and $[ F_{\btheta} (\bX), \bX]$ has the same distribution under $H_0$. Here we use $f(\xb; \mu, \bSigma)$ to denote the density of $N(\bmu, \bSigma)$ at $\xb$.

Furthermore, let $\vb_0 = \Delta \bmu / \sqrt{ \Delta \bmu^\top \bSigma^{-1} \Delta \bmu }$.  We define
$
g  (\xb)   = \vb_0^\top \bSigma^{-1} ( \xb - \bmu)
$
and $ \overline g  (\xb) = g(\xb) \cdot \ind   [  F(\xb) = 1   ]$.
We
consider the  query function
 \#\label{eq:new_quary_cluster}
\overline q  (\xb) =   \overline g (\xb)   \cdot  \ind \bigl \{ |   g(\xb)|      \leq R \cdot  \sqrt{ \log n}   \bigr \},
\#
where $R $ is an absolute constant. Here  we adopt  truncation
to ensure that the query function is bounded.
Moreover, let $\bar Z $ be the random variable returned by the oracle $r$ defined in
  Definition \ref{def::oracle} for query function $\overline q $, and let $\bar z  $ be the realization of $\bar Z $.

  To characterize the effect of truncation in $\overline q $, by Cauchy-Schwarz inequality, under both $H_0$ and  $H_1$,  we have
  \#\label{eq:again_use_cauchy}
&  \bigl |  \EE   \bigl [  \overline q  (\bX) - \overline g  (\bX)   \bigr ]  \bigr | ^2   \leq \EE \bigl [ \overline g  ^2 (\bX)   \bigr ]  \cdot  \PP    \bigl (|   g(\xb)|  >  R \cdot  \sqrt{ \log n}    \bigr ).
  \#
  Note that $ g(\bX)  \sim N(0, 1 )$ under $H_0$ and that
  \$
   g(\bX) \sim \nu \cdot N\bigl [  (1- \nu)\cdot \vb_0^\top \bSigma^{-1} \Delta \bmu  , 1 \bigr ]  + ( 1- \nu) \cdot N\bigl(   - \nu \cdot \vb_0^\top \bSigma^{-1} \Delta \bmu  , 1 \bigr )
   \$ under $H_1$. By the definition of $\vb_0$, we have $\vb_0^\top \bSigma^{-1} \Delta \bmu = \sqrt{ \Delta \bmu ^\top \bSigma^{-1} \Delta \bmu} =o(1)   $. Thus, under both $H_0$ and $H_1$,
     $ g(\bX)$ is a sub-Gaussian random variable such that
   \# \label{eq:tail:of_sigmax}
   \PP \bigl [  \bigl |  g(\bX) \bigr |    \geq t    \bigr ]  \leq C_1 \cdot \exp( -C_2 \cdot t  )
   \#
   for any $t >0$, where
   $C_1$ and $C_2 $ are absolute constants.
   Hence, combining \eqref{eq:again_use_cauchy} and \eqref{eq:tail:of_sigmax}, we can set $R$  sufficiently large such that
 $  |  \EE     [\overline q  (\bX) - \overline g(\bX)   ]   | \leq 1/ n  $ under both $H_0$ and $H_1$.

   Furthermore, under $H_1$, by the definitions of $\cE_1$ and $\cE_2$ in \eqref{eq:two_events}, we have
   \#\label{eq:two_expectations}
   \EE  \bigl[ \overline{g}(\bX) \vert \cE_1 \bigr]  =  \nu  (1-\nu) \cdot  \vb_0 ^\top \bSigma^{-1} \Delta\bmu   , \quad
   \EE  \bigl[ \overline{g}(\bX) \vert \cE_2  \bigr]  =    - \nu  (1-\nu) \cdot  \vb_0 ^\top \bSigma^{-1} \Delta\bmu  .
   \#
 Whereas
 under the null hypothesis,  since there is only one Gaussian component,  we assume that  $F$ assigns  clusters  randomly, i.e., $F(\bX)$ is independent of $\bX$. In this case, we have $\EE [ \overline g(\bX) ] = 0$.

  Furthermore,  since clustering by $F$ is asymptotically accurate, it holds that $\overline{\PP}_{\btheta}(\cE^c) = o(1)$ under~the alternative hypothesis,    where $\cE^c$ denotes the complement of $\cE$. Recall that we assume that $T = O(d^{\eta})$ and $\bar z$ is the response  returned by the oracle $r$ for $\overline q$ in \eqref{eq:new_quary_cluster}. Based on $\{ z_t\}_{t=1}^T$ and $\bar z$, we define a test for  Gaussian mixture detection as
\# \label{eq:testfun33}
\phi_3(\{z_t\}_{t=1}^T, \bar z   )  = \ind \bigl\{  | \bar z| > C'  \sqrt {     \log n \cdot \log (d/\xi)     / n}   \bigr\},
\#
where $C'$ is an absolute constant.
Note that the tolerance parameter of the statistical query model in this case is
\#\label{eq:tauq33}
\tau_q = R \sqrt{ \log n}  \cdot \sqrt{  2  [ \log (T+1) + \log (1/\xi)  ]/ n}  = O\bigl\{ \sqrt {     \log n \cdot \log (d/\xi)     / n}  \bigr\}.
\#
Under $H_0$, by  Definition of \ref{def::oracle}, with probability at least $1- \xi$, we have
$$
 | \bar z| \leq   \bigl |  \EE [ \overline q(\bX)] \bigr  | + \tau_q \leq \bigl |  \EE [ \overline g (\bX)] \bigr  | + \tau_q + 1/ n \leq 2 \tau_q,
$$
which implies that
  type-I error of $\phi_3$ is no more than $\xi$.  Furthermore, under the alternative hypothesis, the type-II error is
\#\label{eq:conditioning_arg}
\overline{\PP} (\phi _3 = 0)  \leq  \overline \PP  (\phi_3 = 0\vert \cE _1 )+  \overline \PP  (\phi_3 = 0\vert \cE_2 )  +  \PP(\cE^c) .
\#
Conditioning on  $\cE_1$ or $\cE_2$  defined in \eqref{eq:two_events}, by \eqref{eq:two_expectations} we have
\$
| \bar z| &  \geq  \bigl |  \EE [ \overline q(\bX) \given \cE_i ] \bigr  |  - \tau_q \geq  \bigl |  \EE [ \overline g (\bX) \given \cE_i ] \bigr  | - \tau_q - 1/n \\
&\quad \geq \nu (1 - \nu)\cdot \sqrt{ \Delta \bmu ^\top \bSigma^{-1} \Delta \bmu} - 2 \tau_q \geq \gamma_n -2 \tau_q.
\$
Note that  $\gamma_n$ exceeds the information-theoretical limit, i.e., $\gamma_n = \Omega [ \sqrt{s \log d / n}]$.
Hence, combining  \eqref{eq:testfun33}, \eqref{eq:tauq33}, and \eqref{eq:conditioning_arg}, we conclude that
the type-II error is no more than $2\xi + o(1)$.
Therefore,
  $\phi_3$ is asymptotically powerful, contradicting the~lower~bound in Theorem \ref{thm::lower_known_cov}.~By combining the three claims, we conclude the proof.
\end{proof}

 \subsubsection{ Proof of Proposition \ref{prop::info_lower_bound_reg}}\label{proof::prop::info_lower_bound_reg}
Now we prove Proposition \ref{prop::info_lower_bound_reg}, which establishes the information-theoretic lower bound for detecting mixture of regressions.

\begin{proof}

We restrict the general detection problem in \S \ref{sec::bg_reg} to testing $H_0 \colon \btheta \in \overbar\cG_0  $ versus~$H_1 \colon \overbar \cG_1 ( s, \gamma_n)$, where the parameter spaces are defined in \eqref{eq:wbarg2}. Here $\sigma$ is an unknown constant. Then  under $H_0$ we have $Y\sim N(0, \sigma^2 + s\beta ^2)$ with $\beta> 0$  and  $\bX$ and $Y$ are independent.  In addition, under $H_1$ we have  $\bbeta \in  \{ \bbeta  = \beta \cdot \vb \colon \vb \in \cG(s) \}$, where $\cG(s) =  \{ \vb \in \{-1,0,1\}^d\colon \| \vb\|_0 = s  \}$. Recall that we define $\bZ = (Y,\bX)$. Hereafter, let  $\{  \bZ_i \}_{i=1}^n   $ be  $n$ independent copies of $\bZ$.

 We denote by   $\PP_{0} $ the probability  distribution of $(Y, \bX)$ under the null hypothesis and denote by $\PP_{\vb}$ the probability distribution under the  alternative hypothesis when $\bbeta = \beta \cdot \vb$.
 Besides, we denote $\overline{\PP} = 2^{-s} {d\choose s}^{-1} \sum_{\vb\in \cG( s)}\PP_\vb^{  n}$, where we use the superscript $n$ to denote the $n$-fold product probability measure.  By the  Neyman-Pearson Lemma, we immediately have
\#\label{eq::111}
 R_n^*(\cG_0, \cG_1)  &\geq \inf_{\phi}  \bigl[ \PP_{0}^{  n}(\phi = 1) + \overline{\PP}   (\phi = 0) \bigr] = 1 -  \frac{1}{2} \cdot\EE_{\PP_0^n}  \biggl[ \biggl| \frac{ \ud \overline\PP}{ \ud \PP_0 ^{n} }(\bZ_1, \ldots, \bZ_n) - 1\biggr|\biggr] \notag\\
&  \geq 1- \frac{1}{2}  \cdot \biggl(\EE_{\PP_0^{  n}} \bigg\{\biggl[\frac{ \ud \overline\PP}{ \ud \PP_0 ^{  n} }(\bZ_1, \ldots, \bZ_n )  \biggr]^2\biggr\} -1  \biggr)^{1/2},
\#
where the second inequality follows from Cauchy-Schwarz inequality. In what follows, we show that $\EE_{\PP_0^{  n}}   [\ud \overline\PP/ \ud \PP_0 ^{  n} (\bZ_1, \ldots, \bZ_n )  ]^2 = 1+ o(1)$, which  implies $ R_n^*(\cG_0, \cG_1) \geq 1 - o(1)$ by  \eqref{eq::111}.

By calculation, we have
\#\label{eq::122}
\EE_{\PP_0^{  n}}  \biggl\{\biggl [\frac{ \ud \overline\PP}{ \ud \PP_0 ^{  n} }(\bZ_1, \ldots, \bZ_n )  \biggr]^2\biggr\} =  2^{-2s}    {d \choose s}^{-2}  \sum_{\vb, \vb' \in \cG(s)}   \EE_{\PP_0^{  n}  }\biggl[\frac{\ud \PP_{\vb}^{  n} } {\ud \PP_0 ^{  n} } \frac{\ud \PP_{\vb'}^{  n}  }{\ud \PP_0^{  n} } (\bZ_1, \ldots, \bZ_n) \biggr].
\#
The following lemma calculates the right-hand side of \eqref{eq::122} in closed form.

\begin{lemma}\label{lemma::compute_chi_square2}
For any $\vb_1, \vb_2 \in \cG(s)$, we have
\$
& \EE_{\PP_0  }  \biggl[ \frac{\ud \PP_{\vb_1} }{\ud \PP_{0}}  \frac{\ud \PP_{\vb_2}  }{\ud \PP_{0}} (\bZ ) \biggr] = \biggl[ 1 -  \frac{\beta^4  \la\vb_1,\vb_2 \ra^2}{(\sigma^2 + s\beta^2)^{2}} \biggr]^{-1},\\
&  \EE_{\PP_0^n }  \biggl[ \frac{\ud \PP_{\vb_1}^n}{\ud \PP_{0}^n}   \frac{\ud \PP_{\vb_2}^n }{\ud \PP_{0}^n} (\bZ_1, \ldots, \bZ_n) \biggr]   =  \biggl[ 1 - \frac{\beta^4   \la \vb_1, \vb_2\ra^2}{(\sigma^2 + s\beta^2)^2} \biggr]^{-n},
  \$
where we use  $ \{ \bZ_i  \}_{i =1}^n $ to denote the $n$ independent copies of $\bZ$.
\end{lemma}
\begin{proof}
See \S \ref{proof::lemma::compute_chi_square2} for a detailed proof.
\end{proof}

From \eqref{eq::122} and the basic inequality
\$
(1- x^2)^{-1}\leq \cosh(2x)= \frac{\exp(2x) + \exp(-2x)}{2}
\$ for any  $x \in [-1/2, 1/2]$, we have
\#\label{eq::123}
\EE_{\PP_0^{  n}}  \biggl[ \frac{\ud \overline\PP}{  \ud \PP_0 ^{  n} } (\bZ_1, \ldots, \bZ_n ) \biggr ]^2 & =   2^{-2s}    {d \choose s}^{-2}    \sum_{\vb, \vb'\in \cG(s)} \biggl[1 - \frac{\beta^4 \la \vb, \vb' \ra^2}{(\sigma^2+s\beta^2  )^{2}} \biggr] ^{-n} \notag\\
&\leq  2^{-2s}    {d \choose s}^{-2}    \sum_{\vb, \vb'\in \cG(s) }  \cosh        \biggl ( \frac{ 2 \beta^2 \la\vb, \vb'\ra    } { \sigma^2 + s\beta^2} \biggr)^n .
\#
Let $ \xi_1, \ldots,\xi_n  $ be $n$  i.i.d. Rademacher random variables, then \eqref{eq::123} can be written as
\#\label{eq::124}
\EE_{\PP_0^{  n}}  \biggl[ \frac{\ud \overline\PP}{  \ud \PP_0 ^{  n} } (\bZ_1, \ldots, \bZ_n ) \biggr ]^2 & =   2^{-2s}    {d \choose s}^{-2}   \sum_{\vb, \vb'\in \cG(s) } \EE_{\bxi} \biggl [\exp \biggl ( \frac{ 2 n \beta^2     } { \sigma^2 + s\beta^2} \sum_{i=1}^n \sum_{j=1}^d \xi_i v_j v_j'\biggr)\biggr] .
\#
We define $\cC(s) = \{ \cS\subseteq [d] \colon  | \cS| =s\}$ as all  subsets of $[d]$ with cardinality $s$. Then for any $\vb\in \cG(s)$, the support of $\vb$ is in $\cC(s)$. We denote $\vb \sim \cS$ if $\textrm{supp} (\vb ) = \cS$ for notational simplicity.  Then we can write \eqref{eq::124} as
\#\label{eq::125}
&\EE_{\PP_0^{  n}}  \biggl[ \frac{\ud \overline\PP}{  \ud \PP_0 ^{  n} } (\bZ_1, \ldots, \bZ_n ) \biggr ]^2  \notag\\
&\quad =  2^{-2s}    {d \choose s}^{-2} \sum_{\cS, \cS' \in \cC(s)}  \sum_{\vb\sim \cS, \vb'\sim \cS' }   \EE_{\bxi} \biggl [ \exp \biggl ( \frac{ 2 n \beta^2     } { \sigma^2 + s\beta^2}  \sum_{i=1}^n \sum_{j=1}^d \xi_i v_j v_j'\biggr) \biggr] \notag \\
 &  \quad = {d \choose s}^{-2}  \sum_{\cS, \cS' \in \cC(s)} 2^{-2s} \sum_{\vb\sim \cS, \vb'\sim \cS' }  \EE_{\bxi} \biggl [ \exp \biggl ( \frac{ 2 n \beta^2     } { \sigma^2 + s\beta^2}  \sum_{i=1}^n \sum_{j\in\cS \cap \cS'} \xi_i v_j v_j'\biggr)\biggr].
\#
Here we denote by $\EE_{\bxi}$ the expectation with respect to the randomness of $\xi_1, \ldots, \xi_n$.
Given $\cS, \cS' \in \cC(s)$, let $\overbar{T}$ be the sum of $| \cS\cap \cS'|$ independent Rademacher random variables. Then we have
\#\label{eq::126}
 2^{-2s} \sum_{\vb\sim \cS, \vb'\sim \cS' }  \EE_{\bxi}\biggl[\exp \biggl ( \frac{ 2 n \beta^2     } { \sigma^2 + s\beta^2}  \sum_{i=1}^n \sum_{j\in\cS \cap \cS'} \eta_i v_j v_j'\biggr) \biggr]  =   \EE_{\overbar{T}, \bxi} \biggl[\exp \biggl (\frac{ 2 n \beta^2     } { \sigma^2 + s\beta^2}  \sum_{i=1}^n  \xi_i \overbar{T} \biggr)\biggr] .
\#
Let $\cS, \cS'$ be two i.i.d. random sets that are uniformly distributed over $\cC(s)$. By \eqref{eq::125} and \eqref{eq::126}~we have
\#\label{eq::1467}
\EE_{\PP_0^{  n}}  \biggl[ \frac{\ud \overline\PP}{  \ud \PP_0 ^{  n} } (\bZ_1, \ldots, \bZ_n ) \biggr ]^2 &=  \EE_{\cS, \cS'}  \EE_{\overbar{T}, \bxi}\biggl[\exp \biggl (\frac{ 2 n \beta^2     } { \sigma^2 + s\beta^2}  \sum_{i=1}^n  \xi_i \overbar{T} \biggr) \biggr] \notag \\
& = \EE_{\cS, \cS'} \biggl\{ \biggl[\cosh\biggl(\frac{2n \beta^2}{\sigma^2 + s \beta^2}\biggr)\biggr]^{|\cS\cap \cS'|} \biggr\}.
\#
By the proof in \cite{arias2017detection}, the last term in \eqref{eq::1467} equals $ 1 + o(1)$ if
\$
\frac{s\beta ^2}{\sigma^2 + s \beta^2} =  o\biggl\{ \max \biggl[ \sqrt { \frac{s\log (d/s)}{n}}, \frac{s\log (d/s)}{n} \biggr]  \biggr\}.\$ Note that $s\log (d/s)/n = o(1)$  and $\sigma$ is a constant. Therefore, we obtain that, if
\$
\frac{ s\beta ^2 }{ \sigma ^2} = \rho(\btheta) = o\biggl\{ \max \biggl[ \sqrt { \frac{s\log (d/s)}{n}}, \frac{s\log (d/s)}{n} \biggr]  \biggr\},
\$
then we have that $\lim_{n\rightarrow \infty} R_n^* ( \cG_0, \cG_1) \geq 1$, which concludes the proof since $s < d$ and $n$ is sufficiently large such that $s\log (d/s)/n < 1$.
\end{proof}

\subsubsection{Proof of Theorem \ref{thm::lower_reg}}\label{proof::thm::lower_reg}
Next we prove Theorem \ref{thm::lower_reg}, which quantifies the hardness of detecting mixture of regressions~under finite computational budgets.
\begin{proof}
Similar to the proof of the information-theoretic lower bound, we study the    restricted detection problem $H_0 \colon \btheta \in \overbar{\cG}_0  $ versus  $H_1 \colon \overbar{\cG}_1 ( s, \gamma_n)$,
 in which $\overbar{\cG}_0 $ and $\overbar{\cG}_1 ( s, \gamma_n)$ are defined in \eqref{eq:wbarg2}.~Following the same notations in the proof of Proposition \ref{prop::info_lower_bound_reg},    we  denote by $\PP_{0}$ the null distribution  and by  $\PP_{\vb}$  the alternative distribution when $ \bbeta = \beta \cdot \vb$ for some
 $\vb \in\cG(s)$.~Here $\beta >0$ is a fixed number. Under the assumption that $\rho(\btheta) = s \beta ^2 /\sigma^2 = o ( \sqrt{s^2/n})$, we have  $ n \beta ^4/ \sigma^4 = o(1)$.

  Moreover, we define  the distribution of the random variables returned by the oracle~under~the~null distribution  as $\overline{\PP}_0$ and define $\overline{\PP}_{\vb}$ similarly.
Hence the risk $\overline {\cR}_n^* ( \cG_0, \cG_1; \mA, r)$ defined in \eqref{eq::minimax_risk_oracle} is  lower bounded by
\#\label{eq::first_step_reduce_to_finite}
\overline {R}_n^* ( \cG_0, \cG_1; \mA, r) \geq {     \inf_{\phi\in \cH(\mA,r)} }  \Bigl[ \overline{\PP}_{0}(\phi = 1) + {  \sup _{\vb\in \cG(s) }} \overline{\PP}_{\vb}(\phi = 0) \Bigr].\#
Following the proof of Theorem   \ref{thm::lower_known_cov}, for any query function $q\in \cQ_{\mA}$, we define $\cC(q)$ as  in \eqref{eq::distinguished_set}, and $\cC_1(q)$ and $\cC_2(q)$ as in \eqref{eq::def_cq}. By Lemma \ref{lemma::distinguish}, to show the right hand side of \eqref{eq::first_step_reduce_to_finite} is not asymptotically negligible, it remains to show that \$T \cdot   \sup_{q \in \cQ_{\mA}}  |\cC_1(q) | + T  \cdot  \sup_{q \in \cQ_{\mA}} | \cC_2(q)| < | \cG(s)|.\$
Also, for any $\ell \in \{ 1,2\}$ we define $ \PP_{\cC_{\ell}(q)}$ as the uniform mixture of $\{ \PP_{\vb} \colon \vb\in \cC_{\ell}(q)\}$ and define $\overline {\cC}_{\ell}(q, \vb)$ as in \eqref{def::bar_C} for $\vb\in \cG(s)$. Combining \eqref{eq::chi_square_divergence} and \eqref{def::bar_C}, for $\ell \in \{ 1, 2\} $, by the definition of $\chi^2$-divergence,~we have
\#\label{eq::chi2_bound}
D_{\chi^2} (\PP_{ \cC_{\ell}(q)}, \PP_0 ) &\leq  \sup _{\vb\in \cC _{\ell}(q) }   \frac{1}{| \overline{ \cC}_{\ell}(q, \vb)|} \sum_{\vb' \in \overline{\cC}_{\ell}(q, \vb)} \EE_{\PP_0}   \biggl [\frac{\ud \PP_{\vb }}{\ud \PP_0} \frac{\ud \PP_{\vb'}}{\ud \PP_0}(\bZ) \biggr] - 1.
\#
By Lemma \ref{lemma::compute_chi_square}  and the  inequality  $(1-x^2 )^{-1} \leq \cosh(2 x )$ for $x\in [-1/2, 1/2]$ we have
\#\label{eq::127}
\EE_{\PP_0}   \biggl [\frac{\ud \PP_{\vb }}{\ud \PP_0} \frac{\ud \PP_{\vb'}}{\ud \PP_0}(\bZ)\biggr ] \leq \cosh\biggl(\frac{2\beta ^2 \la \vb, \vb' \ra}{\sigma^2 + s\beta ^2}\biggr).
\#
For notational simplicity, we define $\mu = 2 \beta ^2 / (\sigma ^2 + s \beta^2 )$  and
$
h(t) =   \cosh [(s-t)   \mu  ]
$ for $t \in \{0,\ldots, s\}$. By combining \eqref{eq::chi2_bound} and \eqref{eq::127}, we obtain that
\#\label{eq::128}
D_{\chi^2} (\PP_{ \cC_{\ell}(q)}, \PP_0 )  \leq  {  \sup _{\vb\in \cC _{\ell}(q) }} \frac{\sum_{\vb' \in \overline{\cC}_{\ell}(q, \vb)} \cosh( \mu \la \vb , \vb'\ra) - 1}{| \overline{ \cC}_{\ell}(q, \vb)|}.
\#
To establish an upper bound for   the right-hand side of \eqref{eq::128},  we define
\$
\cC_j(\vb) = \bigl\{\vb' \in \cG(s) : |\la \vb , \vb' \ra| = s-j\bigr\}
\$
for any $j \in\{0,\ldots, s\}$ and any fixed  $\vb\in \cG(s)$. Then for~$\ell \in \{ 1,2 \}$, any  query function $q \in \cQ_{\mA}$, and any~$\vb \in \cC_{\ell} (q)$, by the monotonicity of function $h(t)$ and  the definition of set~$\overline{\cC}_{\ell}(q, \vb)$ in~\eqref{def::bar_C},  there exists  an integer  $k_{\ell} (q, \vb )$  satisfying
 \$
  \overline{\cC}_{\ell}(q, \vb) = \cC_{0}(\vb) \cup \cC_1(\vb) \cup \cdots \cup  \cC_{k_{\ell} (q, \vb) -1}(\vb) \cup \cC'_{\ell} (q,\vb ),
 \$
 where $\cC'_{\ell} (q,\vb) = \overline \cC_{\ell}(q, \vb ) \setminus { \bigcup_{j=0}^{k_{\ell}(q, \vb) - 1}} \cC_j(\vb)$ has cardinality
 \$
| \cC'_{\ell} (q,\vb) | =  | \cC_{\ell}(q)| -{   \sum_{  j = 0}^{ k_{\ell}(q,\vb) -1}  } | \cC_{j} (\vb)| < | \cC_{k_{\ell} (q, \vb )} (\vb)|.
 \$
 Thus the cardinality of $\overline{\cC} _{\ell} (q , \vb)$  can be bounded  by
  \#\label{eq::set_bound2}
{   \sum_{j=0}^{k_{\ell}(q, \vb ) } } |\cC_j(\vb) | > | \overline{ \cC} _{\ell} (q , \vb)|\geq {  \sum_{j=0}^{k_{\ell} (q, \vb )-1}} |\cC_j (\vb) |.
 \#
	 Then by  \eqref{eq::128} and \eqref{eq::set_bound2} we further obtain that
 \#\label{eq::129}
 1 + D_{\chi^2} (\PP_{  \cC_{\ell}(q)}, \PP_0   )\leq & { \sup _{\vb\in \cC _{\ell}(q) }}  \frac{\sum_{j=0}^ {k_{\ell}(q, \vb )-1} h(j)\cdot |\cC_j(\vb)| + h[k_{\ell}(q, \vb )] \cdot | \cC_{\ell}'(q, \vb ) |}  { \sum _{j=0}^{k_{\ell}(q, \vb)-1} | \cC_j(\vb)| + | \cC_{\ell}'(q, \vb)|}\notag \\
 \leq &{ \sup _{\vb\in \cC _{\ell}(q) }} \frac{\sum_{j=0}^ {k_{\ell}(q , \vb )-1} h(j) \cdot    | \cC_j(\vb)| }  { \sum _{j=0}^{k_{\ell}(q, \vb )-1} | \cC_j(\vb)| }.
 \#
 Here the second inequality in \eqref{eq::129} follows from the   monotonicity of $h(t)$.

To obtain an upper bound for right-hand side of \eqref{eq::129}, note that by Lemma \ref{lemma::growth_Cj} we have
  \$| \cC_{j}(\vb)| \leq \zeta^{j-s} | \cC_{s}(\vb)|,~ \text{for all}~   j \in \{0,\ldots, s\}, \$ where we denote $\zeta = d/ (2s^2)$ for notational simplicity. Then under the assumptions of the theorem, we have $\zeta^{-1} = o(1)$ and $\zeta = \Omega(d ^{\delta})$ for some  constant $\delta$ that is  sufficiently small.
Then by  \eqref{eq::set_bound2}~and the definition of $\overline {\cC}_{\ell}(q, \vb)$ in \eqref{def::bar_C}, for any  $q \in \cQ_{\mA}$, we   obtain that
		\#\label{eq::upper_bound_Cq2}
	 |\cC_{\ell}(q)| = |\overline {\cC}_{\ell}(q, \vb)|&\leq { \sum_{j=0}^{k_{\ell}(q, \vb ) }} |\cC_j(\vb)|
	 \leq   |\cC_{ s }(\vb)| { \sum_{j=0}^{k_{\ell}(q, \vb ) }} \zeta^{j - s} \notag \\
	 &\leq \frac{\zeta^{-[ s  - k_{\ell}(q, \vb ) ]} |\cG(s) |}{1 - \zeta^{-1}} \leq 2\zeta^{-[ s  - k_{\ell}(q, \vb ) ] } | \cG(s)|,
	 \#
	where the last inequality follows from    $\zeta^{-1} = 2s^2 /d = o(1)$. In the following, we denote $k_{\ell} = k_{\ell}(q, \vb)$ to simplify the  notations.
	
 By  combining  \eqref{eq::129}, \eqref{eq::upper_bound_Cq2}, and  \eqref{eq::useless2} with   $a_j = | \cC_j(\vb)|$~and $b_j = \zeta^j$, we  obtain
  \#\label{eq::compute_expectation2}
	1 + D_{\chi^2} (\PP_{ \cC_{\ell}(q)}, \PP_0 )&\leq \frac{ \sum_{j=0}^{k_{\ell}-1} \zeta^j  \cosh\bigl[( s  - j) \mu \bigr] }{\sum_{j=0}^{k_{\ell}-1} \zeta^j} \notag\\
	&\leq  \frac{ \cosh\bigl[(s - k_{\ell} + 1) \mu\bigr]  \cdot (1-\zeta^{-1})}{1 - \zeta^{-1}\cosh({\mu})},
	\#
	where we use the fact that $\cosh(\mu) / \zeta = o(1)$, which holds because
	  $ n  \beta ^ 4 /\sigma ^4= o(1)$ and  $s^2 /d =o(1)$. Moreover, by the  inequality $\cosh(x) \leq \exp(x^2/2)$, we have
\#\label{eq::130}
1 + D_{\chi^2} (\PP_{ \cC_{\ell}(q)}, \PP_0 )\leq \exp\bigl[ (s - k _{\ell} + 1) ^2 \mu^2 /2 \bigr] \cdot \frac{1- \zeta^{-1}}{1 - \zeta^{-1}\cosh({\mu})}.
\#	
In the following, for notational simplicity, we denote $\sqrt{2 \log (T/\xi)/(3n)}$ by $\tau$. By combining \eqref{eq::130}~and  Lemma \ref{lemma::bound_chi_square_div}, we obtain that
	\#\label{eqn::second_bound2}
	 (s  - k_{\ell} + 1 )^2 \geq \frac{2 \log (1+\tau^2 )}{\mu^2} - 2\log \biggl[ \frac{1-\zeta^{-1}}{1 - \zeta^{-1}\cosh(\mu )} \biggr]\bigg/\mu^2.
	\#
	Moreover, by Taylor expansion and the fact that $\cosh(\mu) /\zeta = o(1)$, we have
	\#\label{eq::bound_by_taylor}
	\log \biggl[\frac{1-\zeta^{-1}}{1 - \zeta^{-1}\cosh({\mu} )} \biggr] = \log \biggl \{ 1 + \frac{\zeta^{-1} \bigl[ \cosh(\mu) -1\bigr]   }{1-  \zeta^{-1} \cosh(\mu ) } \biggr\} = O( \zeta^{-1} \mu^2).
	\#
	Similar to  \eqref{eq::bound_by_taylor1}, we conclude that the first term on the right-hand side of \eqref{eqn::second_bound2} is   dominant, which implies that
	$
	 (s  - k_{\ell} + 1 )^2 \geq  \log (1+\tau^2 ) / \mu^2
$ when $n$ is sufficiently large.

	Combining  \eqref{eqn::second_bound2} and \eqref{eq::bound_by_taylor}, we finally have
	\#\label{eq::bound_integer_k_22}
	  k_{\ell} (q, \vb) \leq  s +1  - \sqrt{  \log (1+\tau^2 )/\mu^2  },  ~\text{for~all}~\ell \in \{ 1,2\}.
	\#
	Furthermore, it can be seen that  \eqref{eq::bound_integer_k_22} holds for all $q \in \cQ_{\mA}$ and all $\vb \in \cC_{\ell}(q)$.
	After obtaining upper bounds for $k_1$ and $k_2$,
	 combining   \eqref{eq::set_bound}, \eqref{eq::upper_bound_Cq2},  and \eqref{eq::bound_integer_k_22}, we further obtain
\#\label{eq::bound_Main_term2}
 \frac{T \cdot    \sup  _{q\in \cQ_{\mA}} | \cC(q)|}{|\cG(s)|}
 \leq 4T \cdot   \exp \Bigl\{ -\log \zeta\cdot \bigl[  \sqrt{\log (1+\tau^2 )/\mu^2  }-1\bigr] \Bigr\}.
 \#

 For any constant $\eta >0$, we set $T = O(d^\eta)$.
Remind that we denote $\tau = \sqrt{2 \log (T/\xi) / (3n)} $  where~$\xi = o(1)$. By inequality $\log(1+x) \geq x/2$, it holds that
	$\log (1+ \tau^2) \geq \tau^2/2 = \log (T/\xi) / (3n) $.~Furthermore, under the condition that  $ n \beta ^4/ \sigma^4 = o(1)$,~we have
	\$
	\frac{\log (T/\xi)}{3n\mu^2} \geq \frac{\sigma ^4  \log (T/\xi)}{3n \beta^4}\rightarrow \infty.
	\$
	 Let $n$ be sufficiently large such that
	\#\label{eq::12345}
	\frac{\log (1+ \tau^2)}{\mu^2} \geq \frac{\sigma^4 \log (T/\xi)}{3n\beta^4} > C^2,
	\#
	where $C$ is an absolute constant satisfying  $\delta (C-1) > \eta$. Then  combining \eqref{eq::bound_Main_term2} and \eqref{eq::12345} we have
	\#\label{eq::bound_key_quantity_11}
	\frac{T \cdot   \sup_{q \in \cQ_{\mA}} |\cC(q)|}{|\cG(s) |} &\leq 4T \cdot   \exp \Bigl\{ -\log \zeta\cdot \bigl[  \sqrt{\log (1+\tau^2 )/\mu^2  }-1\bigr] \Bigr\} \notag\\
	& =  O\bigl[4 d^\eta \zeta^{-(C-1) }\bigr]=   O\bigl[4d^{\eta-\delta (C-1)}\bigr] = o(1).
	\#
	Finally, by   \eqref{eq::bound_key_quantity_11}  and Lemma \ref{lemma::distinguish}, we conclude that $\overline {R}_n^* ( \cG_0, \cG_1; \mA, r) \rightarrow 1$ if $\gamma_n = o ( \sqrt{s^2/n})$.
\end{proof}

\subsubsection{ Proof of Theorem \ref{col::implications_reg}}\label{proof::col::implications_reg}

\begin{proof}
~In the following, we prove by contradiction. Suppose there exists an absolute constant $\eta >0$
  and $\mA \in \cA(T)$ in which $T = O(d^{\eta})$, such that under the mixture of regression model with parameter $\btheta = ( \bbeta, \sigma^2)$,  for any oracle $r \in \cR[\xi,n,T, M, \eta(\cQ_{\mA}) ]$, we  obtain an estimator $\hat \bbeta$ of $\bbeta$ such that
 $
 \| \hat \bbeta - \bbeta \|_2^2  /  \sigma^2  \leq   \gamma_n /64
 $
 holds with probability tending to one.
 Recall that we have $ \| \bbeta \|_2^2 / \sigma^2  = \gamma _n.$

 By Cauchy-Schwarz inequality we have
\#\label{eq::implication1}
\bigl|  \| \hat \bbeta \|_2^2 - \| \bbeta \|_2^2 \bigr|  ^2  = \bigl| ( \hat \bbeta - \bbeta ) ^\top ( \hat \bbeta + \bbeta ) \bigr|^2 \leq  \| \hat \bbeta - \bbeta \|_2^2 \cdot \| \hat \bbeta + \bbeta \|_2^2.
\#
In addition, by triangle inequality and the fact that $ \| \bbeta \|_2^2 / \sigma^2  = \gamma _n $ we have
\#\label{eq::implication2}
  \| \hat \bbeta + \bbeta \|_2^2 & \leq \bigl( \| \hat \bbeta - \bbeta \|_2 + 2\| \bbeta \|_2 \bigr)^ 2 \leq 2 \| \hat \bbeta - \bbeta \|_2^2 +8 \| \bbeta \|_2^2 \notag \\
&  \leq  8 \| \bbeta \|_2^2  +   \sigma^2 \cdot \gamma_n  /32 \leq 9 \| \bbeta \|_2 ^2 .
\#
Combining \eqref{eq::implication1} and \eqref{eq::implication2}  we obtain that
\$
\bigl|  \| \hat \bbeta \|_2^2 - \| \bbeta \|_2^2 \bigr|  ^2  \leq  \sigma^2/ 64  \cdot  \gamma_n      \cdot  9\| \bbeta \|_2^2  = 9/64\cdot  \sigma^4\gamma_n^2,
\$
which then implies that $|  \| \hat \bbeta \|_2^2 - \| \bbeta \|_2^2 |  /  \sigma^2  \leq 3/8\cdot \gamma_n$. Thus, under the alternative hypothesis, with   probability tending to one, we have
$
\| \hat \bbeta \|_2^2 / \sigma^2 \geq 5/8\cdot \gamma_n.
$

 Furthermore, under the null hypothesis, since $\bbeta = {\bf 0}$, the algorithm produces an estimator $\hat \bbeta $ such that
 $\| \hat \bbeta - \bbeta \|_2 ^2 /\sigma^2 = \| \hat \bbeta \|_2 ^2 / \sigma^2 \leq \gamma_n  / 64$ with high probability.

Therefore, the test function $
 \phi(\{w_t\}_{t=1}^T ) = \ind ( \| \hat \bbeta \|_2^2 / \sigma^2 \geq 5/8 \cdot \gamma_n)$
 is asymptotically powerful, where $w_t$ is the realization of the random variable $W_t$ returned by the oracle  for query function $q_t$. Since $\gamma _n = o( \sqrt{s^2 / n})$, the existence of an asymptotically powerful test with polynomial oracle complexity contradicts the computational lower bound  in Theorem \ref{thm::lower_reg}.
\end{proof}

\subsection{Proofs of Upper Bounds}\label{pf::upper}
In this section we lay out the proofs of the upper bounds for Gaussian mixture detection. In specific, we prove that the  hypothesis tests  in \eqref{eq::test_fun1} and \eqref{eq::test_fun2} are asymptotically powerful, which implies~the tightness of   the lower bounds established in \S\ref{sec::lower_bound}.
\subsubsection{Proof of Theorem \ref{thm::test1}} \label{proof::thm::test1}
\begin{proof}
To simplify the notation,  for any $\vb\in \cG(s)$, we define $q_{\vb}^*(\xb) = (\vb^\top \bSigma^{-1}\xb)^2 / (\vb ^\top \bSigma^{-1} \vb)$.
 Note that under $\PP_0$, for any $\vb \in \RR^d$,
 $
 \vb^\top \bSigma^{-1}\bX / \sqrt{   \vb^\top \bSigma^{-1} \vb} $ is a standard normal random variable.
Thus  $q_{\vb}^*(\bX)\sim \chi_1^2$ under  $\PP_0$, which implies that $\EE_{\PP_0} [q_{\vb}^*(\bX)] = 1$. As in the proofs of the  lower bounds in \S\ref{pf::lower}, we denote by $\PP_{\vb}$ the probability distribution under the alternative hypothesis with model parameter $\btheta = [ - \beta (1- \nu) \vb, \beta \nu \vb, \Ib]$. Let $\overline{\PP}_{0}$ and $\overline{\PP}_{\vb}$ be the distributions of the random variables returned by the oracle under $\PP_0$ and $\PP_{\vb}$, respectively.~Then under  $\PP_{\vb}$, for~any $\vb' \in \cG(s)$,  we have
\#\label{eq::distri_test_stat}
\frac{ {\vb'}^\top\bSigma ^{-1}   \bX}{ \sqrt{ {\vb'}^\top \bSigma^{-1} \vb'} }  &\sim \nu \cdot  N  \biggl( -\frac{ \beta (1- \nu) {\vb'}^\top\bSigma ^{-1}\vb}{ \sqrt{ {\vb'}^\top \bSigma^{-1} \vb'}  } ,  1  \biggr)   +     (1- \nu)\cdot N  \biggl ( \frac{ \beta  \nu{\vb'}^\top \bSigma ^{-1}\vb}{   \sqrt{ {\vb'}^\top \bSigma^{-1} \vb'}         }  ,  1 \biggr  ) .
\#
 Therefore, the expectation of $q_{\vb'}^*(\bX)$ under $\PP_{\vb}$ is given by
\$
\EE_{\PP_{\vb}}[q_{\vb' }^*(\bX)] &=  1 + \frac{\beta^2 \nu(1- \nu)  | {\vb'}^\top \bSigma ^{-1} \vb |^2}{\vb'^ \top \bSigma^{-1} \vb'}\leq   1+ \beta^2 \nu(1- \nu) \vb ^\top \bSigma^{-1} \vb,
\$
where the inequality follows from Cauchy-Schwarz inequality and  equality is attained by $\vb' = \vb $.~Thus we have
\#\label{eq::diff_mean}
 \sup_{\vb' \in \cG(s) } \Bigl\{ \EE_{\PP_{\vb}}\bigl[q_{\vb' }^* (\bX)\bigr]  -\EE_{\PP_0} \bigl[q_{\vb'}^*(\bX)\bigr] \Bigr\}=  \beta^2 \nu(1- \nu) \vb^\top \bSigma^{-1} \vb.
\#

In the following, we characterize the effect of truncation in \eqref{eq::query_fun1} by bounding the difference between $q_{\vb'}(\bX)$ and $q_{\vb'}^* (\bX)$ under $\PP_0$ and $\PP_{\vb}$ for any   any $\vb ' \in \cG(s)$. Under the null hypothesis, since $q_{\vb'}^* (\bX) \sim \chi_1^2$, by Cauchy-Schwarz inequality, we have
\# \label{eq:use_cauchy1}
\bigl | \EE_{\PP_0} [ q_{\vb'}(\bX) - q_{\vb'}^* (\bX)  ]\bigr | ^2     & \leq \EE _{\PP_0} \bigl \{  [ q_{\vb' }^* (\bX)]^2 \bigr \} \cdot  \PP_0  \Bigl (  |  {\vb'}^\top\bSigma ^{-1}   \bX | > R \sqrt{\log n} \cdot  \sqrt{ {\vb'}^\top \bSigma^{-1} \vb'}  \Bigr )
\notag \\
& \leq 6 \cdot \exp( -R^2 \log n /2 ).
\#
Here in the last inequality  we use the fact that  $\PP( \varepsilon > t) \leq \exp(-t^2 / 2 )$ for all $t >  0$, where $\varepsilon \sim N(0,1)$.
Similarly, under $\PP_{\vb}$, Cauchy-Schwarz inequality implies that
\#\label{eq:use_cauchy2}
\bigl | \EE_{\PP_{\vb}} [ q_{\vb'}(\bX) - q_{\vb'}^* (\bX)  ]\bigr | ^2  \leq  \EE _{\PP_\vb} \bigl \{  [ q_{\vb' }^* (\bX)]^2 \bigr \} \cdot  \PP_{\vb}  \bigl [   q_{\vb'}^*(\bX)> R ^2 \cdot \log n  \bigr] .
\#
Note that \eqref{eq::distri_test_stat} implies that  $ \vb^\top \bSigma^{-1}\bX / \sqrt{   \vb^\top \bSigma^{-1} \vb} $ can be written as the sum of   a Bernoulli  and  a  standard normal random variable. In addition, recall that the $\psi_1$-norm of a random variable $W \in \RR$ is defined as
$\| W \|_{\psi_1} = \sup_{p\geq 1} p^{-1}\cdot  ( \EE|W| ^p )^{1/p}$.  We  denote the $\psi_1$-norm under $\PP_{\vb}$ by   $\| \cdot \|_{\psi_1, \vb}$  hereafter. Using  the  inequality $(a+b)^2 \leq 2 a^2 + 2b^2$, we have
\$
\bigl\| q_{\vb'}^* (\bX)   \bigr\|_{\psi_1, \vb} \leq 2    \beta^2  \vb ^\top \bSigma^{-1} \vb+ 2 \| \varepsilon^2 \|_{\psi_1} ,
\$
where $\varepsilon \sim N(0,1)$.
Thus, under the assumption that
\#\label{eq:signal_assume}
 \beta^2  \nu(1- \nu) \vb^\top \bSigma^{-1} \vb  = \Omega \{  \log n \cdot   \sqrt{ [ s \log (2d ) + \log(1/\xi) ]  /n} \},
 \#  when $n$ is sufficiently large, we have $\bigl\| q_{\vb'}^* (\bX)   \bigr\|_{\psi_1, \vb} \leq 3  \| \varepsilon^2 \|_{\psi_1}$.
 Thus, by applying the sub-exponential tail to \eqref{eq:use_cauchy2}, we obtain that
 \#\label{eq:use_cauchy22}
 \bigl | \EE_{\PP_0} [ q_{\vb'}(\bX) - q_{\vb'}^* (\bX)  ]\bigr | ^2  \leq C_1 \cdot \exp( - C_2 \cdot R^2 \log n),
 \#
 where
 $C_1$ and $C_2$
  are absolute constants. Thus,   when $R$ is   sufficiently large, by combining \eqref{eq:use_cauchy1} and  \eqref{eq:use_cauchy22},
  we have
  \#\label{eq:truncation_effect}
  \max_{\vb'\in \cG(s) } \Bigl \{  \bigl | \EE_{\PP_0} [ q_{\vb'}(\bX) - q_{\vb'}^* (\bX)  ]\bigr |  + \bigl | \EE_{\PP_\vb} [ q_{\vb'}(\bX) - q_{\vb'}^* (\bX)  ]\bigr | \Bigr \}   \leq 1 / n.
  \#

Moreover, by the definition of  statistical query model in Definition \ref{def::oracle}, since the query functions are bounded by $R^2  \log n$ in absolute value, under both the null and alternative hypotheses, we have
\#\label{eq::tol_param}
\tau_{q_{\vb'}} & =R^2\cdot  \log n \cdot \sqrt{ 2   \bigl[ \log | \cG(s) | +\log(1/\xi)\bigr]/ n }\notag \\
& \leq 2 R^2\cdot \log n \cdot \sqrt{   \bigl[ s \log (2d ) + \log(1/\xi) \bigr] / n }
\#
for all $\vb ' \in \cG(s)$.
For notational simplicity, let  $\Lambda =   2 R^2 \log n \cdot   \sqrt{  [ s \log (2d ) + \log(1/\xi) ]/n}$. By
   \eqref{eq:signal_assume},      it holds that  $ \beta^2 \nu(1- \nu)    \cdot \vb ^\top \bSigma^{-1} \vb \geq 3 \Lambda$. Combing this with \eqref{eq::diff_mean} and  \eqref{eq:truncation_effect}, we~have
   \# \label{eq:final_mean_diff}
   \sup_{\vb' \in \cG(s) } \Bigl\{ \EE_{\PP_{\vb}}\bigl[q_{\vb' }  (\bX)\bigr]  -\EE_{\PP_0} \bigl[q_{\vb'} (\bX)\bigr] \Bigr\} \geq  \beta^2 \nu(1- \nu) \vb^\top \bSigma^{-1} \vb - 2/ n   \geq 2 \Lambda.
  \#
   Finally, combining
  \eqref{eq::tol_param} and \eqref{eq:final_mean_diff},
 we have
\$
\overline{R}(\phi) &= \overline{\PP}_0 \Bigl(   \sup_{\vb'\in \cG( s)} Z_{q_{\vb'}}\geq 1 +  \Lambda \Bigr) + \sup_{\vb \in \cG(s)} \overline{\mathbb{P}}_{\vb} \Bigl(  { \sup _{\vb'\in \cG( s)}}Z_{q_{\vb'} }<1 +   \Lambda \Bigr) \notag\\
& \leq \overline{\mathbb{P}}_0 \biggl(   { \bigcup_{\vb'\in \cG(  s)}} \Bigl\{ \bigl| Z_{q_{\vb'}}-\EE_{\PP_0}\bigl[q_{\vb'} (\bX)\bigr]\bigr| \geq  \tau_{q_{\vb'}} \Bigr\} \biggr)\notag\\
&\quad +  \sup_{\vb \in \cG( s)}  \overline{ \mathbb{P}}_{\vb} \biggl(   { \bigcup_{\vb'\in \cG(  s)}} \Bigl\{ \bigl| Z_{q_{\vb'}}-\EE_{\PP_{\vb}}\bigl[q_{\vb'} (\bX)\bigr] \bigr| \geq  \tau_{q_{\vb'} } \Bigr\} \biggr) \\
& \leq  2\xi,
\$
which concludes the proof of Theorem \ref{thm::test1}.
\end{proof}

\subsubsection{Proof of Theorem \ref{thm::test2} }\label{proof::thm::test2}
\begin{proof}
To simplify the notation, for any $j \in [d]$, we define $q_j^*(\bX) = X_j^2 / \sigma_j$. Under the null hypothesis, since
 $X_j  \sim N(0, \sigma_{j})$,  $q_{j}^*(\bX) $ is a $\chi_1^2 $ random variable,  which further implies that  $\EE_{\PP_0} [q_j^*(\bX)] = 1$. Moreover, for any $\vb \in \cG( s)$,  under $\PP_{\vb}$ we have
 \#\label{eq:entry_null_gaussianmix}
 X_j \sim  \nu N \bigl[  -(1-\nu)  \beta v_j , \sigma_j  \bigr]  + (1- \nu) N ( \nu\beta  v_j ,\sigma_j),~\text{for~all~} j \in \supp(\vb),
 \#
 and   $X_j \sim N(0, \sigma_j)$ otherwise. Here $v_j \in \{-1, 0, 1\}$ is the $j$-th entry of $\vb$.
Thus,  it  holds for any $j\in \supp(\vb)$ that
\#\label{eq::diff_mean2}
&\EE_{\PP_{\vb}}\bigl[q_j^* (\bX)\bigr] - \EE_{\PP_0} \bigl[q_j^* (\bX)\bigr] = \nu(1-\nu)\beta ^2/ \sigma_{j} .
\#
Similar to the proof of Theorem \ref{thm::test1}, we need to bound the difference between the expectations of $q_j(\bX)$ and $q_j^* (\bX)$. To this end, under the null hypothesis, by Cauchy-Schwarz inequality, we have
\#\label{eq:cauchy-111}
\bigl | \EE_{\PP_0} [ q_{j}(\bX) - q_{j}^* (\bX)  ]\bigr | ^2     & \leq \EE _{\PP_0} \bigl \{  [ q_{j }^* (\bX)]^2 \bigr \} \cdot  \PP_0  \bigl (  |   X_j | / \sigma_j  > R \sqrt{\log n} \bigr )
\notag \\
& \leq 6 \cdot \exp( -R^2 \log n /2 ).
\#
Similarly, under $\PP_{\vb}$,  Cauchy-Schwarz inequality   implies that
\#\label{eq:cauchy-112}
\bigl | \EE_{\PP_\vb} [ q_{j}(\bX) - q_{j}^* (\bX)  ]\bigr | ^2     & \leq \EE _{\PP_\vb} \bigl \{  [ q_{j }^* (\bX)]^2 \bigr \} \cdot  \PP_\vb  \bigl (  |    X_j | / \sigma_j  > R \sqrt{\log n} \bigr ).
\#
By \eqref{eq:entry_null_gaussianmix}, for any $j \in \supp(\vb)$, $X_j /\sqrt{ \sigma_j}$ can be written as   $\varepsilon + \varphi$, where $\varepsilon \sim  N(0, 1)$ and $\varphi$ is a Bernoulli random variable satisfying   $$
\PP\bigl [\varphi =  - ( 1- \nu)\cdot \beta v_j / \sqrt{ \sigma _j}  \bigr ] = \nu, \quad \text{and}\quad \PP( \varphi =  \nu \beta v_j / \sqrt{ \sigma_j} ) =  1- \nu. $$
Thus, using the  inequality $(a+b)^2 \leq 2 a^2 + 2b^2$, for any $j\in \supp(\vb)$, we have
\#\label{eq:q_j_star_psi}
\| q_j^*(\bX) \|_{\psi_1, \vb} \leq 2 \| \varepsilon^2 \|_{\psi_1} + 2 \| \varphi^2 \|_{\psi_1} \leq 2 \| \varepsilon^2 \|_{\psi_1} +2  \beta^2 \cdot v_j^2 /\sigma_j,
\#
where $\| \cdot \|_{\psi_1, \vb}$ denotes the the $\psi_1$-norm  under $\PP_{\vb}$.
Under the condition that
\#\label{eq:test2_condition}
  \max_{j\in [d]}  \nu(1-\nu) \beta^2 / \sigma_{j}  = \Omega[ \sqrt{\log ( d / \xi)\cdot \log n /n} ],
\# when $n$ is sufficiently large, by \eqref{eq:q_j_star_psi}  we have $\| q_j^* (\bX) \|_{\psi_1, \vb} \leq 3 \| \varepsilon ^2 \|_{\psi_1}$.
 Moreover, for any $j \notin \supp(\vb)$, since $q_j^* (\bX)\sim \chi_1^2 $, we  have $\| q_j^* (\bX) \|_{\psi_1, \vb} \leq   \| \varepsilon ^2 \|_{\psi_1}$. Thus, by \eqref{eq:cauchy-112}, there exist constants $C_1$ and $C_2$ such that
 \#\label{eq:cauchy-113}
 \bigl | \EE_{\PP_\vb} [ q_{j}(\bX) - q_{j}^* (\bX)  ]\bigr | ^2     & \leq C_1\cdot \exp(- C_2 \cdot R^2 \log n )
 \#
 for any $j\in [d]$,
 Combining \eqref{eq:cauchy-111} and \eqref{eq:cauchy-113},
we obtain that
\#\label{eq:truncation_bias}
\max_{j\in[d]} \Bigl \{ \bigl | \EE_{\PP_0} [ q_{j}(\bX) - q_{j}^* (\bX)  ]\bigr | + \bigl | \EE_{\PP_\vb} [ q_{j}(\bX) - q_{j}^* (\bX)  ] \bigr |  \Bigr \} \leq 1/ n.
\#
when $R$ is sufficiently large.

 Furthermore, since the query functions $\{ q_j \}_{j \in [d]}$ are bounded by $R^2  \cdot \log n$ in absolute value.
By Definition \ref{def::oracle},  the tolerance parameters  in the statistical query model are
   \#\label{eq::tol_para2}
   \tau_{q_{j}} = R ^2 \cdot \log  n \cdot  \sqrt{ 2\log(d/\xi) / n} ,~\text{for~all~} j \in [d].
   \#

   In the sequel, we conclude the proof by bounding the risk of the hypothesis test in \eqref{eq::test_fun2}.
  To simplify the  notation,
    we define  $\Lambda = R^2 \cdot \log n  \cdot  \sqrt{ 2\log(d/\xi) / n}$   and  $j^* = \argmin _{j\in[d] }  \sigma_{j}$. By \eqref{eq:test2_condition},
   it holds that $\max_{j\in [d]} \nu(1-\nu) \beta^2/ \sigma_{j}    = \nu(1-\nu) \beta^2/ \sigma_{j^*}  \geq 3 \Lambda.$
   Hence, for any   $\vb\in \cG( s)$ such that $j ^* \in \supp(\vb)$,
   combining \eqref{eq::diff_mean2} and \eqref{eq:truncation_bias} we obtain that
   	\#\label{eq::expect3}
   	\max_{j \in [d] } \Bigl \{   \EE_{\PP_\vb} [ q_{j}(\bX)   ] - \EE_{\PP_0} [ q_{j}(\bX) ]      \Bigr \} = \sup_{j\in \supp(\vb)}  \nu(1-\nu)\beta ^2/ \sigma_{j} -  2 / n \geq 2 \Lambda.
   	\#
   	Furthermore, for this $\vb$,
   	 by \eqref{eq::tol_para2} and \eqref{eq::expect3}  we have
\#\label{eq::bound_prob_oracle2}
& \overline{\PP}_{\vb } \Bigl( \max_{j\in[d]} Z_{q_{j}}  < 1+    \Lambda \Bigr)  \leq  \overline{\PP}_{\vb} \Bigl\{   \max_{j\in[d]}Z_{q_{j}}  < \EE_{\PP_{\vb}}  \bigl[q_{j^*} (\bX)\bigr]  -   \Lambda \Bigr\} \notag\\
&\quad \leq \overline{\PP}_{\vb} \Bigl\{ \mathbb{E}_{\PP_{\vb}} \bigl[q_{j^*}(\bX)\bigr]  - Z_{q_{j^*}}  >   \Lambda \Bigr\} =  \overline{\PP}_{\vb} \Bigl\{ \mathbb{E}_{\PP_{\vb}} \bigl[q_{j^*}(\bX)\bigr]  - Z_{q_{j^*}}  >  \tau_{q^*} \Bigr\}  .
\#
 By taking a union over $j\in [d]$, the last term in \eqref{eq::bound_prob_oracle2} can be  further upper bounded by
\#\label{eq::bound_prob_oracle22}
& \overline{\PP}_{\vb} \Bigl(  \max_{j\in[d]} Z_{q_{j}}  < 1+   \Lambda \Bigr) \leq  \overline{\PP}_{\vb} \Bigl\{ \mathbb{E}_{\PP_{\vb}}\bigl[q_{j^*} (\bX)\bigr] -Z_{q_{j^*}}  \geq   \tau_{q_{j^*}} \Bigr\} \notag\\
  & \quad \leq \overline{\PP}_{\vb} \biggl( {\bigcup_{j\in[d]}} \Bigl\{ \bigl|Z_{q_{j}}-\mathbb{E}_{\PP_{\vb}}\bigl[q_j(\bX)\bigr] \bigr| \geq \tau_{q_j} \Bigr\} \biggr).
\#
By the definition of the statistical query model  in \eqref{eq::query_2}  and \eqref{eq::bound_prob_oracle22}, we finally obtain
\$
  \bar{R}(\phi) &  =\overline{\mathbb{P}}_0 \Bigl(  \sup_{j\in[d]} Z_{q_{j}}  >  1+   \Lambda \Bigr) + \sup_{\vb \in \cG(s) } \PP_{\vb}\Bigl( \sup_{j\in[d]} Z_{q_{j}}<   1+   \Lambda \Bigr)\notag\\
&  \leq\overline{\mathbb{P}}_0\biggl( {\bigcup_{j\in[d]}} \Bigl\{ \bigl| Z_{q_{j}} - \mathbb{E}_{\overline{\PP}_0}\bigl[q_j(\bX)\bigr] \bigr | \geq  \tau_{q_{j} } \Bigr\} \biggr)   + \sup_{\vb \in \cG(s) } \overline{\mathbb{P}}_{\vb} \biggl( {\bigcup_{j\in[d]}} \Bigl\{ \bigl| Z_{q_{j}}-\mathbb{E}_{\mathbb{P}_{\vb}}\bigl[q_j(\bX)\bigr] \bigr|\geq  \tau_{q_{j} } \Bigr\} \biggr) \notag \\
&
\leq 2\xi,
 \$
which concludes the proof of Theorem \ref{thm::test2}.
\end{proof}

\section{Conclusion} 
In this paper, we characterize the computational barriers in high dimensional heterogeneous models, with  sparse Gaussian mixture model, mixture of sparse linear regressions, and sparse phase retrieval model as concrete instances. Under  an oracle-based computational model that is free of  computational hardness conjectures,  we  establish  computationally feasible minimax lower bounds for these models,  which quantify the minimum signal strength required for the existence of any algorithm that is both computationally tractable and statistically accurate.   Furthermore, we show  that there exist significant gaps between computationally feasible minimax risks and classical ones, which characterizes the fundamental tradeoffs between statistical accuracy and computational tractability in  the presence of data heterogeneity.   Interestingly, our results reveal a new but counter-intuitive phenomenon in heterogeneous data analysis that more data might lead to less computation complexity.

\begin{appendix}
\end{appendix}
\section{More General Upper Bounds for Gaussian Mixture Model}\label{ap:GMM}

In this appendix, we extend the hypothesis tests in  \S \ref{sec::upper_bound} to more general settings of Gaussian mixture detection, i.e., 
\#\label{eq:testing_prob}
H_0\colon \btheta \in {\cG}_0(\bSigma) ~~\text{versus}~~ H_1\colon \btheta \in  {\cG}_1(\bSigma, s, \gamma_n),
\#
as defined in \eqref{eq:wg0} and \eqref{eq:wg1}. Here  $\bSigma$ is assumed to be known. Equivalently, this  testing problem~can be written as 
\$
H_0 \colon \bX \sim N(\bmu , \bSigma) ~~\text{versus}~~
H_1 \colon \bX \sim \nu  N(\bmu_1, \bSigma) + (1- \nu) N(\bmu_2,\bSigma),
\$
in which $\Delta \bmu = \bmu_2 - \bmu_1$ is $s$-sparse.  Similar to the tests in \eqref{eq::test_fun1} and \eqref{eq::test_fun2}, we construct hypothesis tests based on the covariance matrix of $\bX$. 

Before  presenting the query functions, we first introduce a few quantities that will be used later. For each  index set $\cS \subseteq [d]$ with $| \cS| = s$, we define  the rescaled sparse unit  sphere as 
\$
\cU( \bSigma, \cS) = \bigl\{ \vb \in \RR^d \colon \vb^{\top} \bSigma^{-1} \vb =  1, \supp(\vb) = \cS\bigr\}.
\$
 For  any $\delta \in (0,1)$, we denote by $\cM( \delta;\bSigma, \cS)$   the minimal  $\delta$-covering subset of $\cU(\bSigma, \cS)$. That is to say, $\cM(\delta; \bSigma, \cS)$ satisfies the property that, for any $\vb \in \cU(\bSigma, \cS)$, there exists ${\vb'}\in \cM( \delta; \bSigma, \cS)\subseteq \cU (\bSigma, \cS)$ with $\supp({\vb'}) = \cS$ such that 
\$
( \vb - {\vb'}) ^\top  \bSigma^{-1} ( \vb - {\vb'}) \leq \delta^2.
\$
Moreover, the cardinality of $\cM(\delta; \bSigma, \cS)$ is the smallest among all subsets of $\cU (\bSigma, \cS)$ possessing  such  property.
It can be shown (see, e.g., \cite{vershynin2010introduction} for details) that  
\$
| \cM(\delta; \bSigma, \cS) | \leq ( 1+ 2/\delta )^{s}.
\$ 
 With slight abuse of notations, we denote by $\cM(\delta; \bSigma)$ the union of $\cM(\delta; \bSigma, \cS)$ over all index sets $\cS$ with $|\cS | = s$.  
To attain the information-theoretic lower bound~in Proposition \ref{prop::info_lower_bound}, we first consider~the following~sequence of query functions
\#\label{eq::query_fun110}
q_{\vb}(\xb ) =   \vb^\top \bSigma^{-1}  \xb \cdot \ind \bigl \{ | \vb^\top \bSigma^{-1} \xb | \leq R  \cdot \sqrt{ \log n}  \bigr \} ,    
\#
where $\vb \in \cM( 1/2; \bSigma)$ and  $R > 0$ is an absolute constant. Here  we apply truncation to have bounded queries. 
For query function $q_{\vb}$, let the random variable returned by the oracle be $Z_{q_{ \vb} }$. Given a realization $z_{q_{\vb}}$ of $ Z_{q_{\vb}}$ for each $\vb\in \cM(1/2;\bSigma)$,  we  query the oracle with another sequence of query functions
\#\label{eq::query_fun11}
\overbar q_{\vb} (\xb) = (\vb^\top \bSigma^{-1}  \xb - z_{q_{\vb}} )^2\cdot    \ind \bigl \{ | \vb^\top \bSigma^{-1} \xb   | \leq   R  \cdot \sqrt{ \log n}   \bigr \}
\#
where  $ \vb \in \cM( 1/2; \bSigma)$.
Let  $Z_{\overbar q_{\vb}} $ be the random variable returned by the oracle for $\overbar q_{\vb}$ and $z_{\overbar q_{\vb}}$ be its realization.
In this~case, the query complexity is 
\$
T =2   |\cM( 1/2; \bSigma ) |  \leq  2 \cdot 5^{ s } \cdot {d \choose { s }},
\$
 and $\eta(\cQ_{\mA}) = \log T\leq   s \log (5d)$. Finally,   we define the   test function as
\#\label{eq::test_fun11}
\ind \Bigl\{\sup_{\vb\in \cM(1/2; \bSigma) } z_{\overbar q_{ \vb} }\geq  1 + 16 R^2 \cdot  \log n \cdot  \sqrt{  2\bigl[s \log  (5d) + \log ( 1/\xi) \bigr]/ n } \Bigr\}. 
\#
 The subsequent theorem proves that the information-theoretic  lower bound in Proposition~\ref{prop::info_lower_bound} is tight within $ {\cG}_0(\bSigma)$ and $ {\cG}_1(\bSigma, s, \gamma_n)$.

\begin{theorem}\label{thm::test_full1}
We consider the   mixture detection problem in \eqref{eq:testing_prob}. Under the assumption that 
\# \label{eq:signal_general1}
\rho(\btheta)  =\nu(1- \nu) \Delta\bmu ^\top \bSigma^{-1} \Delta \bmu \geq \gamma_n = \Omega \Bigl\{ \log n \cdot  \sqrt{  \bigl[s \log  (5d) + \log ( 1/\xi)\bigr]/ n } \Bigr\},
\#
 the test function in \eqref{eq::test_fun11}, which is denoted by $\phi$, satisfies 
\$
\sup_{\bSigma}  \Bigl[\sup_{\btheta\in  {\cG}_0(\bSigma)} \overline{\PP}_{\btheta}(\phi = 1) + \sup_{\btheta\in {\cG}_1(\bSigma, s, \gamma_n)} \overline{\PP}_{\btheta}(\phi = 0) \Bigr] \leq 2\xi.
\$
\end{theorem}
\begin{proof}
We first  note that, due to the truncation in \eqref{eq::query_fun110}, for any $\vb \in \cM(1/2, \bSigma)$, we have $| z_{q_{\vb}} |  \leq R\cdot \sqrt{\log n}$, which implies that $\overbar q_{\vb}$ is bounded by $4 R^2 \cdot \log n$. 
In the sequel, 
for notational simplicity, we define $\cQ = \{ q_{\vb}, \overbar q_{\vb} \colon \vb \in \cM(1/2; \bSigma) \}$ and   
\$
\cE =   \bigcap _{q \in \cQ } \Bigl\{ \bigl| Z_{q  }  - \EE _{\PP_{\btheta}} \bigl[ q (\bX) \bigr] \bigr| \leq  \tau_{q } \Bigr\},
\$
where    $\tau_q$ is the tolerance parameter of the statistical query~model, which satisfies 
\#\label{eq::tol_param3}
\tau_{q} \leq 4 R^2\cdot  \log n \cdot  \sqrt{ 2\bigl[ \log T  +\log(1/\xi) \bigr]/ n } \leq 4 R^2\cdot  \log n\cdot  \sqrt{ 2 \bigl[ s \log (5d ) + \log(1/\xi) \bigr]/ n}  
\#
for all $ q \in \cQ.$
  By Definition \ref{def::oracle}, for any $\btheta \in \cG_0 (\bSigma) \cup \cG_1(\bSigma, s, \gamma_n)$, we have  $\PP_{\btheta} ( \cE  ) \geq 1 - \xi$.

 In the sequel, we prove that both the type-I and type-II errors of the test function in \eqref{eq::test_fun11} are bounded by $\xi$. More specifically, we prove this by showing that  the test function takes value zero on $\cE$ under $H_0$ and  one under $H_1$.  
 
 Similar to the proof of Theorem \ref{thm::test1}, to characterize the effect of truncation, we define 
 \#\label{eq:queryfun_star}
 q_{\vb}^* (\xb ) =   \vb^\top \bSigma^{-1}  \xb  , \quad   \overbar q_{\vb} ^*(\xb) = (\vb^\top \bSigma^{-1}  \xb - z_{q_{\vb}} )^2 \quad \text{for all} \quad \vb \in \cM( 1/2 , \bSigma),
 \#
 where $z_{q_{\vb}}$ is the realization of $Z_{q_{\vb}}$. We first show that, under the assumption   in \eqref{eq:signal_general1},  $q_{\vb}^*(\bX) $ and $\overbar q_{\vb}^*(\bX)$ are close to   $q_{\vb} (\bX) $ and $\overbar q_{\vb} (\bX)$ respectively in expectation.

 Hereafter, we denote $\max \{  \| \bSigma^{-1/2 } \bmu\|_2,\| \bSigma^{-1/2 } \bmu_1\|_2, \| \bSigma^{-1/2 } \bmu_2 \|_2 \}$ by $\Upsilon_0$. 
 Under the null hypothesis, for any $\btheta \in \cG_0( \bSigma)$ and   any $\vb\in \cM(1/2; \bSigma)$,  since $\vb^\top \bSigma ^{-1} \vb = 1$, it holds that  $\vb^\top \bSigma^{-1} \bX \sim N( \vb ^\top \bSigma ^{-1} \bmu,  1 )$. This  implies 
 $\EE _{\PP_{\btheta}} [q_{\vb}^* (\bX )] =  
 \vb^\top \bSigma^{-1} \bmu  $ and 
 \#\label{eq:expected_null_bar}
 \EE _{\PP_{\btheta}} \bigl[ \overbar q_{\vb} ^* (\bX) \bigr]  = 1 + (\vb ^\top \bSigma^{-1} \bmu - z_{q_{\vb}})^2.
\#
Since $\vb^\top \bSigma ^{-1} \vb = 1$, we have $| \vb ^\top \bSigma^{-1} \bmu | \leq \| \bSigma^{-1/2} \bmu\|_2 \leq \Upsilon_0$. 
Note that  when $n$ is sufficiently large, $\tau_q$ in \eqref{eq::tol_param3} is bounded by one. Thus, 
by the definition of the statistical query model, on event $\cE$ we have $| Z_q - \EE_{\PP_{\theta} } [ q(\bX)] | \leq 1$ for all $q \in \cQ$.   
Moreover, by Cauchy-Schwarz inequality, we have 
\#\label{eq:cauchy_qv}
\bigl | \EE_{\PP_{\btheta} } [ q_{\vb}(\bX ) - q_{\vb}^* (\bX)] \bigr | ^2 \leq \EE_{\PP_{\btheta} } \bigl \{  [ q_{\vb}^*(\bX) ]^2 \bigr \}  \cdot \PP_{\btheta} \bigl \{ |q_{\vb}^*(\bX )  |  > R \cdot \sqrt{ \log n}  \bigr \}.
\#
Since   $q_{\vb}^* (\bX)$ is a Gaussian random variable，    there exist  constants $C_1$ and $C_2$ such that 
\$
\PP_{\btheta} \bigl \{ |q_{\vb}^*(\bX )  |  > R \cdot \sqrt{ \log n}  \bigr \} \leq C_1 \cdot \exp( -C_2 \cdot R^2  \log n).
\$
Notice that $\EE_{\PP_{\btheta} } \bigl \{  [ q_{\vb}^*(\bX) ]^2 \bigr \} \leq 1 + \Upsilon_0^2 $. Thus, by \eqref{eq:cauchy_qv}, we have 
\#\label{eq:bound_qv_bias}
\bigl | \EE_{\PP_{\btheta} } [ q_{\vb}(\bX ) - q_{\vb}^* (\bX)] \bigr | ^2 \leq C_1 \cdot   ( 1+ \Upsilon_0^2  ) \cdot  \exp( -C_2 \cdot R^2  \log n).
\#
In addition,  for $\overbar q_{\vb}^* (\bX)$, Cauchy-Schwarz inequality implies that 
\#\label{eq:cauchy_qv2}
\bigl | \EE_{\PP_{\btheta} } [ \overbar q_{\vb}(\bX ) - \overbar q_{\vb}^* (\bX)] \bigr | ^2 &  \leq \EE_{\PP_{\btheta} } \bigl \{  [ \overbar  q_{\vb}^*(\bX) ]^2 \bigr \}  \cdot \PP_{\btheta} \bigl \{ |q_{\vb}^*(\bX )  |  > R \cdot \sqrt{ \log n}  \bigr \} \notag \\
& \leq C_1 \cdot \EE_{\PP_{\btheta} } \bigl \{  [ \overbar  q_{\vb}^*(\bX) ]^2 \bigr \} \cdot  \exp( -C_2 \cdot R^2  \log n) .
\#
 Since $z_{q_{\vb}} $ is bounded, by \eqref{eq:queryfun_star}, 
 $  \EE_{\PP_{\btheta} } \bigl \{  [ \overbar  q_{\vb}^*(\bX) ]^2 \bigr \}$ is also bounded. Thus, \eqref{eq:cauchy_qv2} implies that 
 \#\label{eq:bound_qv_bias2}
 \bigl | \EE_{\PP_{\btheta} } [ q_{\vb}(\bX ) - q_{\vb}^* (\bX)] \bigr | ^2 \leq C_3 \cdot  \exp( -C_2 \cdot R^2  \log n),
 \#
 where $C_3$ is an absolute constant depending on $\Upsilon_0$.

 Furthermore, under the alternative hypothesis with parameter  $\btheta = ( \bmu_1, \bmu_2, \bSigma)\in \cG_1( \bSigma,  s, \gamma_n)$,  for any $\vb \in \cM(1/2; \bSigma)$, it holds that 
 \#\label{eq::dist_query_function00}
 q_{\vb}^*  (\bX)  = \vb^{\top} \bSigma^{-1} \bX \sim  \nu N( \vb^{\top} \bSigma ^{-1}\bmu_1, 1) + (1- \nu) N( \vb^\top \bSigma^{-1} \bmu_2, 1).
 \#
By Cauchy-Schwarz inequality, we have  $| \vb^\top \bSigma^{-1} \bmu_1 |  \leq \| \bSigma^{-1/2} \bmu_1 \|_2 \leq \Upsilon_0$ and $| \vb^\top \bSigma^{-1} \bmu_2 |  \leq \| \bSigma^{-1/2} \bmu_2\|_2 \leq \Upsilon_0$,
which implies that $q_{\vb}^* (\bX)$ is a sub-Gaussian random variable, i.e.,
\$
\PP_{\btheta} \bigl \{ |q_{\vb}^*(\bX )  |  > t  \bigr \} \leq C_4 \cdot \exp( -C_5 \cdot t^2 ) 
\$
for any $t >  0$, where $C_4$ and $C_5$ are absolute constants depending on $\Upsilon_0$. 
Moreover, since $| Z_{q_{\vb} } - \EE_{\PP_{\btheta}} [ q_{\vb} (\bX)] |  \leq 1 $ on $\cE$,    we have 
\$
|Z_{q_{\vb}} |  \leq 1 +\bigl |  \EE_{\PP_{\btheta}} [ q_{\vb} (\bX)] \bigr | \leq 1 +   \EE_{\PP_{\btheta}} [ |  q_{\vb} (\bX) | ]  \leq 1 + \sqrt{ \pi/2 } +  \Upsilon_0 .
\$
Thus, there exists an absolute constant $C_6$ depending on $\Upsilon_0$ such that $$\max \Bigl  \{ \EE_{\PP_{\btheta} } \{  [q_{\vb}^*(\bX)]^2 \},  \EE_{\PP_{\btheta} } \{    [\overbar q_{\vb}^*(\bX)]^2 \} \Bigr \}  \leq C_6.$$ 
Thus, similar to \eqref{eq:cauchy_qv} and \eqref{eq:cauchy_qv2}, by Cauchy-Schwarz inequality,  we have 
\#\label{eq:cauchy_qv3}
&\max \Bigl\{ \bigl | \EE_{\PP_{\btheta} } [ q_{\vb}(\bX ) - q_{\vb}^* (\bX)] \bigr | ^2, ~\bigl | \EE_{\PP_{\btheta} } [ \overbar q_{\vb}(\bX ) - \overbar q_{\vb}^* (\bX)] \big |^2 \Bigr \} \notag \\
&\quad  \leq \max \Bigl (  \EE_{\PP_{\btheta} } \bigl \{  [ q_{\vb}^*(\bX) ]^2 \bigr \},~\EE_{\PP_{\btheta} } \bigl \{  [ \overbar q_{\vb}^*(\bX) ]^2 \bigr \} \Bigr )   \cdot \PP_{\btheta} \bigl \{ |q_{\vb}^*(\bX )  |  > R \cdot \sqrt{ \log n}  \bigr \} \notag \\
& \quad \leq C_4 \cdot  C_6 \cdot \exp( - C_5 \cdot R^2 \log n).
\#
Combining \eqref{eq:bound_qv_bias}, \eqref{eq:bound_qv_bias2}, and \eqref{eq:cauchy_qv3}, 
we conclude that, when $R$ is sufficiently large, we have 
\#\label{eq:truncation_bias_final}
 \max \Bigl \{  \bigl | \EE_{\PP_{\btheta} } [ q_{\vb}(\bX ) - q_{\vb}^* (\bX)] \bigr |, ~ \bigl | \EE_{\PP_{\btheta} } [ \overbar q_{\vb}(\bX ) -\overbar  q_{\vb}^* (\bX)] \bigr | \Bigr \} \leq 1/n 
\#
for any $\btheta \in \cG_0(\bSigma ) \cup \cG_1(\bSigma, s, \gamma_n) $ and any $\vb \in \cM(1/2, \bSigma)$. 

Now we consider the value of the test function in \eqref{eq::test2_selection} under $H_0$.
Conditioning on event $\cE$, for 
 $\overbar q_{\vb}$ in \eqref{eq::query_fun11},  by \eqref{eq:expected_null_bar} and  \eqref{eq:truncation_bias_final}, we have 
 \#\label{eq:null_test_fun_val1}
 Z_{\overbar q_{\vb}} & \leq \EE_{\PP_{\btheta }} [ \overbar q_{\vb} (\bX) ] + \tau _{\overbar q_{\vb}} \leq  \EE_{\PP_{\btheta }} [ \overbar q_{\vb} ^* (\bX) ] + \tau _{\overbar q_{\vb}} + 1/n  \notag \\
 & \leq 1 + (\vb^\top \bSigma^{-1} \bmu - z_{q_{\vb}})^2 + \tau _{\overbar q_{\vb}} + 1/n \notag   \\
 & = 1 + \bigl \{ \EE_{\PP_{\btheta}} [ q_{\vb}^* (\bX)] - z_{q_{\vb}} \bigr \} ^2  + \tau _{\overbar q_{\vb}} + 1/n  .
 \#
 Note that $\tau_{q} \leq 1$ for all $q \in \cQ$  when $n$ is sufficiently large. Applying  inequality $(a+b)^2 \leq 2a^2 + 2b^2 $ to \eqref{eq:null_test_fun_val1} we have 
\#\label{eq:null_test_fun_val2} 
	Z_{\overbar q_{\vb}} & \leq 1 + 2  \bigl \{ \EE_{\PP_{\btheta}} [ q_{\vb}  (\bX)] - Z_{q_{\vb}} \bigr \} ^2 + 2 \bigl \{ \EE_{\PP_{\btheta}} [ q_{\vb}  (\bX)] -\EE_{\PP_{\btheta}} [ q_{\vb}^*  (\bX)]  \bigr \} ^2 +\tau _{\overbar q_{\vb}} + 1/n  \notag \\
	& \leq  1+ 2 \tau_{q_{\vb}}^2 + 2 / (n^2) + \tau _{\overbar q_{\vb}} + 1/n \leq   1+ 2  \tau_{q_{\vb}} +\tau _{\overbar q_{\vb}} + 3/n .
\#
Thus, by \eqref{eq::tol_param3} and \eqref{eq:null_test_fun_val2}, we obtain   
\$ 
 	Z_{\overbar q_{\vb}} < 1 + 16 R^2\cdot  \log n\cdot  \sqrt{ 2 \bigl[ s \log (5d ) + \log(1/\xi) \bigr]/ n}  
\$ 
for any $\vb \in \cM(1/2 , \bSigma)$, 
which implies that the test function in \eqref{eq::test_fun11} takes value zero on $\cE$ and that the  type-I error   is bounded by $\xi$.

 It remains to bound the type-II error. 
 Under $H_1$ with parameter $\btheta = ( \bmu_1, \bmu_2, \bSigma)$,   by \eqref{eq:queryfun_star} and \eqref{eq::dist_query_function00},  we have 
 \#
 \EE_{\PP_{\btheta}} [ q_{\vb}^*(\bX)] & =  \nu \cdot   \vb^{\top} \bSigma^{-1} \bmu_1 + (1 - \nu) \cdot  \vb^{\top} \bSigma^{-1} \bmu_2 , \notag \\
   \EE_{\PP_{\btheta}} [ \overbar q_{\vb}^*(\bX)] & =  \textrm{Var} _{\PP_{\btheta}}[ q_{\vb}(\bX) ]+ \bigl \{ z_{q_{\vb}} -\EE_{\PP_{\btheta}} [ q_{\vb}^*(\bX)]   \bigr \} ^2    \geq  1 + \nu(1- \nu) \cdot | \vb^\top \bSigma ^{-1} \Delta \bmu | ^2, \label{eq:alter_hypo_means2}
\#
for any $\vb \in \cM(1/2; \bSigma)$. Here $ \textrm{Var} _{\PP_{\btheta}} $ denotes the variance under $\PP_{\btheta}$ and $\Delta \bmu = \bmu_2 - \bmu_1$. 
 Furthermore, by the construction of $\cM(1/2; \bSigma)$,   there exists $\vb_1 \in \cM(1/2; \bSigma)$ such that 
\#\label{eq::covering_rate}
( \vb_0 - \vb_1) ^{\top} \bSigma ^{-1} (\vb_0 - \vb_1) \leq 1/4, ~~\text{where}~\vb_0 = \Delta \bmu / \sqrt{\Delta \bmu ^{\top}\bSigma^{-1} \Delta \bmu}.
\#
Note that $\vb_1^{\top} \bSigma ^{-1} \vb_1 = \vb_0^\top \bSigma^{-1} \vb_0 = 1$. 
Then by \eqref{eq::covering_rate}, we have  $\vb_0 ^\top \bSigma^{-1} \vb_1 \geq 7/8$. Setting $\vb  = \vb_1$ in \eqref{eq:alter_hypo_means2}, we obtain 
\#\label{eq::compute_q_v_exp_alt}
\EE_{\PP_{\btheta}} \bigl[\overbar q_{\vb_1 }^*(\bX)\bigr] & \geq 1 +  \nu(1- \nu) \cdot  | \vb_1   ^\top \bSigma ^{-1} \Delta \bmu |^2 \notag\\
&  = 1 +  \nu(1- \nu) \cdot | \vb_1 \bSigma^{-1} \vb_0 |^2 \cdot ( \Delta \bmu^\top \bSigma^{-1} \Delta \bmu)
 \geq  1 + 7/8 \cdot \rho (\btheta), 
\#
where we denote $ \nu(1- \nu) \cdot  \Delta \bmu^\top \bSigma^{-1} \Delta \bmu$ by $\rho(\btheta)$. 

Now we consider the value of the test function in \eqref{eq::test_fun11} 
under the condition  that  $ \gamma_n  = \Omega \{ \log n \cdot \sqrt{[ s \log (5d ) + \log(1/\xi) ]/n} \} $. On event $\cE$, combining  \eqref{eq::tol_param3},  \eqref{eq:truncation_bias_final} and  \eqref{eq::compute_q_v_exp_alt}  we have 
 \#\label{eq:zq_alt_val}
 Z_{\overbar q_{\vb_1 }} &  \geq \EE_{\PP_{\theta}} [ \overbar q_{\vb_1 }(\bX)] - \tau_{\overbar q_{\vb_1 }}  \geq  \EE_{\PP_{\theta}} [ \overbar q_{\vb_1 }^* (\bX)] - 1/ n -  \tau_{\overbar q_{\vb}} \notag \\
 & \geq 1 + 7/ 8  \cdot \rho (\btheta) - 1/ n -   \tau_{\overbar q_{\vb}} \notag \\
 & \geq 1 + 7/ 8  \cdot \rho (\btheta)  - 5 R^2\cdot  \log n\cdot  \sqrt{ 2 \bigl[ s \log (5d ) + \log(1/\xi) \bigr]/ n} .
 \#
   Therefore, when $\rho(\btheta ) \geq 24  R^2\cdot  \log n\cdot  \sqrt{ 2 \bigl[ s \log (5d ) + \log(1/\xi) \bigr]/ n} $, \eqref{eq:zq_alt_val} implies that, for  any $\btheta \in \cG_1(\bSigma, s, \gamma_n)$,  we have 
   \$
    \sup_{\vb\in \cM(1/2; \bSigma)} Z_{\overbar q_{\vb_1 }}  \geq 1 + 16 R^2\cdot  \log n\cdot  \sqrt{ 2 \bigl[ s \log (5d ) + \log(1/\xi) \bigr]/ n} 
   \$ 
   with probability one 
on event $\cE$. 
   That is,  the test function in \eqref{eq::test_fun11} takes value one on $\cE$. Hence the type-II error is upper bounded by $\xi$, which concludes the proof of this theorem.
\end{proof}

In the sequel, we present a computationally tractable hypothesis test for the detection problem in \eqref{eq:testing_prob}.   
Similar to the test function in \eqref{eq::test_fun11}, we also establish a test function based on the covariance of $\bX$.  For $j \in [d]$, we consider   query function   
\#  \label{eq::query_general_fun20}
q_j (\xb) =  x_j  / \sqrt{  \sigma_{j}  } \cdot \ind \{ | x_j /\sqrt{\sigma_j }  | \leq R \cdot \sqrt{\log n }  \},
\#
 where $\sigma_{j}$ is the $j$-th diagonal element of $\bSigma$ and $R>0$ is an absolute constant. Let $z_{q_j}$ be the realization of $Z_{q_j}$ returned by the oracle. We construct  another  query function   
\#\label{eq::query_general_fun2}
\overbar q_j (\xb ) = (x_j /\sqrt{\sigma_j}- z_{q_j} )^2 \cdot \ind \{ | x_j /\sqrt{\sigma_j }  | \leq R \cdot \sqrt{\log n }  \}.
\#
 Now the query complexity is $T = 2 d$~and~we have $\eta(\cQ_{\mA}) = \log (2d)$.~We define the test function as
\#\label{eq::test_fun21}
\ind \Bigl[\max _{j\in [d]} z_{\overbar q_j}   \geq  1 + 16 R^2\cdot \log n \cdot  \sqrt{ \log (2 d / \xi)/ n} \Bigr],
\#
 where  $\varepsilon \sim N(0,1)$ and   $z_{\overbar q_j}$  is   the realizations of   $Z_{\overbar q_j}$ obtained from the oracle.  The following theorem proves that the hypothesis  test  defined above is asymptotically powerful if $\gamma_n = \Omega ( \log n \cdot \sqrt{s^2\log d/ n}  )$ and the energy of $\Delta \bmu$ is evenly spread over its support.

\begin{theorem}\label{thm::test_full2}
We consider the sparse mixture detection problem in \eqref{eq::reduced_testing}. We denote by $\Delta \mu_j $ the $j$-th element of $\Delta \bmu$ for $j\in [d]$.  If
\#\label{eq:wnvsd}
 \max_{j\in [d]}   \nu(1-\nu) \cdot (\Delta \mu_j )^2/ \sigma_{j}    = \Omega \bigl[\log n \cdot  \sqrt{  \log (2 d / \xi) / n}   \bigr],
\#
then for $\phi$ being the test function in \eqref{eq::test_fun21}, we have
\$
\sup_{\bSigma} \Bigl[\sup_{\btheta\in  {\cG}_0(\bSigma)} \overline{\PP}_{\btheta}(\phi = 1) + \sup_{\btheta\in  {\cG}_1(\bSigma, s, \gamma_n)} \overline{\PP}_{\btheta}(\phi = 0) \Bigr] \leq 2\xi.
\$
\end{theorem}
Recall that $\lambda_*$ and $\lambda^*$ are defined in \eqref{eq:weigen}. If the energy of $\Delta \bmu$ is evenly spread over its~support, that is, $\| \Delta \bmu \|_{\infty}$ is of the same order as $\| \Delta \bmu \|_2 / \sqrt{s}$, then  since  
\$
{\lambda_*} \leq  \min_{j\in [d]}  \sigma_j \leq  \max _{j\in [d]} \sigma_j \leq {\lambda^*},
\$
 the condition in \eqref{eq:wnvsd} is equivalent to $\gamma_n = \Omega[ \log n \cdot \sqrt{ s ^2\log ( d / \xi)/n}  ]$. Setting $\xi = 1/d$, we conclude that the test function in \eqref{eq::test_fun21} is asymptotically powerful if $$\gamma_n = \Omega \bigl ( \log n \cdot  \sqrt{s^2\log d/ n}  \bigr ).$$ 

\begin{proof} 
Similar to the proof of Theorem \ref{thm::test_full1}, for notational simplicity we define 
\$
\cQ = \bigl\{ q_{j }, \overbar q_{j } \colon j \in [d] \bigr\}\quad \text{and}\quad \cE =  \bigcap _{q\in \cQ} \Bigl\{ \bigl| Z_{q}  - \EE _{\PP_{\btheta}} \bigl[ q (\bX) \bigr] \bigr| \leq \tau_{q } \Bigr\},
\$
where  the tolerance parameter  $\tau_q$ satisfies 
\#\label{eq::tol_param2}
\tau_{q} \leq  4R^2 \cdot \log n \cdot  \sqrt{ 2\bigl[\log (2d) +\log(1/\xi)\bigr]/ n }  , ~\text{for~all~} q \in \cQ.
\#
Here \eqref{eq::tol_param2} follows from the fact that both $q_j$ and  $ \overbar q_j $ are bounded by $4 R^2 \cdot \log n $ in absolute value for any $j\in [d]$. 
  By Definition \ref{def::oracle}, for any  $\btheta \in \cG_0 (\bSigma) \cup \cG_1(\bSigma, s, \gamma_n)$, we have $\PP_{\btheta} ( \cE  ) \geq 1 - \xi$.
  
  Following the same proof strategy,  in the sequel, we show that  the hypothesis test   in \eqref{eq::test_fun21} is correct on event $\cE$. That is, we prove that the test function takes value zero on $\cE$ under $H_0$ and  one under $H_1$, which implies that the risk is upper bounded by $2 \xi$.     
 
 To this end, we first quantify the bias of truncation in \eqref{eq::query_general_fun20} and \eqref{eq::query_general_fun20}. Specifically, for any $j \in [d]$, we define 
 \#\label{eq::query_star_2}
q_j^* (\xb) = x_j / \sigma_j, \quad \overbar q_j^* (\xb) = (x_j / \sqrt{ \sigma_j}   - z_{q_j})^2 .
\#
Moreover, we assume that $n$ is sufficiently large such that the right-hand side  in \eqref{eq::tol_param2} is less than one. In this case, $\tau_q \leq 1$ for any $q \in \cQ$.

 For any $\btheta =  (\bmu, \bmu,\bSigma) \in \cG_0(\bSigma)$ and $j\in [d]$,  since the marginal distribution of $ X_j$ is $N(\mu_j, \sigma_j)$ under $\PP_{\btheta}$,   we have $q_{j}^* (\bX ) \sim N( \mu_j /\sqrt{\sigma_j}, 1)$. Here $\mu_j$ is the $j$-th coordinate of $\bmu$. For ease of presentation, let 
 \$
 \Upsilon_0 = \max _{j\in [d] } \bigl ( \max \{  | \mu_j / \sqrt{ \sigma_j } |,~ |   \mu_{1, j} / \sqrt{ \sigma_j } |, ~ |   \mu_{2, j} / \sqrt{ \sigma_j } |\}   \big ) ,
 \$
 where $\mu_{1, j}$ and $\mu_{2, j}$ are the $j$-th entries of $\bmu_1$ and $\bmu_2$, respectively.
 In addition, by triangle inequality, we have 
 \#\label{eq:bound_zq}
 | z_{q_j} | \leq 1 + |\EE_{\PP_{\btheta} } [ q_j (\bX) ] | \leq 1 + \EE_{\PP_{\btheta}} [ | q_j (\bX)  | ] \leq 1 + \sqrt{\pi/2} + \Upsilon_0.
 \#
 
 Moreover, for any $\btheta = (\bmu_1, \bmu_2, \bSigma) \in \cG_1(\bSigma , s, \gamma_n)$, we have   
 \$
 q_j^* (\bX) = X_j /\sqrt{\sigma_j} \sim  \nu N (  \mu_{1,j}  /\sqrt{\sigma_j}  , 1 )   + (1- \nu) N ( \mu_{2,j}  /\sqrt{\sigma_j}  ,1),~~\forall  j \in [d].\$
Thus, \eqref{eq:bound_zq} also holds under the alternative hypothesis. 
Hence, there exists an absolute constant $C_1$ such that, for any $\btheta \in \cG_0(\bSigma) \cup \cG_1(\bSigma , s, \gamma_n)$, we have 
 \#\label{eq:bound_quad_moment}
\max \Bigl (  \EE_{\PP_{\btheta} } \bigl \{ [ q_j^* (\bX)]^2  \bigr \} , ~ \EE_{\PP_{\btheta} } \bigl \{ [ \overbar q_j^* (\bX)]^2  \bigr \}  \Bigr )  \leq C_1,~~\forall  j \in [d].
 \# 
 It is also easy to see that, for any $\btheta \in \cG_0(\bSigma) \cup \cG_1(\bSigma , s, \gamma_n)$,  $q_j^*(\bX)$ is a sub-Gaussian random variable under $\PP_{\btheta}$, i.e., there exists constants $C_2$ and $C_3$ such that 
 \#\label{eq:subgaussian_qj}
 \PP_{\btheta} \bigl  [ |  q_j^*(\bX) | \geq t \bigr ] \leq C_2 \cdot \exp(-C_3 \cdot t^2 )
 \#
for any $t >0$.  Setting $t = R\cdot \sqrt{\log n}$ in \eqref{eq:subgaussian_qj} and using Cauchy-Schwarz inequality, we have 
\#\label{eq:bound_truncation_qj}
&\max \Bigl\{ \bigl | \EE_{\PP_{\btheta} } [ q_{j}(\bX ) - q_{j}^* (\bX)] \bigr | ^2, ~\bigl | \EE_{\PP_{\btheta} } [ \overbar q_{j}(\bX ) - \overbar q_{j}^* (\bX)] \big |^2 \Bigr \} \notag \\
&\quad  \leq \max \Bigl (  \EE_{\PP_{\btheta} } \bigl \{  [ q_{j}^*(\bX) ]^2 \bigr \},~\EE_{\PP_{\btheta} } \bigl \{  [ \overbar q_{j}^*(\bX) ]^2 \bigr \} \Bigr )   \cdot \PP_{\btheta} \bigl [  |q_{j}^*(\bX )  |  > R \cdot \sqrt{ \log n}  \bigr ] \notag \\
& \quad \leq C_1 \cdot  C_2 \cdot \exp( - C_3 \cdot R^2 \log n),
\# 
 where the last inequality follows from \eqref{eq:bound_quad_moment} and \eqref{eq:subgaussian_qj}. We remark that  absolute constants $C_1$, $C_2$, and $C_3$ depend on $\Upsilon_0$. Hence, setting $R$ to be a sufficiently large constant  in \eqref{eq:bound_truncation_qj}, we obtain  that, for all $\btheta \in \cG_0(\bSigma) \cup \cG_1(\bSigma , s, \gamma_n)$, it holds that 
 \#\label{eq:bound_truncation_qj2}
 \max \Bigl\{ \bigl | \EE_{\PP_{\btheta} } [ q_{j}(\bX ) - q_{j}^* (\bX)] \bigr | ,  ~\bigl | \EE_{\PP_{\btheta} } [ \overbar q_{j}(\bX ) - \overbar q_{j}^* (\bX)] \big |  \Bigr \} \leq 1/n.
 \#
 
 In the sequel, we consider the value of the test function in \eqref{eq::test_fun21} under the null and alternative hypotheses separately.  For any $\btheta \in \cG_0(\bSigma)$ and for any $j \in [d]$, under $\PP_{\btheta}$ we have 
 \#\label{eq:z_qj_upper1}
 Z_{\overbar q_j } & \leq \EE_{\PP_{\btheta} } [\overbar q_j (\bX)  ] + \tau_{\overbar q_j}\leq \EE_{\PP_{\btheta} } [\overbar q_j ^*  (\bX)  ]  + 1/ n + \tau_{\overbar q_j} \notag \\
 & = \bigl| z_{q_j}  - \EE _{\PP_{\btheta}}  [ q_j(\bX) ] \bigr|^2 + 1 +  \tau_{\overbar q_j} + 1/n,
 \#
 where the first inequality follows from Definition \ref{def::oracle} and the second inequality follows from \eqref{eq:bound_truncation_qj2}. Using   $(a+b)^2 \leq 2 a^2 + 2b^2 $ and combining \eqref{eq:bound_truncation_qj2} and \eqref{eq:z_qj_upper1}, we have 
\#\label{eq:z_qj_upper2}
  Z_{\overbar q_j }&  \leq   2 \bigl|  \EE _{\PP_{\btheta}}  [ q_j(\bX)- q_j^*(\bX) ] \bigr|^2 + 3/n + 1+  \tau_{\overbar q_j} \notag \\
&\quad  \leq 1 +2  \tau_{q_j}^2 + 3/ n + \tau_{\overbar q_j}  \leq 1 + 2 \tau_{q_j}+ \tau_{\overbar q_j}+3/n  .
\#
 Thus, combining  \eqref{eq::tol_param2} and \eqref{eq:z_qj_upper2}, on event $\cE$ we obtain  
 \$
\max_{j\in [d]} Z_{\overbar q_j } < 1+ 16R^2 \cdot \log n \cdot  \sqrt{ 2 [\log (2d) +\log(1/\xi) ]/ n } .
 \$
 Thus, for any $\btheta \in \cG_0(\bSigma)$, the type-I error of the test function in \eqref{eq::test_fun21} is no more than $\xi$. 
 
 It remains to bound the type-II error. For any $\btheta \in 
 \cG_1( \bSigma , s, \gamma_n)$,   for $\overbar q_j^*(\xb)  $ defined in  \eqref{eq::query_star_2}, by direct copmutation, we have 
  \#\label{eq::exp_second_query_alt}
  \EE_{\PP_{\btheta}} \bigl[ \overbar q_j^*(\bX) \bigr]  & = \EE_{\PP_{\btheta}} \Bigl(\bigl\{  q_j^* (\bX) - \EE_{\PP_{\btheta}}   [q_j^* (\bX) ] \bigr\}^2 \Bigr)  + \bigl| z_{q_j} -  \EE_{\PP_{\btheta}} \bigl[ q_j^* (\bX)\bigr]\bigr|^2 \notag \\
  &\geq 1 +  \nu(1- \nu) ( \Delta \mu_j)^2/ \sigma_j  ,
  \#
where $\Delta\mu _j$ is the $j$-th element of $\Delta \bmu = \bmu_2 - \bmu_1$  
 Let  $j^* = \argmax _{j\in[d] } \{   (\Delta \mu_j)^2 /  \sigma_{j} \} $.  Then by \eqref{eq:bound_truncation_qj2},  \eqref{eq::exp_second_query_alt},   and the definition of $\cE$, we obtain
\#\label{eq:test_zq_upper_final}
Z_{\overbar q_{j^*}} &\geq  \EE_{\PP_{\btheta}}  \bigl[ \overbar q_{j^*}^* (\bX) \bigr] - \tau_{\overbar q_{j^*}}  \geq   \EE_{\PP_{\btheta}}  \bigl[ \overbar q_{j^*} (\bX) \bigr] -  \tau_{\overbar q_{j^*}} - \bigl | \EE_{\PP_{\btheta} } [ \overbar q_{j^*}(\bX ) - \overbar q_{j^*}^* (\bX)] \big |  \notag   \\
&\geq  1 + \max _{j \in [d] } \bigl \{  \nu (1- \nu) (\Delta \mu_j)^2/ \sigma_{j} \bigr \} -\tau_{\overbar q_{j^*}}  - 1/n \notag  \\
& \geq  1 + \max _{j \in [d] } \bigl \{  \nu (1- \nu) (\Delta \mu_j)^2/ \sigma_{j} \bigr \}  - 5 R^2 \cdot \log n \cdot  \sqrt{ 2 [\log (2d) +\log(1/\xi) r]/ n }.
\#
Thus, when \eqref{eq:wnvsd} holds, by \eqref{eq:test_zq_upper_final} we have 
\$
\sup_{j\in[d]} Z_{\overbar q_{j}}  \geq Z_{\overbar q_{j^*}} \geq   1 + 16  R^2 \cdot \log n \cdot  \sqrt{ 2 [\log (2d) +\log(1/\xi) r]/ n } .
\$
Therefore, the test function in \eqref{eq::test_fun21} takes value one on $\cE$, which implies that the type-II error is upper bounded by $\xi$.   This concludes the proof of Theorem \ref{thm::test_full2}.
\end{proof}


\section{Proofs of Auxiliary Results}\label{sec::append}
In this section, we first show how to select the truncation levels for the query functions in  \eqref{eq::query_reg1} and \eqref{eq::query_reg2}, and then present the proofs of the auxiliary results in \S \ref{sec::proof}.

\subsection{Truncation Levels for Query Functions in \S \ref{sec::upper_bound_reg} } \label{sec::truncation_level}
Remind that the query functions in \eqref{eq::query_reg1} and \eqref{eq::query_reg2}  involve truncation on the response variable $Y$. In the following, we prove that the truncation levels are absolute constants by explicitly characterizing the effect of truncation. 

For the query function $q_{\vb}$ defined in \eqref{eq::query_reg1},  to show \eqref{eq:goal_truncation_reg}, it suffices to   find an absolute constant $R$ such that 
\# \label{eq:truncation_goal00}
   \EE_{\PP_{\vb}} \bigl [q_{\vb } (Y, \bX)\bigr]  -  \EE_{\PP_0} \bigl[q_{\vb }( Y , \bX) \bigr ]   \geq s\beta ^2.
\#
Note that  by \eqref{eq::regression_model} we have 
\#\label{eq:expectations_no_truncation}
\EE_{\PP_{0}}   \{  Y^2 [ s^{-1}  ( \bX^\top \vb  )^2- 1] \}  = 0  \quad \text{and}\quad  \EE_{\PP_{\vb}} \{  Y^2 [ s^{-1}  ( \bX^\top \vb)^2- 1] \} = 2s  \beta^2.
 \# 
  For ease of presentation, we  define $ \tilde q_{\vb}(y, \xb) = y^2 \cdot \ind \{ | y | \leq \sigma \cdot R \} \cdot [ s^{-1} ( \xb ^\top \vb )^2 - 1 ]$. Since $Y$ and $\bX$ are independent under $\PP_0$, we have 
$\EE_{\PP_0} [ \tilde q_{\vb} ( Y, \bX)] = 0$.

 Our derivation of \eqref{eq:truncation_goal00} consists of two steps.
 We first show that 
\# \label{eq:truncation_reg_step1}
\max \Bigl \{ \bigl | \EE_{\PP_{\vb} } [ q_{\vb} ( Y, \bX) - \tilde q_{\vb} (Y, \bX)] \bigr | , ~ \bigl | \EE_{\PP_{0} } [ q_{\vb} ( Y, \bX) - \tilde q_{\vb} (Y, \bX)] \bigr |  \Bigr \}  \leq s\beta^2 / 4. 
\#
 Then we further show that 
 \#\label{eq:truncation_reg_step2}
  \EE_{\PP_{\vb}} \bigl [\tilde q_{\vb } (Y, \bX)\bigr]  -  \EE_{\PP_0} \bigl[\tilde q_{\vb }( Y , \bX) \bigr ]  \geq 3 s^2 \beta^2 / 2.
 \#
 Combining \eqref{eq:truncation_reg_step1} and \eqref{eq:truncation_reg_step2}, we obtain  \eqref{eq:truncation_goal00}. 
 
 In the following, we establish \eqref{eq:truncation_reg_step1}. By definition, we have 
 \$
 \tilde q_{\vb} (Y, \bX)] - q_{\vb} ( Y, \bX) = Y ^2 \cdot \ind \bigl  \{ | Y  | \leq \sigma \cdot R \bigr \}   \cdot [ s^{-1} ( \xb ^\top \vb )^2 - 1 ] \cdot \ind \bigl  \{ | \bX^\top \vb|  >  R \sqrt{ s\log n} \bigr \}.
 \$
 By Cauchy-Schwarz inequality, we have 
 \#\label{eq:truncation_step1_1}
 & \bigl | \EE_{\PP_0} \bigl [  \tilde q_{\vb} (Y, \bX)] - q_{\vb} ( Y, \bX)  \bigr ] \bigr | ^2 \notag \\
 &\quad   \leq  \EE_{\PP_0}  \Bigl \{   Y ^4     \cdot [ s^{-1} ( \xb ^\top \vb )^2 - 1 ]^2  \Bigr \}   \cdot \PP_0\bigl ( | \bX^\top \vb| /\sqrt{s}  >  R \sqrt{  \log n} \bigr   ) \notag \\
 &\quad  =     \EE_{ \PP_0} (Y^4 ) \cdot \EE_{ \PP_0} \bigl \{   [ s^{-1} ( \xb ^\top \vb )^2 - 1 ]^2 \bigr \} \cdot \PP_0\bigl ( | \bX^\top \vb| /\sqrt{s}  >  R \sqrt{  \log n} \bigr   ),
 \#
where the last equality follows from the fact that $Y$ and $\bX$ are independent under $\PP_0$. Since $Y \sim N(0, \sigma^2)$ and $\bX^\top \vb / \sqrt{s} \sim N(0,1)$, we have 
$  \EE_{ \PP_0} (Y^4 )  = 3\sigma^4$ and $ \EE_{ \PP_0}  \{   [ s^{-1} ( \xb ^\top \vb )^2 - 1 ]^2 \} = 2.$ Moreover, by the tail inequality of Gaussian random variables, we have 
 \$
  \PP_0\bigl ( | \bX^\top \vb| /\sqrt{s}  >  R \sqrt{  \log n} \bigr   ) \leq 2 \exp( -R^2 / 2 \cdot  \log n ).
 \$
 Thus, by \eqref{eq:truncation_step1_1}, we have 
 \#\label{eq:truncation_step1_2}
 \bigl | \EE_{\PP_0} \bigl [  \tilde q_{\vb} (Y, \bX)] - q_{\vb} ( Y, \bX)  \bigr ] \bigr | ^2 \leq 12 \sigma^2 \cdot  \exp( -R^2 / 2 \cdot  \log n ).
 \# 
 
 Similarly, under $\PP_{\vb}$, Cauchy-Schwarz inequality implies that 
 \#\label{eq:truncation_step1_3}
 & \bigl | \EE_{\PP_\vb } \bigl [  \tilde q_{\vb} (Y, \bX)] - q_{\vb} ( Y, \bX)  \bigr ] \bigr | ^2 \notag \\
 &\quad   \leq  \EE_{\PP_\vb }  \Bigl \{   Y ^4     \cdot [ s^{-1} ( \bX ^\top \vb )^2 - 1 ]^2  \Bigr \}   \cdot \PP_{\vb}\bigl ( | \bX^\top \vb| /\sqrt{s}  >  R \sqrt{  \log n} \bigr   ) \notag \\
 & \quad \leq \sqrt{  \EE_{\PP_{\vb} } (Y^8 ) \cdot \EE_{\PP_{\vb} } \{  [ s^{-1} ( \bX ^\top \vb )^2 - 1 ]^4  \} } \cdot \PP_{\vb}\bigl ( | \bX^\top \vb| /\sqrt{s}  >  R \sqrt{  \log n} \bigr   ).
 \#
 Note that $Y \sim N(0, \sigma^2 + s\beta ^2 )$ and $\bX ^\top \vb / \sqrt{s} \sim N(0,1)$ under $\PP_{\vb}$. By \eqref{eq:truncation_step1_3} we obtain that there exists an absolute constant $C$ such that 
 \#\label{eq:truncation_step1_4}
   \bigl | \EE_{\PP_\vb } \bigl [  \tilde q_{\vb} (Y, \bX)] - q_{\vb} ( Y, \bX)  \bigr ] \bigr | ^2 \leq C \cdot \exp(- R^2 / 2\cdot \log n).
 \# 
 Hence, combining \eqref{eq:truncation_step1_2} and  \eqref{eq:truncation_step1_4}, when $R$ is sufficiently large, we have 
 \#\label{eq:set_R_expoential}
 \max \Bigl \{ \bigl | \EE_{\PP_{\vb} } [ q_{\vb} ( Y, \bX) - \tilde q_{\vb} (Y, \bX)] \bigr | , ~ \bigl | \EE_{\PP_{0} } [ q_{\vb} ( Y, \bX) - \tilde q_{\vb} (Y, \bX)] \bigr |  \Bigr \}  \leq 1/ n,  
 \#
 which implies \eqref{eq:truncation_reg_step1} when \eqref{eq:reg_signal1} holds. 
 
It remains to establish \eqref{eq:truncation_reg_step2}. By \eqref{eq:expectations_no_truncation} and the fact that $ \EE_{\PP_0} [ \tilde q_{\vb} ( Y, \bX)] = 0$, it  suffices to find an $R$ such that 
\#\label{eq::truncation_creteria1}
 & \EE_{\PP_{\vb}} \Bigl \{  Y^2 [ s^{-1}  ( \bX^\top \vb)^2- 1] - \tilde q_{\vb} (Y, \bX ) \Bigr  \} \notag \\
 &\quad = \EE_{\PP_{\vb}} \Bigl\{ Y ^2\cdot  \ind(|Y| > \sigma R) \cdot   \bigl  [s^{-1} (   \bX^\top \vb )^2 - 1 \bigr] \Bigr \} \leq s\beta^2 / 2.
\#

Notice that, under $\PP_{\vb}$,  we have   $|Y| \stackrel{D}{=}  |\beta \bX^\top \vb + \epsilon|$. We denote $W = (\beta \bX^\top \vb + \epsilon)/\sqrt{ \varsigma + \sigma^2}$ and $Z = \bX^\top \vb /\sqrt{s}$, where $\varsigma= \sqrt{s \beta^2}$. Then  $W$ and $Z$ are both standard Gaussian random variables marginally and~their  correlation  is $\varsigma  / \sqrt{ \varsigma ^2 + \sigma^2}$.
Then we can write the left-hand side of \eqref{eq::truncation_creteria1} as 
\#\label{eq::reduce_to_normal}
&\EE_{\PP_{\vb}} \Bigl\{ Y ^2\cdot  \ind(|Y| > \sigma R) \cdot   \bigl [s^{-1} (   \bX^\top \vb )^2 - 1 \bigr ]  \Bigr \} \notag \\
 &\quad = \EE \bigl [ ( \varsigma ^2 + \sigma^2)\cdot | W  |^2 \cdot \ind\bigl (|W | >  \sigma R/ \sqrt{ \varsigma ^2 + \sigma^2}  \bigr) \cdot      ( Z^2 -1)\bigr].
\#
Note that the right-hand side of \eqref{eq::reduce_to_normal} is of the  form $\EE [f (W)\cdot g(Z)]$ for some functions $f$ and $g$. By expanding the joint density of $W$ and $Z$ using Hermite polynomials, the following lemma enables us to calculate $\EE [f (W)\cdot g(Z)]$ in general settings.

\begin{lemma}\label{lemma::hermite_poly}
Let $\{ H_k \}_{k \geq 0}$ be the   normalized Hermite polynomials such that 
\$
\int_{-\infty} ^{+\infty} H _k (x) H_{\ell} (x) \phi(x) \ud x = \delta_{k \ell},
\$
where $\phi(x)$ is the density of $N(0,1)$ and $\delta _{k \ell}$ is the Kronecker delta function. Let centered random variables $W$ and $Z$ follow  bivariate  Gaussian distribution with variance one and correlation $\zeta$.  For any functions $f   = \sum_{k=0}^\infty a_k H_k   $ and $g = \sum_{k=0}^\infty b_k  H_k   $ such that $ \sum_{k=0}^\infty a_k^2 < \infty $ and $\sum_{k=0}^\infty b_k^2 < \infty$,  we have
\$
\EE \bigl[ f(W) \cdot g(Z) \bigr] = \sum_{k = 0}^\infty a_k b_k  \zeta ^k.
\$
\end{lemma}
\begin{proof}
We denote the joint density of $(W,Z)$ by $\psi( w, z; \zeta)$. It is known that $\psi( w, z; \zeta)$ can be written as a power series of  the correlation $\zeta$ by
\#\label{eq::expansion_density}
\psi( w, z; \zeta)  =\phi (w) \cdot \phi(z) \sum_{k=0}^ \infty  \zeta^k \cdot  H_k(w) \cdot H_k(z).
\#
See Chapter 11 of \cite{balakrishnan2009continuous} for more details.
Hence, for any  integers $\ell, m \geq 0$,  we have 
\$
\EE \bigl [ H_{\ell}(W)\cdot H_m (Z)\bigr] = \sum_{k=0}^ \infty    \zeta^k \cdot \EE \bigl [ H_k (W)\cdot  H_{\ell} (W) \bigr  ] \cdot  \EE\bigl  [ H_{k} (Z) \cdot H_m (Z)\bigr  ]  =   \sum_{k=0}^ \infty   \zeta^k \delta_{k\ell} \delta _{k m}.
\$
Therefore, for $\EE [ f(W) \cdot g(Z) ] $,   we have
\$
\EE [ f(W) \cdot g(Z) ] =  \sum_{\ell, m =0}^{\infty}  a_{\ell} b_{m} \cdot  \EE \bigl  [ H_{\ell}(W)\cdot H_m (Z)\bigr ]  =   \sum_{k = 0}^\infty  a_k b_k \zeta^k,
\$
which  concludes the proof of this Lemma.
\end{proof}

 For notational simplicity, let $f_{t} (w) =   w^2\cdot  \ind ( |w| > t) $ and $g(z) = z^2 - 1$. Note that $g(z) = \sqrt{2} \cdot H_2(z)$. By Lemma  \ref{lemma::hermite_poly}, we have
 \$
 \EE \bigl [ ( \varsigma ^2 + \sigma^2)  \cdot f_t(W) \cdot g(Z)\bigr  ] =  ( \varsigma ^2 + \sigma^2) \cdot \bigl(   \varsigma /   \sqrt {\varsigma ^2 + \sigma^2}  \bigr )^2 \cdot a_2(t) = \varsigma ^2  a_2(t).
 \$
 Here $a_2(t) $ is  the  coefficient of $H_2$ in  expansion $\sqrt{2} \cdot f_t = \sum_{k=0} ^\infty a_k (t) \cdot H_k,$ which is given by
 \#\label{eq::final_threshold}
 a_2 (t) = \int_{-\infty} ^{\infty}  w^2 (w^2 - 1)\cdot  \ind ( |w| > t)\cdot \phi(w) \ud w, ~~t\geq 0.
 \# 
 Note that $a_2 \colon \RR\rightarrow \RR$ in \eqref{eq::final_threshold} is   monotonically nonincreasing. Moreover, 
 $a_2(0) = 2$ and $a_2(t) $ tends to zero as $t$ goes to infinity. Thus,  we can set  the truncation level $R$ in \eqref{eq::truncation_creteria1} to be a sufficiently large constant such that     
 \#\label{eq::set_trunc_level_final}
 R \geq  2 \cdot \inf \bigl \{ t\colon  a_2(t) \leq 1/2  \bigr\},
 \#
 and that  \eqref{eq:set_R_expoential} also holds.
which is an absolute constant. Since $\varsigma^2 = s\beta^2 $ is negligible compared with  $\sigma^2 $, by \eqref{eq::reduce_to_normal}  we have
\#\label{eq:use_a_fancylemma}
&\EE \bigl [ ( \varsigma ^2 + \sigma^2)\cdot | W  |^2 \cdot \ind\bigl(|W | >  \sigma R /   \sqrt{ \varsigma ^2 + \sigma^2}  \bigr) \cdot      ( Z^2 -1)\bigr] \notag \\
&\quad \leq \EE \bigl [ ( \varsigma ^2 + \sigma^2)\cdot | W  |^2 \cdot \ind(|W | > R /2 ) \cdot      ( Z^2 -1)\bigr ]  \leq \varsigma^2 \cdot a_2 (R/2) \leq  s\beta ^2 / 2,
\#
which implies \eqref{eq::truncation_creteria1} and \eqref{eq:truncation_reg_step2}. 
Therefore, our choice of $R$ in \eqref{eq::set_trunc_level_final} satisfies the desired condition in \eqref{eq:goal_truncation_reg}. 

Similarly, for query function $q_j$ in \eqref{eq::query_reg2},  we define $\tilde q_j(y, \xb) = y^2 \cdot \ind \{ | y | \leq \sigma R \} \cdot (X_j^2 - 1) $.
By Cauchy-Schwarz inequality, we have 
\#\label{eq:truncation_q_j_cauchy}
  \bigl | \EE  \bigl [  \tilde q_{j} (Y, \bX)] - q_{j} ( Y, \bX)  \bigr ] \bigr | ^2  &  \leq  \EE\bigl [  Y ^4    (    X_j^2 - 1) ^2    \bigr ]   \cdot \PP \bigl ( |X_j|   >  R \sqrt{  \log n} \bigr   ) \notag \\
&  \leq \sqrt{  \EE  (Y^8 ) \cdot \EE  \bigl  [   ( X_j ^2 - 1 )^4  \bigr ] }  \cdot \PP ( |  X_j|   >  R \sqrt{  \log n} \bigr   ), 
\#
where the expectation is taken under either $\PP_0$ or $\PP_{\vb}$. By direct computation, it can be shown that there exists an absolute constant $\tilde C$ such that 
\$
\max \Bigl \{ \EE_{\PP_0} (Y^8), ~\EE_{\PP_{\vb}} (Y^8), ~ \EE_{\PP_{0}} [ ( X_j^2 -1 )^4 ], ~\EE_{\PP_{\vb}}   [ ( X_j^2 -1 )^4 ] \Bigr \} \leq \tilde C,
\$
where $\tilde C$ depends on $\sigma$. Since $X_j \sim N(0,1)$, by \eqref{eq:truncation_q_j_cauchy} we obtain that 
\#\label{eq:truncation_q_j_final}
 \bigl | \EE  \bigl [  \tilde q_{j} (Y, \bX)  - q_{j} ( Y, \bX)  \bigr ] \bigr | ^2 \leq 2 \tilde C \cdot \exp( - R^2 / 2\cdot \log n).
\#
Thus, by setting $R$ in \eqref{eq:truncation_q_j_final} to be a  sufficiently large constant, we have 
\# \label{eq:truncation_q_j_final2}
 \max \Bigl  \{  \bigl | \EE_{\PP_0}   \bigl [  \tilde q_{j} (Y, \bX)  - q_{j} ( Y, \bX)  \bigr ] \bigr |,~ \bigl | \EE_{\PP_\vb }   \bigl [  \tilde q_{j} (Y, \bX)  - q_{j} ( Y, \bX)  \bigr ] \bigr | \Bigr \} \leq 1/ n \leq \beta^2 /4 
\#
for any $j \in [d]$, 
where the last inequality follows from \eqref{eq:reg_signal2}.

In the sequel, we show that, for any $j \in \supp(\vb)$, it holds that 
\#\label{eq:truncation_step22}
  \EE_{\PP_{\vb} } [ \tilde q_j(Y, \bX)] - \EE_{\PP_0} [ \tilde q_j (Y, \bX) ]   \geq 3 \beta^2/2.
\#
Combining \eqref{eq:truncation_q_j_final2} and  \eqref{eq:truncation_step22} yields the desired inequality in  \eqref{eq:goal_truncation_reg2}.
 
Notice that  $\EE_{\PP_0} [ \tilde q_j (Y, \bX) ] =0 $  since $Y$ and $\bX$ are independent under $\PP_0$. By the definition of $\tilde q_j$, to show \eqref{eq:truncation_step22}, it suffices to 
 find a truncation level  $R $ such that 
 \#\label{eq::second_truncation_goal} 
\EE_{\PP_{\vb}}  \bigl [  Y ^2\cdot  \ind( | Y| > \sigma R) \cdot   (X_j^2- 1 ) \bigr ] \leq \beta^2 /2.
\#
In this case, we also set   $R$ such that   \eqref{eq::set_trunc_level_final} and \eqref{eq:truncation_q_j_final2} hold simultaneously. In the following, we show that the condition  in \eqref{eq::second_truncation_goal}  is satisfied.

Let $ W = (\beta \bX^\top \vb + \epsilon)/\sqrt{ \varsigma^2  + \sigma^2}$ and $Z_1= X_j$, where $j \in \supp(\vb)$. Under $\PP_{\vb}$, we have 
\$
|Y| \stackrel{D}{=}  \sqrt{ \varsigma^2 + \sigma^2}  \cdot |W|.
\$  Besides,   $W$ and $Z_1$ are centered bivariate Gaussian random variables with variance one and correlation $\beta / \sqrt{ \varsigma ^2  + \sigma^2}$. Since  $\varsigma ^2 = s\beta^2 $ is negligible compared with $\sigma^2$, similar to \eqref{eq:use_a_fancylemma}, we use Lemma \ref{lemma::hermite_poly}  to obtain that 
\$
&  \EE_{\PP_{\vb}} \bigl  [  Y ^2\cdot  \ind( | Y| > \sigma R) \cdot   (X_j^2- 1 ) \bigr ]  =  \EE \bigl[ ( \varsigma ^2 + \sigma^2)\cdot | W  |^2 \cdot \ind\bigl (|W | >   \sigma R/ \sqrt{ \varsigma ^2 + \sigma^2}    \bigr  ) \cdot      ( Z_1^2 -1)\bigr ]\\
& \quad  \leq \EE \bigl [ ( \varsigma ^2 + \sigma^2)\cdot | W  |^2 \cdot \ind(|W | >   R/2  ) \cdot      ( Z_1^2 -1) \bigr] \leq \beta^2\cdot a_2 (R/2) \leq \beta^2 /2 ,
\$
where $a_2 $ is defined in \eqref{eq::final_threshold}. 
Therefore, we   conclude the derivation of  \eqref{eq:goal_truncation_reg2}. 

\subsection{Proofs of Auxiliary Results in \S \ref{sec::proof}}
In the following, we prove the supporting lemmas used in the proofs of the main results in \S \ref{sec::proof}.
\subsubsection{Proof of Lemma \ref{lemma::distinguish}} \label{proof::lemma::distinguish}
\begin{proof}
Given an algorithm $\mA \in \cA(T)$, suppose that $\mA$   makes  queries~$\{q_t\}_{t=1}^T \subseteq \cQ_{\mA}$. By \eqref{eq:capapcity_bound}, it follows that $\cG(s) \setminus \bigcup_{t\in [T]} \cC(q_t)$ is not empty. Thus, there exists $   \vb_0  \in \cG(s)$ such that 
	\#\label{eq:oracle_confuse}
\bigl 	|    \EE_{\PP_{0}} \bigl[q(\bX)\bigr] -  \EE_{\PP_{\vb_0}} \bigl[q(\bX)\bigr] \bigr| \leq      \tau_{q, \vb_0}  
	\#
	for any $q\in \cQ_{\mA}$, where $\tau_{q, \vb_0}$ is the tolerance parameter in \eqref{eq::query_2} under $\PP_{\vb_0}$.  To prove this lemma, it suffices to show that there exists an oracle $r$ such that 
	\#\label{eq::reduce_risk}
\inf_{\phi\in \cH(\mA,r)} \Bigl[ \overline{\PP}_{0}(\phi = 1) +  \overline{\PP}_{\vb_0}(\phi = 0) \Bigr]  = 1.
	\#
	
To this end, in the sequel, we construct an  the oracle $r_0$  as follows.   When $\bX$ follows $\PP_0$, for any query function $q \in \cQ_{\mA}$, $r_0$ returns $\EE_{\PP_0} [ q(\bX) ]$.  Moreover,  when $\bX $ follows $ \PP_{\vb}$ for any $\vb \in \cG(s)$, for any query function $q$, if $ \EE_{\PP_0} [ q(\bX) ] $ is a valid response in the sense that $ Z_q = \EE_{\PP_0} [ q(\bX) ] $ satisfies \eqref{eq::query_00}, $r_0$ returns $\EE_{\PP_0} [ q(\bX) ]$ as the response. Otherwise it returns $ \EE_{\PP_\vb} [ q(\bX) ] $. It  is not hard to see that, $r_0$ satisfies Definition \ref{def::oracle} and  is thus a valid oracle.

By the construction of $r_0$ and \eqref{eq:oracle_confuse},   when $\bX \sim \PP_{\vb_0}$,  for any $t \in [T]$,  $Z_{q_t} = \EE_{\PP_0} [ q_t(\bX)]$ satisfies \eqref{eq::query_00} for query function $q_t$. Thus, when we query  oracle $r_0$ by $q_t$ in the $t$-th round, $r_0$ returns $Z_{q_t} = \EE_{\PP_0} [ q_t(\bX)]$. As a result,  $r_0$ returns the same responses for query functions $ \{q_t \}_{t =1}^T $ when $\bX$ follows either   $\PP_0$ or  $\PP_{\vb_0}$.

Furthermore, since any test function $\phi$ computed by $\mA\in \cA(T)$ is a function of $\{ Z_{q_t} \}_{t=1}^T$, which  have the same distribution under $\overline{\PP}_0 $ and  $\overline{\PP}_{\vb_0}$, we conclude that
\$\overline {\PP}_0 (\phi = 1) + \overline{\PP}_{\vb_0}(\phi = 0) \geq 1.\$
Meanwhile, note that a hypothesis test that randomly rejects the null hypothesis  with probability~$1/2$ incurs risk one. 
Therefore, we establish \eqref{eq::reduce_risk}, which concludes  the proof of Lemma \ref{lemma::distinguish}.
\end{proof}

\subsubsection{Proof of Lemma \ref{lemma::compute_chi_square}}\label{proof::lemma::compute_chi_square}
\begin{proof}
We define a random variable  $\eta\in \{ 1- \nu, \nu\} $ such that
\$
\PP(\eta = 1- \nu) = \nu~~\text{and}~~\PP(\eta = -\nu) = 1- \nu.
\$ 
For any $\vb \in \cG(s)$, by the definition of $\PP_{\vb}$, we have
\$
\frac{\ud \PP_{\vb}}{\ud \PP_0} (\xb) = \EE_{\eta} \bigl [ \exp ( - \eta  \beta \vb^\top \xb) \cdot \exp ( -\eta^2 /2\cdot   s  \beta^2 )\bigr],
\$
where $\EE_{\eta}$ is the expectation with respect to the randomness of $\eta$.
  Then for any $\vb_1$ and $\vb_2 $ in $\cG(s)$,   by Fubini's theorem we obtain
\#\label{eq::use_expectation_like_ratio30}
\EE_{\PP_{0}} \biggl[ \frac{\ud  \PP_{\vb_1}}{\ud \PP_0} \frac{\ud   \PP_{\vb_2}}{\ud \PP_0} (\bX) \biggr ]
 & =  \EE_{\PP_0} \EE_{\eta_1,\eta_2} \Bigl\{\exp\bigl [ - \beta (\eta_1  \vb_{1} + \eta_2  \vb_2)^\top \bX -  s  \beta^2 /2 \cdot(\eta_1^2 + \eta_2^2 ) \bigr  ]  \Bigr \}\\
 & = \EE_{\eta_1,\eta_2} \Bigl(  \EE_{\PP_{0}} \Bigl\{ \exp\bigl[ - \beta (\eta_1   \vb_{1} + \eta_2  \vb_2)^\top \bX\bigr ]\Bigr \} \cdot \exp\bigl[-  s \beta^2 /2 \cdot(\eta_1^2 + \eta_2^2 ) \bigl  ]\Bigr ) .\notag
\#
Since $\EE_{\PP_0} [\exp(\ab^\top \bX)] = \exp( \| \ab\|_2^2 /2)$, we have
\#\label{eq::use_expectation_like_ratio4}
\EE_{\PP_0} \Bigl\{ \exp\bigl[ - \beta (\eta_1   \vb_{1} + \eta_2  \vb_2)^\top \bX\bigr ]\Bigr \}  &= \exp\bigl(  \beta^2 /2 \cdot \| \eta_1  \vb_1  + \eta_2  \vb_2\|_2^2 \bigr )\notag\\
&=  \exp\bigl[  s     \beta^2/2 \cdot (\eta_1 ^2 + \eta_2^2) + \beta^2  \eta_1 \eta_2 \la \vb_1 ,\vb_2\ra \bigr ].
\#
Combining \eqref{eq::use_expectation_like_ratio30} and \eqref{eq::use_expectation_like_ratio4}, we have
\$
\EE_{\PP_{0}} \biggl[
\frac{\ud  \PP_{\vb_1}}{\ud \PP_0}  \frac{\ud   \PP_{\vb_2}}{\ud \PP_0} (\bX) \biggr ]   &= \EE_{\eta_1,\eta_2} \bigl [\exp  ( \beta^2 \eta_1 \eta_2 \la \vb_1 ,\vb_2\ra   ) \bigr]= \EE_{U} \bigl [ \cosh  ( \beta^2 U \la \vb_1 , \vb_2 \ra  )\bigr],
\$
where the random variable $U = \eta_1 \eta_2$  satisfies \$\PP\bigl [U = (1-\nu)^2\bigr ] = \nu^2,~~\PP\bigl [ U =- \nu(1-\nu) \bigr] = 2\nu(1-\nu), ~~\PP(U = \nu^2) = (1- \nu)^2.\$
Thus, we conclude the proof.
\end{proof}

\subsubsection{Proof of Lemma \ref{lemma::growth_Cj}}\label{proof::lemma::growth_Cj}
\begin{proof}
Since $| \cC_{j}(\vb)|$ does not depend on $\vb$, without any loss of generality, we assume that the first~$s$ entries of $\vb$ are equal to one and the rest are all zero.  For two integers $a, b\geq 0$, we define
\#\label{eq::S_ab}
\cS_{a,b} =\biggl   \{ \ub\in \cG(s)\colon \sum_{i=1}^s \ind{( u_j= 1) } = a ~\text{and}~  \sum_{i=1}^s \ind {( u_j = -1)  = b}  \biggr\}.
\#
Then by definition, for any $\vb' \in \cS_{a,b}$, $\la \vb, \vb' \ra = a - b$. In addition, we define
\$M(k) =\bigl | \{ \vb' \in \cG(s)\colon \la \vb, \vb'\ra = k \}\bigr  | ~~\text{and}~~ N_{a,b}  = | \cS_{a,b}| .\$
By the symmetry of $\cG(s)$, it holds that $M(k) = M(-k)$. Also, by the definition of $\cS_{a,b}$ in \eqref{eq::S_ab},~for $k \in\{ 0, \ldots, s\}$, we have
\$
M(k) =     N_{k,0} + N_{k+1, 1} + \cdots + N_{\floor{\frac{s+k}{2}}, \floor{\frac{s-k}{2}}}.
\$
For any $a,b \geq 0$ satisfying $a+b \leq s$, by calculation, we have
\#\label{eq::N_ab}
N_{a,b} = { s\choose a}{s-a \choose b} {d-s \choose s-a-b}\cdot  2^{s-a-b}.
\#
Hence, if $a+b + 1\leq s$, by  \eqref{eq::N_ab} we obtain
\#\label{eq::compare}
\frac{N_{a,b}}{N_{a+1,b}} & = \frac{{ s\choose a}}{{ s\choose a+1}} \cdot \frac{{s-a \choose b}}{ { s- a-1 \choose b } }\cdot \frac{{ d-s \choose s-a-b}}{{d-s \choose s-a-b-1}} \cdot 2  \notag\\
&=\frac{ a+1}{s-a} \cdot \frac{s-a}{ s-a-b} \cdot \frac{ d-s+a+b+1}{s-a-b}\cdot 2 \notag\\
&= \frac{2 (a+1)(d-s+a+b-1)}{(s-a-b)^2} \geq \frac {d} {s^2}.
\#
Here the last inequality holds because $2s \leq d$. Now we consider $M(k-1)/ M(k)$ for $k \in[s]$.  If $s+k$ is odd, we have $\floor {(s+k)/2} = (s+k-1)/2.$ Hence, we have
\#\label{eq::compare2}
M(k-1) = N_{k-1, 0} + \cdots+  N_{\frac{s+k-1}{2}, \frac{s-k+1}{2}}~~\text{and}~~ M(k) = N_{k,0} + \cdots+ N_{\frac{s+k-1}{2} , \frac{s-k-1}{2}}.
\#
Combining \eqref{eq::compare} and \eqref{eq::compare2}, we have $M(k-1) \geq d/s^2 \cdot M(k)$. Moreover, if $s+k$ is even, we have $\floor {(s+k)/2} = (s+k)/2.$ In this case we have
\#\label{eq::compare3}
M(k-1) = N_{k-1, 0} + \cdots+ N_{\frac{s+k-2}{2}, \frac{s-k}{2}}~~\text{and}~~ M(k) = N_{k,0} + \cdots + N_{\frac{s+k}{2} , \frac{s-k}{2}}.
\#
Combining  \eqref{eq::compare3} and  \eqref{eq::compare}, we also have  $M(k-1) \geq d/s^2\cdot  M(k)$. By the definition of $\cC_j(\vb)$,~it holds that~$|\cC_s(\vb)| = M(0)$~and~$|\cC_{j}(\vb) |  = M(s-j) + M(j-s)$~for~$j\in\{0, \ldots, s-1\}$. Thus, we have
$
|\cC_{j+1}(\vb)| / | \cC_j(\vb)| \geq d/(2s^2)$ for all $j$ in $\{0, \ldots, s-1\}.
$
\end{proof}

\subsubsection{Proof of Lemma \ref{lemma::bound_chi_square_div}}\label{proof::bound_chi_square_div}

\begin{proof}[Proof  of Lemma \ref{lemma::bound_chi_square_div}]
We only prove this lemma for $\ell = 1$. The proof is identical for $\ell = 2$.
For any query function $q\in \cQ_{\mA}$, by the definition of $\cC_1(q)$ in \eqref{eq::def_cq}, for any $\vb \in \cC_1(q)$,  we have
\#\label{eq:apply_defin_Cq1}
\EE_{\PP_{\vb}}\bigl[q(\bX)\bigr] - \EE_{\PP_0}\bigl[q(\bX)\bigr] \geq  \tau_{q, \vb}    ,
\#
where the tolerance parameters   $\tau_{q, \vb}$ is defined in \eqref{eq::query_2} under distribution $\PP_{\vb}$. In the following, we adapt   Lemma 3.5 in \cite{feldman2013statistical} to lower bound the left-hand side of \eqref{eq:apply_defin_Cq1} using $\EE_{\PP_0} [q(\bX)]$.
\begin{lemma}\label{lemma:lower_bound_tau}
	Suppose  the  query function $q \in \cQ_{\mA}$ and   $ \vb \in  \cG(s)$ satisfy 
	\#\label{eq::aux_lemma_difference}
	\bigl| \EE_{\PP_{0}} \bigl[q(\bX)\bigr] -  \EE_{\PP_{\vb}} \bigl[q(\bX)\bigr] \bigr| \geq     \tau_{q, \vb}    ,
	\#
	where $\tau_{q, \vb}$ is  the tolerance parameter  defined in \eqref{eq::query_2} under   $\PP_{\vb}$.
	Then 
	we also have 
	\$
	\bigl| \EE_{\PP_{0}} \bigl[q(\bX)\bigr] -  \EE_{\PP_{\vb}} \bigl[q(\bX)\bigr] \bigr| \geq \sqrt{      2   \log (T/\xi) \cdot \Bigl (M^2 - \bigl \{ \EE _{\PP_0}[ q(\bX) ] \bigr\}   ^2  \Bigr ) \big / (3 n)  }  .
	\$
\end{lemma}
\begin{proof}
	See \S\ref{sec:proof:lemma:lower_bound_tau} for a detailed proof. 
\end{proof}
By this lemma, we obtain 
\#\label{eq::compute_chi_eq1}
\sqrt{      2   \log (T/\xi) \cdot \Bigl (M^2 - \bigl \{ \EE _{\PP_0}[ q(\bX) ] \bigr\}   ^2  \Bigr ) \big / (3 n)  }  &\leq  \frac{1}{|\cC_1(q)|} \sum_{\vb\in\cC_1(q)}\Bigl \{\EE_{\PP_{\vb}}\bigl [q(\bX)\bigr] - \EE_{\PP_0}\bigl[q(\bX)\bigr]\Bigr \} \notag \\
& = \frac{1}{|\cC_1(q)|} \sum_{\vb\in\cC_1(q)}  \EE_{\PP_{\vb}}\bigl [\overline q(\bX)\bigr]  ],
  \#
where in the  the last equality we define $\overline q(\xb) = q(\xb) - \EE_{\PP_0} [ q(\bX) ] $. Writing 
$\ud \PP_{\vb} = \ud \PP_{\vb}/ \ud \PP_{0} \cdot \ud \PP_0$    in \eqref{eq::compute_chi_eq1}, we have
\#\label{eq::compute_chi_square_ineq}
 \frac{1}{|\cC_1(q)|} \sum_{\vb\in\cC_1(q)}  \EE_{\PP_{\vb}}\bigl [\overline q(\bX)\bigr]  ] =  \EE_{\PP_0} \biggl(\overline{q} (\bX) \cdot \biggl\{\frac{1}{|\cC_1(q)|} { \sum_{\vb\in\cC_1(q)} }\biggl[\frac{\ud \PP_{\vb}}{\ud \PP_0}(\bX)  - 1\biggr] \biggr\}\biggr).
\#
 Meanwhile, Cauchy-Schwarz inequality implies that
\#\label{eq::compute_chi_square2}
&\EE_{\PP_0} \biggl( \overline q  (\bX) \cdot \biggl\{\frac{1}{|\cC_1(q)|}\sum_{\vb\in\cC_1(q)} \biggl[\frac{\ud \PP_{\vb}}{\ud \PP_0}(\bX)  - 1\biggr] \biggr\}\biggr) \notag\\
&\quad \leq  \Bigl ( \EE_{\PP_0} \Bigl \{ \bigl[\overline{q}(\bX)\bigr]^2\Bigr \}\Bigr)^{1/2}  \cdot  \biggl[\EE_{\PP_0} \biggl(\biggl\{\frac{1}{|\cC_1(q)|}\sum_{\vb\in\cC_1(q)} \biggl[\frac{\ud \PP_{\vb}}{\ud \PP_0}(\bX)  - 1\biggr] \biggr\}^2\biggr) \biggr ]^{1/2}  .
\#
For the first term on the right-hand side of \eqref{eq::compute_chi_square2}, by the definition of $\overline{q}$ and the fact that $q (\bX) \in [-M, M]$,  we have
\#\label{eq::compute_chi_square3}
      \EE_{\PP_0} \Bigl \{  \bigl[\overline{q}(\bX)\bigr]^2  \Bigr \}  =   \EE_{\PP_0} \Bigl \{\bigl [ q(\bX) \bigr ] ^2 \Bigl \} -\Bigl \{ \EE_{\PP_0}\bigl[q(\bX)\bigr ]\Bigr \}^2 
 \leq  M^2 - \Bigl \{ \EE_{\PP_0}\bigl[q(\bX)\bigr ]\Bigr \}^2 .
\#
Thus, combining \eqref{eq::compute_chi_eq1}, \eqref{eq::compute_chi_square_ineq},  \eqref{eq::compute_chi_square2}, and  \eqref{eq::compute_chi_square3},  we have 
\#\label{eq::compute_chi_square_additional}
\sqrt{ \frac{ 2 \log (T/ \xi)}   {3n}  } \leq   \biggl[\EE_{\PP_0} \biggl(\biggl\{\frac{1}{|\cC_1(q)|}\sum_{\vb\in\cC_1(q)} \biggl[\frac{\ud \PP_{\vb}}{\ud \PP_0}(\bX)  - 1\biggr] \biggr\}^2\biggr) \biggr ]^{1/2}.
\# 
It remains to bound the right-hand side of \eqref{eq::compute_chi_square_additional}. By direct computation, we have 
\#\label{eq::compute_chi_square4}
&\biggl[\EE_{\PP_0}\biggl(\biggl\{\frac{1}{|\cC_1(q)|}\sum_{\vb\in\cC_1(q)} \biggl[\frac{\ud \PP_{\vb}}{\ud \PP_0}(\bX)  - 1\biggr] \biggr\}^2\biggr)\biggr]^{1/2} \notag\\
&\quad = \biggl(\frac{1}{|\cC_1(q)|^2}\sum_{\vb_2,\vb_2\in\cC_1(q)}   \EE_{\PP_0} \biggl\{ \biggl[ \frac{\ud \PP_{\vb_1}}{\ud \PP_0}(\bX)  - 1\biggr] \cdot \biggl[\frac{\ud \PP_{\vb_2}}{\ud \PP_0}(\bX)  - 1\biggr] \biggr\}\biggr)^{1/2}   \notag\\
&\quad= \biggl\{ \frac{1}{|\cC_1(q)|^2} \sum_{\vb_2,\vb_2\in\cC_1(q)}   \EE_{\PP_0} \biggl[ \frac{\ud \PP_{\vb_1}}{\ud \PP_0}\frac{\ud \PP_{\vb_2}}{\ud \PP_0}(\bX) - 1\biggr]\bigg\} ^{1/2}=\biggl(\EE_{\PP_0} \biggl\{\biggl[\frac{\ud \PP_{  \cC_1(q)}}{\ud \PP_0}(\bX) -1 \biggr]^2\biggr\}\biggr)^{1/2}\notag\\
&\quad = \bigl[D_{\chi^2} (\PP_{ \cC_1(q)}, \PP_0 )\bigr]^{1/2},
\#
 Therefore, combining \eqref{eq::compute_chi_square_additional} and  \eqref{eq::compute_chi_square4},      we    conclude the proof of Lemma \ref{lemma::bound_chi_square_div}.
\end{proof}

\subsubsection{Proof of Lemma \ref{lemma::compute_h_2}} \label{proof::lemma::compute_h_2}
\begin{proof}
 Let $\eta_1, \eta_2$  be two independent Rademacher random variables. For notational simplicity,
we denote $\kappa = s\beta^2$.  For any $\vb_1, \vb_2\in \cG(s)$, by the definition of $\PP_{\vb}$, we have
\#\label{eq::first_hfun_symmetric}
\EE_{\PP_0} \biggl[ \frac{\ud  \PP_{\vb_1}}{\ud  \PP_{0}} \frac{\ud  \PP_{\vb_2}}{\ud \PP_{0}} (\bX) \biggr]  = \exp\biggl[- \frac{\kappa^2}{ (1-\kappa^2)} \biggr ]\cdot (1-\kappa^2)^{-1} \cdot \EE_{\eta_1,\eta_2}  (
   \Xi ),
\#
were $\Xi$ is defined as
\$
\Xi &=  \EE_{\PP_0} \biggl\{ \exp \biggl  [-\frac{\kappa ^2  (\bX^\top \vb_1)^2}{2(1-\kappa^2)} + \frac{\kappa\eta_1  (\bX^\top \vb_1) }{1-\kappa^2}    -\frac{\kappa ^2  (\bX^\top \vb_2)^2}{2(1-\kappa^2)} + \frac{\kappa \eta_2 (\bX^\top \vb_2)}{1-\kappa^2}   \biggr ] \biggr\}\\
&= \EE_{\PP_0} \biggl \{ \exp \biggl[-\frac{\kappa^2 U_1^2 }{2(1-\kappa^2)}  + \frac{\kappa  U_1}{1-\kappa^2} -\frac{\kappa^2 U_2^2}{2(1-\kappa^2)}    + \frac{\kappa  U_2 }{1-\kappa^2}    \biggr]\biggr\}.
\$
Here we define $U_1 =    s^{-1/2}\eta_1    \bX^\top\vb_1 $ and $U_2 =  s^{-1/2}  \eta_2   \bX^\top \vb_2  $. Let $\alpha = s^{-1}\eta_1 \eta_2  \la \vb_1, \vb_2 \ra$. We define
\$
V = ( 1- \alpha^2 )^{-1/2} ( U_2 - \alpha U_1).
\$ By definition, $U_1$  and $V$ are independent  standard normal random variables.  By definition, we have
\$
& \Xi 
=\EE_{\PP_0} \biggl ( \exp\biggl\{ -\frac{\kappa^2  \bigl [U_1^2 + (\alpha U_1+\sqrt{1- \alpha ^2}V)^2\bigr] }{2(1-\kappa^2)}   + \frac{\kappa (U_1 +  \alpha  U_1+\sqrt{1- \alpha ^2} V)}{1-\kappa^2}  \biggr\} \biggr) \notag\\
& \quad =  \EE_{\PP_0}\biggl\{ \exp\biggl[ -\frac{\kappa^2(1- \alpha ^2) V^2}{2(1-\kappa^2)} + \frac{\kappa \sqrt{1- \alpha^2} \cdot(1-\kappa  \alpha  U_1) V}{1-\kappa^2} -\frac{\kappa^2(1+ \alpha  ^2)U_1^2}{2(1-\kappa^2)} + \frac{\kappa(1+ \alpha) U_1}{1-\kappa^2} \biggr]\biggr\}.
\$
Note that for  any $a < 1/2$ and $b \in \RR$, we have
\$
\EE\bigl[\exp(a Z^2 + b Z)\bigr] = (1-2a)^{-1/2} \exp\bigl[b^2/(2-4a)\bigr],
\$
where $Z \sim \cN(0,1)$.
By first taking expectation with respect to $V$, we obtain that
\$
\Xi &=  \sqrt{\frac{1-\kappa^2}{1-\kappa^2 \alpha ^2}} \cdot  \EE_0\biggl\{  \exp\biggl[\frac{\kappa^2(1- \alpha  ^2) (1-\kappa \alpha  U_1  )^2}{2(1-\kappa^2)(1-\kappa^2 \alpha^2)} -\frac{\kappa^2(1+ \alpha  ^2) U_1^2 }{2(1-\kappa^2)}+ \frac{\kappa(1+ \alpha )U_1}{1-\kappa^2} \biggr]\biggr\} \notag\\
&= \sqrt{\frac{1-\kappa^2}{1-\kappa^2 \alpha  ^2}}\cdot
\exp \biggl [\frac{\kappa^2(1- \alpha  ^2)}{2(1-\kappa^2)(1-\kappa^2 \alpha  ^2)} \biggr] \cdot \\
& \quad \quad\quad
\EE_0 \biggl\{ \exp\biggl [-\frac{\kappa^2 (1 +  \alpha  ^2 -2 \kappa^2 \alpha  ^2)U_1^2}{2(1-\kappa^2)(1-\kappa^2 \alpha  ^2)}  + \frac{\kappa (1 +  \alpha   ) (1 -\kappa^2  \alpha  )U_1}{(1-\kappa^2)(1-\kappa^2 \alpha  ^2)} \biggr]\biggr\}.
\$
By further taking expectation with respect  to $U_1$, we have
\#\label{eq::last_hfun_symmetric}
\Xi &= \sqrt{\frac{1-\kappa^2}{1-\kappa^2 \alpha  ^2}}  \cdot
\exp \biggl[{\frac{\kappa^2(1- \alpha  ^2)}{2(1-\kappa^2)(1-\kappa^2 \alpha  ^2)}} \biggr]\cdot \notag \\
&\quad \quad \quad
\sqrt{\frac{(1-\kappa^2)(1-\kappa^2 \alpha  ^2)}{1-\kappa^4 \alpha  ^2}}\cdot
\exp \biggr[\frac{\kappa^2 (1+ \alpha  )^2(1-\kappa^2 \alpha  )^2}{2(1-\kappa^2)(1-\kappa^2 \alpha  ^2)(1-\kappa^4 \alpha  ^2)}\biggr] \notag\\
&= \frac{1-\kappa^2}{\sqrt{1-\kappa^4 \alpha  ^2}} \cdot\exp\biggl[\frac{\kappa^2(1-\kappa^2 \alpha  ^2)}{(1-\kappa^2)(1-\kappa^4 \alpha  ^2)}+\frac{\kappa^2 \alpha  }{1-\kappa^4 \alpha  ^2}\biggr].
\#
Combining \eqref{eq::first_hfun_symmetric} and  \eqref{eq::last_hfun_symmetric}, we finally obtain that
\#\label{eq::last_two_step}
\EE_{\PP_0} \biggl[ \frac{\ud  \PP_{\vb_1}}{\ud \PP_{0}} \frac{\ud  \PP_{\vb_2}}{\ud \PP_{0}} (\bX) \biggr]  = \EE_{\eta_1, \eta_2}
  \biggl[ (1-\kappa^4 \alpha  ^2)^{-1/2}  \exp\biggl( - \frac{\kappa^4  \alpha  ^2}{1-\kappa^4 \alpha  ^2}+\frac{\kappa^2 \alpha  }{1-\kappa^4 \alpha  ^2}\biggr)\biggr] .
\#
Recall that $s \alpha =   \eta_1 \eta_2  \la \vb_1, \vb_2 \ra   $ is a sum of $|\la \vb_1, \vb_2 \ra| $ independent Rademacher random variables.~Let $W = s \alpha $. Thus by replacing $\kappa = s\beta^2$ in \eqref{eq::last_two_step}, we obtain
\$
& \EE_{\PP_0} \biggl[ \frac{\ud  \PP_{\vb_1}}{\ud  \PP_{0}} \frac{\ud  \PP_{\vb_2}}{\ud \PP_{0}} (\bX) \biggr]  =
\EE_{W} \biggl[(1-\beta^4 W^2)^{-1/2} \exp\biggl (\frac{-\beta ^4 W^2}{1-\beta ^4 W^2}\biggr) \cdot\cosh\biggl(\frac{\beta^2 W}{1-\beta^4 W^2}\biggr)\biggr].
\$
Therefore, we conclude the proof of Lemma \ref{lemma::compute_h_2}.
\end{proof}

\subsubsection{Proof of Lemma \ref{lemma::compute_chi_square2}}\label{proof::lemma::compute_chi_square2}
\begin{proof}

Under $\PP_0$, we have $\bZ \sim N({\bf 0}, \Ab_0)$, where $\Ab_0= \text{diag}( \sigma^2 + s\beta^2 , \Ib_{d})$. In addition, under $\PP_{\vb}$,  we have
  \$
  \bZ \sim 1/2 \cdot N \bigl [{ \bf 0}, \Ab(\vb)\bigr  ]+ 1/2 \cdot N \bigl [{ \bf 0}, \Ab(-\vb) \bigr],
  \$
  where we define \$\Ab(\vb) = \begin{bmatrix} \sigma^2 + \beta^2 \| \vb\|^2_2 & \beta \vb^\top \\
  \beta \vb & \Ib_{d}
  \end{bmatrix}
 \$ for any $\vb\in \cG(s)$.  By definition, for any $\vb \in \cG(s)$  we have
  \$
  \frac{\ud  \PP_{\vb}}{\ud \PP_0}(\zb) =  \EE_{\eta}  \biggl (  \text{det}^{1/2} (\Ab_0)\cdot \text{det}^{-1/2} \bigl[ \Ab(\eta \vb )\bigr]\cdot \exp \Bigl\{ - 1/2 \cdot \zb^{\top}  \bigl[ \Ab^{-1}(\eta \vb) - \Ab_0^{-1} \bigr ] \zb \Bigr \} \biggr    ),
  \$
  where we denote $\zb =(y, \xb^\top)^\top$.
  Thus for any $\vb_1$ and $\vb_2$ in $\cG(s)$, we further have
  \#\label{eq::use_expectation_like_ratio3}
  \EE_{\PP_{0}} \biggl[
  \frac{\ud  \PP_{\vb_1}}{\ud \PP_0}  \frac{\ud   \PP_{\vb_2}}{\ud \PP_0} (\bZ) \biggr ]  & = \EE_{\PP_0} \EE_{\eta_1,\eta_2} \biggl (\text{det} (\Ab_0)\cdot \text{det}^{-1/2} \bigl[ \Ab(\eta _1 \vb_1)\bigr]\cdot \text{det}^{-1/2}\bigl [ \Ab(\eta _2 \vb_2 )\bigr]  \cdot \notag\\
  &\qquad \exp \Bigl\{ - 1/2 \cdot \bZ^{\top} \bigl  [ \Ab^{-1}(\eta_1 \vb_1) + \Ab^{-1}(\eta_2 \vb_2)  - 2 \Ab_0^{-1}  \bigr] \bZ \Bigr \} \biggr  ).
  \#
  Here we denote $\bZ = ( Y, \bX^\top )^\top$.
  By calculation, we have
  \$ 
  &\EE_{\PP_0}\biggl  (  \exp\Bigl  \{ - 1/2 \cdot \bZ^{\top}  \bigl[ \Ab^{-1}(\eta_1 \vb_1) + \Ab^{-1} (\eta_2 \vb_2)  - 2 \Ab_0^{-1} \bigr] \bZ  \Bigr\}\biggr )\notag\\
  &\quad = (2\pi)^{-(d+1)/2} \text{det}^{-1/2} (\Ab_0)\cdot  \int_{\zb\in \RR^{d+1}} \exp  \Bigl \{  -1/2\cdot  \zb^{\top } \bigl[ \Ab^{-1}(\eta_1 \vb_1) + \Ab^{-1}(\eta_2 \vb_2)  -   \Ab_0^{-1}\bigr  ] \zb \Bigr  \} \ud \zb \notag \\
  & \quad =  \text{det}^{- 1/2} \bigl [\Ab^{-1}(\eta_1 \vb_1) + \Ab^{-1}(\eta_2 \vb_2)  -   \Ab_0^{-1}  \bigr]\cdot \text{det}^{-1/2} (\Ab_0).
  \$
  Thus, by Fubini's  theorem, the right-hand side of \eqref{eq::use_expectation_like_ratio3} is reduced to
  \#\label{eq::fubini}
  & \EE_{\PP_{0}} \biggl[
  \frac{\ud  \PP_{\vb_1}}{\ud \PP_0} \frac{\ud   \PP_{\vb_2}}{\ud \PP_0} (\bZ) \biggr ]  \notag \\
  &\quad  =\EE_{\eta_1, \eta_2}  \Bigl\{   \text{det}  \bigl[\Ab^{-1}(\eta_1 \vb_1) + \Ab^{-1}(\eta_2 \vb_2)  -   \Ab_0^{-1}  \bigr] \cdot\notag\\
&\qquad \qquad \qquad  \text{det} \bigl[\Ab(\eta _1 \vb_1)\bigr] \cdot \text{det}€€\bigl[\Ab(\eta _2 \vb_2)\bigr]\cdot  \text{det}^{-1}(\Ab_0) \Bigr\} ^{-1/2}\notag\\
  & \quad=\text{det}^{1/2} (\Ab_0) \cdot   \EE_{\eta_1, \eta_2} \Bigl\{ \text{det}\bigl [ \Ab(\eta_1 \vb_1) + \Ab(\eta_2 \vb_2) -    \Ab(\eta_1 \vb_1)\cdot \Ab_0^{-1}\cdot \Ab(\eta_2 \vb_2)
 \bigr ] ^{-1/2} \Bigr\}  .
  \#
By calculation  we obtain
  \#\label{eq::compute_1}
  &\Ab(\eta_1\cdot \vb_1)\cdot \Ab_0^{-1}\cdot \Ab(\eta_2\cdot \vb_2) =  \begin{bmatrix} a_1 & \btheta ^\top\\
  	\btheta & \Ab_1
  \end{bmatrix}\notag\\
  &\Ab(\eta_1\cdot \vb_1) + \Ab(\eta_2\cdot \vb_2) =
   \begin{bmatrix}
   	2 \sigma^2  + 2s\beta^2& \btheta^{\top} \\
   	\btheta    & 2\Ib_{d}
   \end{bmatrix},
  \#
  where we define $\btheta  = \beta (\eta_1\vb_1 + \eta_2 \vb_2 )$, $a_1 = \sigma^2 + \beta^2 ( s + \eta_1 \eta_2 \vb_1^\top \vb_2)$, and $ \Ab_1 = \Ib_d + ( \sigma^2 + s\beta ^2)^{-1} \beta^2 \eta_1 \eta_2  \vb_1 \vb_2^\top$. Finally, combining \eqref{eq::fubini} and \eqref{eq::compute_1}, we obtain that
  \#\label{eq::result}
  \EE_{\PP_0} \biggl[\frac{\ud  \PP_{\vb_1}}{\ud \PP_0}  \frac{\ud   \PP_{\vb_2}}{\ud \PP_0} (\bZ) \biggr ]  =  \EE_{\eta_1, \eta_2} \Bigl \{  \bigl [ 1 -  (\sigma^2 + s\beta ^2)  ^{-1}\beta^2  \eta_1 \eta_2 \vb_1^\top \vb_2\bigr]^{-1}\Bigr\}.
  \#
  Since $\eta_1\eta_2$ is a Rademacher random variable, \eqref{eq::result} is reduced to
  \$
  \EE_{\PP_0} \biggl[\frac{\ud  \PP_{\vb_1}}{\ud \PP_0} \frac{\ud   \PP_{\vb_2}}{\ud \PP_0} (\bZ) \biggr ]  = \bigl [ 1 - (\sigma^2 + s\beta^2)^{-2}  \beta^4  | \vb_1^\top \vb_2| ^2\bigr ]^{-1}.
  \$
  To prove the second argument, note that for any $\vb \in \cG(s)$, we have 
  \$
  \frac{\ud \PP_{\vb }^n}{\ud \PP_{0}^n}  (\zb_1, \ldots, \zb_n) &  = \text{det}^{n/2} (\Ab_0)\cdot \\
  & \quad\prod_{i=1}^n \biggl[ \EE_{\eta^i}  \biggl( \text{det}^{-1/2}\bigl [ \Ab(\eta^i \cdot \vb )\bigr] \cdot \exp \Bigl  \{ - 1/2 \cdot \zb_i^{\top}\bigl [ \Ab^{-1}(\eta ^i\cdot \vb) - \Ab_0^{-1} \bigr] \zb _i \Bigr \}\biggr )\biggr ],
  \$
  where $\{\eta^{i}\}_{i=1}^n$ are independent Rademacher random variables. Let $\{ \bZ_i\}_{i=1}^n$ be $n$ independent copies of $\bZ$ and let  $\{\eta_1^{i}, \eta_{2}^i \}_{i=1}^n$ be $2n$ independent Rademacher random variables. Therefore, by Fubini's  theorem  and \eqref{eq::result}, we have 
\$
&\EE_{\PP_0^n }  \biggl[ \frac{\ud \PP_{\vb_1}^n}{\ud \PP_{0}^n}     \frac{\ud \PP_{\vb_2}^n }{\ud \PP_{0}^n} (\bZ_1, \ldots, \bZ_n) \biggr] =  \EE_{\eta_1^1, \ldots, \eta_1^n, \eta_2^1, \ldots, \eta_2^n} \EE_{\PP_0}\biggl[ \frac{\ud \PP_{\vb_1}^n}{\ud \PP_{0}^n}   \frac{\ud \PP_{\vb_2}^n }{\ud \PP_{0}^n} (\bZ_1, \ldots, \bZ_n) \biggr]\notag\\
&\quad =  \prod_{i=1}^n \biggl(\EE_{\eta_1^i , \eta_2^i} \Bigl\{  \bigl [ 1 - (\sigma^2 + s\beta ^2)^{-1}  \beta^2 \eta_1 ^i\eta_2^i \vb_1^\top \vb_2 \bigr  ]^{-1}\Bigr\}\biggr )\notag\\
&\quad   = \bigl [ 1 - (\sigma^2 + s\beta^2)^{-2} \beta^4  | \vb_1^\top \vb_2| ^2 \bigr ]^{-n} .\$
  This concludes the proof of Lemma \ref{lemma::compute_chi_square2}.                             
\end{proof}

\subsection{Proof of Lemma \ref{lemma:lower_bound_tau}} \label{sec:proof:lemma:lower_bound_tau}
\begin{proof}
	For notational simplicity, within this proof, we define  $ \varphi =  \EE_{\PP_{\vb}} [ q(\bX) / (2M) ]  + 1/2 $  and $\overline \varphi  =    \EE_{\PP_{0}} [ q(\bX) / (2M) ]  + 1/ 2 $. Since $q $ takes values in $[-M, M]$, we have $\varphi, \overline \varphi \in [0, 1]$. Moreover,  we denote $\log (T / \xi)$ by $\kappa$ for simplicity. By \eqref{eq::aux_lemma_difference} and \eqref{eq::query_2}, we have 
	\#\label{eq::aux_diff_mean}
	| \varphi - \overline \varphi | \geq  \max \biggl \{ \frac{\kappa}{2n}, \sqrt{ \frac{ \kappa \cdot \varphi \cdot (1 - \varphi) } {2n} }  \biggr \}.
	\#
	Moreover, our goal is to establish
	\#\label{eq::aux_goal}
	| \varphi - \overline \varphi | \geq   \sqrt{ \frac{ \kappa \cdot  \overline  \varphi \cdot (    1 - \overline \varphi ) } {6n}  } . 
	\#
	Note that both \eqref{eq::aux_diff_mean} and \eqref{eq::aux_goal} remain  the same if we replace $\varphi$ and $\overline \varphi$ by $1-\varphi$ and $1 - \overline  \varphi$, respectively. Thus, it suffices to show \eqref{eq::aux_goal} given that $\varphi \leq 1/2$. 
	Suppose it holds that $\varphi \geq 2 \overline \varphi / 3$, then by \eqref{eq::aux_diff_mean} we have 
	\$
	| \varphi - \overline \varphi |  \geq   \sqrt{ \frac{ \kappa \cdot \varphi \cdot (1 - \varphi) } {2n} } \geq  \sqrt{ \frac{ \kappa \cdot ( 2 \overline \varphi /3 ) \cdot 1/2  } {2n} } = \sqrt{ \frac{\kappa \cdot \overline \varphi}{6n} }.
	\$ 
	Moreover, when $\varphi < 2 \overline \varphi/ 3$, we have 
	$ 
	\overline \varphi   - \varphi \geq \overline \varphi /3 .
	$ 
	Combining this with \eqref{eq::aux_diff_mean}, we obtain that $| \overline \varphi   - \varphi| \geq \sqrt{ \kappa \cdot \overline \varphi / (6n) }$.  Since $\overline \varphi \in [0,1]$,  we conclude the proof of Lemma \ref{lemma:lower_bound_tau}.
\end{proof}

\bibliographystyle{ims}
\bibliography{mixture}
\end{document}